\documentclass[reqno]{amsart}

\usepackage[utf8]{inputenc}
\usepackage{a4wide, hyperref, microtype, url, amssymb, amsmath, pifont, graphicx,mathrsfs}
\usepackage{tikz, tikz-cd}
\usepackage[noabbrev,capitalise]{cleveref} 
\input xy
\xyoption{all}
\usepackage[all,knot]{xy}
\usepackage{comment}

\usepackage{quiver}

\usepackage{amsthm}
\theoremstyle{plain}
\newtheorem{theorem}{Theorem}[section]
\newtheorem{corollary}[theorem]{Corollary}
\newtheorem{proposition}[theorem]{Proposition}
\newtheorem{lemma}[theorem]{Lemma}
\newtheorem{conjecture}[theorem]{Conjecture}

\theoremstyle{definition}
\newtheorem{definition}[theorem]{Definition}
\newtheorem{example}[theorem]{Example}
\newtheorem{remark}[theorem]{Remark}
\newtheorem{assumption}[theorem]{Assumption}
\newcommand{\dil}{L}
\newcommand{\End}{\operatorname{End}}

\numberwithin{equation}{section}
\numberwithin{figure}{section}

\usepackage{todonotes}

\newcommand{\proj}{\operatorname{proj}}
\newcommand{\inj}{\operatorname{inj}}
\renewcommand{\mod}{\operatorname{mod}}
\newcommand{\add}{\operatorname{add}}
\newcommand{\ind}{\operatorname{ind}}
\newcommand{\dimv}{\underline{\dim}}
\newcommand{\Hom}{\operatorname{Hom}}
\newcommand{\rad}{\operatorname{rad}}
\newcommand{\Spec}{\operatorname{Spec}}
\newcommand{\hF}{\widehat F}

\newcommand{\thick}{\operatorname{thick}}
\newcommand{\coht}{\operatorname{coht}}
\newcommand{\bC}{\mathbb{C}}

\newcommand{\bQ}{\mathbb{Q}}
\newcommand{\bZ}{\mathbb{Z}}
\newcommand{\bd}{\mathbf{d}}

\newcommand{\cC}{\mathcal{C}}
\newcommand{\cD}{\mathcal{D}}
\newcommand{\cI}{\mathcal{I}}
\newcommand{\cM}{\mathcal{M}}
\newcommand{\cU}{\mathcal{U}}
\newcommand{\cZ}{\mathcal{Z}}
\newcommand{\wu}{\widehat v}

\newcommand{\tI}{\widetilde{{I}}}
\newcommand{\uu}{\breve{u}}
\newcommand{\uQ}{\breve{Q}}
\newcommand{\up}{\mathbf{p}}
\newcommand{\Gr}{\operatorname{Gr}}
\newcommand{\spl}{\operatorname{sp}}
\newcommand{\trop}{\operatorname{trop}}

\newcommand{\Ext}{\operatorname{Ext}}
\newcommand{\ext}{\operatorname{ext}}
\newcommand{\soc}{\operatorname{soc}}
\newcommand{\wU}{\widetilde V}
\newcommand{\smid}{\,|\,}
\newcommand{\e}{\mathbf{e}}

\newcommand{\Res}{\operatorname{Res}}
\title{Configuration Spaces of Finite Representation Type Algebras}

\author[N. Arkani-Hamed]{Nima Arkani-Hamed}
\address[Nima Arkani-Hamed]
{}
\email{arkani@ias.edu}
\urladdr{}

\author[H. Frost]{Hadleigh Frost}
\address[Hadleigh Frost]
{}
\email{frost@ias.edu}
\urladdr{}

\author[P.-G. Plamondon]{Pierre-Guy Plamondon}
\address[Pierre-Guy Plamondon]
{}
\email{pierre-guy.plamondon@uvsq.fr}
\urladdr{\url{https://plamondon.pages.math.cnrs.fr/plamondon/}}

\author[G. Salvatori]{Giulio Salvatori}
\address[Giulio Salvatori]
{}
\email{salvatori@ias.edu}
\urladdr{}

\author[H. Thomas]{Hugh Thomas}
\address[Hugh Thomas]
{}
\email{thomas.hugh\_r@uqam.ca}
\urladdr{\url{https://hughrthomas.github.io/}}

\begin{document}

\maketitle

\begin{abstract}
To every finite-dimensional $\mathbb C$-algebra $\Lambda$ of finite representation type we associate an affine variety. These varieties are a large generalization of the varieties defined by ``$u$ variables" satisfying ``$u$-equations'', first introduced in the context of open string theory and moduli space of ordered points on the real projective 
line by Koba and Nielsen, rediscovered by Brown as ``dihedral co-ordinates", and recently generalized to any finite type hereditary algebras. We show that each such variety is irreducible and admits a rational parametrization. The assignment is functorial: algebra quotients correspond to monomial maps among the varieties. The non-negative real part of each variety has boundary strata that are controlled by Jasso reduction. These non-negative parts naturally define a generalization of open string integrals in physics, exhibiting factorization and splitting properties that do not come from a worldsheet picture. We further establish a family of Rogers dilogarithm identities extending results of Chapoton beyond the Dynkin case.
\end{abstract}

\section{Introduction}

To every finite-dimensional algebra over $\mathbb C$, $\Lambda$, of finite representation type, we define and study an affine variety $\widetilde{\mathcal M}_\Lambda$. We show that this assignment of varieties to algebras behaves functorially under algebra quotients. We also study the real, non-negative part of these varieties, $\widetilde{\mathcal M}_\Lambda^{\geq 0}$. The boundaries of $\widetilde{\mathcal M}_\Lambda^{\geq 0}$ decompose into strata that are themselves the non-negative parts of $\widetilde{\mathcal{M}}$-varieties for other algebras, obtained as Jasso reductions of $\Lambda$. The $\widetilde{\mathcal M}$-varieties sit at a fascinating meeting point between algebraic geometry, tropical geometry, and representation theory. They are also a broad generalization of the worldsheet moduli spaces of string theory. In string theory, the boundaries of the moduli spaces have a physics interpretation called factorization. The varieties we study go well beyond the string worldsheet, but continue to exhibit factorization. Though the main focus of this paper is to study the mathematical properties of the $\widetilde{\mathcal M}$-varieties we introduce, we also emphasize here how these varieties define ``string-y amplitudes'' that continue to exhibit the properties of string amplitudes but which do not come from a worldsheet picture. 

This paper is a large generalization of earlier work, by physicists and mathematicians, in the setting of Dynkin type algebras. The ``$u$-equations'' associated to Dynkin type $A$ quivers were first introduced by Koba and Nielsen in 1969 \cite{KN1, KN2, KN3}. These equations define an affine variety that we denote $\widetilde{\mathcal{U}}_{A_n}$. For example, in the case of $A_2$, the~$u$-equations are five equations in five variables, and define a 2-dimensional variety:
\begin{center}
 \begin{tabular}{lclcl}
  $u_{13} + u_{24}u_{25} = 1$ & & $u_{24} + u_{35}u_{13} = 1$ && $u_{35} + u_{14}u_{24} = 1$\\
  $u_{14} + u_{25}u_{35} = 1$ && $u_{25} + u_{13}u_{14} = 1$
 \end{tabular}
\end{center}
Here, the variables can be identified with the chords of a pentagon and the monomials appearing in the equations are determined by how the chords cross. The~$u$-equations led Koba and Nielsen to the first definition of tree-level string amplitudes, which are integrals over the non-negative part, $\mathcal U_{A_n}^{\geq 0}$.
Koba and Nielsen showed that the boundary strata of $\mathcal U_{A_n}^{\geq 0}$ intersect in the same way as the faces of the $n$-dimensional associahedron. Each boundary can be specified by the $u_{ij}$ that are zero on that boundary. The $u$-equations ensure that any two $u$-variables can be zero simultaneously only if their corresponding chords are non-crossing. For example, in the $A_2$ example, when $u_{13}=0$, the first $u$-equation imposes that $u_{24}, u_{25} \neq 0$. Correspondingly, in the pentagon, the chords $24$ and $25$ are the chords that cross $13$. This property captures the important physical property of factorization, and played a key role in the early development of string theory.

Arkani-Hamed, He, and Lam \cite{AHL} showed that similar~$u$-equations can be defined for any Dynkin quiver, $Q$. These define affine varieties that we denote as $\widetilde{\mathcal U}_{\mathbb CQ}$. Motivated by cluster algebras, the same authors also introduced a second system of $\hF$-equations, and showed that they define an $n$-dimensional irreducible affine variety, $\widetilde{\mathcal M}_{\mathbb CQ}$. For example, in the case of $A_2$, the~$\hF$-equations are three equations in the same five variables introduced above:
\begin{center}
 \begin{tabular}{lclcl}
  $ u_{13} + u_{24}u_{25}=1$ && $u_{13}u_{14} + u_{14}u_{24}u_{25} + u_{25}u_{35}=1$  && $u_{14}u_{24} + u_{35}=1$
 \end{tabular}
\end{center}
These equations again define a 2-dimensional variety. In fact, \cite{AHL} proved that these two families of varieties are identical, $\widetilde{\mathcal U}_{\mathbb CQ} =\widetilde{\mathcal M}_{\mathbb CQ}$. 
They found a rational parametrisation of $\widetilde{\mathcal M}_{\mathbb CQ}$. And they showed that the boundary strata of the non-negative part, $\widetilde{\mathcal {M}}_{\mathbb CQ}^{\geq 0}$, intersect in the same way as the face lattice of the generalized associahedron of type $Q$.

In this paper, we define affine varieties $\widetilde{\mathcal U}_\Lambda$ and $\widetilde{\mathcal M}_\Lambda$ for every finite-dimensional algebra, $\Lambda$, of finite representation type. Note that this includes non-hereditary algebras. These varieties agree with those defined by \cite{AHL} when we take $\Lambda$ to be the path algebra $\mathbb CQ$ of a Dynkin quiver $Q$. We show that our assignment of varieties to algebras is functorial under algebra quotients, which correspond to toric blowdowns of the $\widetilde{\mathcal M}_\Lambda$ varieties. Moreover, we show that the divisors $u_X = 0$ of $\widetilde{\mathcal M}_\Lambda$, where one variable is set to zero, correspond to so-called \emph{Jasso reductions} of $\Lambda$. See below (Section \ref{intro:summary}) for an overview of our results. Taken together, our results generalize results of Arkani-Hamed, He, and Lam \cite{AHL}, Chapoton \cite{Chapoton}, and Kus--Reineke \cite{KusReineke} for the Dynkin case, and show how geometric properties of the variety $\widetilde{\mathcal M}_\Lambda$ is controlled by the representation theory of $\Lambda$.

\subsection{Summary of main results.}\label{intro:summary}
Fix a finite representation type $\mathbb C$-algebra, $\Lambda$, with $n$ isomorphism classes of simple modules. We can view such a $\Lambda$ as the path algebra of some quiver with relations, $Q$, with $n$ vertices. We consider the homotopy category $K_\Lambda$ of 2-term complexes of projective modules. We write the objects of $K_\Lambda$ as $X = (X^{-1} \rightarrow X^0)$, for $X^{-1}, X^0$ in $\proj\Lambda$. The indecomposable objects $X^{-1}\rightarrow X^0$ of $K_\Lambda$ are of two forms: minimal projective presentations of the module $H^0(X)$, or $P_i\rightarrow 0$, where $P_i$ is an indecomposable projective module. Both $K_\Lambda$ and the module category $\mod\Lambda$ play a key role in our definitions. On the one hand, the indecomposable objects of $K_\Lambda$,~$X \in \ind K_\Lambda$ index the set of~$u$-variables,~$u_X$. On the other hand, we define the \emph{compatibility degree}, $c(X,Y)$, of $X,Y \in \ind K_\Lambda$ in terms of the dimensions of $\Hom$ groups in $\mod \Lambda$ (Lemma \ref{lem:cXY}):
\[
c(X,Y) = \hom_\Lambda (H^0(X),\,H^0(\tau Y)) + \hom_\Lambda(H^0(Y),\, H^0(\tau X)).
\]
Here, $\tau$ is the \emph{Auslander-Reiten translation} on $K_\Lambda$. Given this, the~$u$-equations are
\begin{equation}\label{intro:newu}
u_X + \prod_{Y \in \ind K_\Lambda} u_Y^{c(X,Y)} = 1,
\end{equation}
for each $X \in \ind K_\Lambda$. This system of equations defines an affine variety $\widetilde{\mathcal U}_\Lambda$. We can also define a polynomial in the same variables, $\hF_M \in \mathbb{Z}[u_X\,|\, X \in \ind K_\Lambda]$, for each indecomposable module $M \in \mod\Lambda$ (Definition \ref{def:Fhat}). Then the~$\hF$-equations are,
\[
\hF_M = 1,
\]
for each indecomposable $M \in \mod\Lambda$. These define an affine variety $\widetilde{\mathcal M}_\Lambda$ (Definition \ref{def:varieties}). 

The varieties $\widetilde{\mathcal M}_\Lambda$ are the focus of our paper. These varieties are a natural generalization of the affine varieties introduced by \cite{AHL} for the Dynkin cases. However, it is interesting to comment on how the two families of affine varieties, $\widetilde{\mathcal{M}}_\Lambda$ and $\widetilde{\mathcal{U}}_\Lambda$, are related. We prove (Corollary \ref{coro:F-tilde-imply-u}) that $\widetilde{\mathcal{M}}_\Lambda \subseteq \widetilde{\mathcal{U}}_\Lambda$. Moreover, we prove (Proposition \ref{prop:u-equations-imply-F-hat-sometimes}) that $\widetilde{\mathcal{M}}_\Lambda = \widetilde{\mathcal{U}}_\Lambda$ in the case that the \emph{Auslander-Reiten quiver} of $K_\Lambda$ has no oriented cycles. However, we have also verified in examples (Example \ref{ex:U=M}), that this equality holds more generally, leading us to
\begin{conjecture}\label{U=F}
For any finite representation type algebra $\Lambda$, $\widetilde{\mathcal{U}}_\Lambda = \widetilde{\mathcal M}_\Lambda$.
\end{conjecture}

\medskip

We establish the following properties of the $\widetilde{\mathcal M}_\Lambda$ varieties.

\begin{enumerate}
\setcounter{enumi}{0}

\item  $\widetilde{\mathcal M}_\Lambda$ is an irreducible variety of dimension $n$, and there is an open embedding of $\widetilde{\mathcal M}_\Lambda$ into a toric variety $X_\Lambda$ (Theorem \ref{thm:Mtildetoric}). In particular, the ideal generated by the~$\widehat{F}_M - 1$ is prime.
\end{enumerate}
\noindent
The fan of this toric variety is the~$g$-vector fan of $\Lambda$ (Section \ref{sec:gfan}). The rays of the~$g$-vector fan correspond to indecomposable objects of $K_\Lambda$ which are \emph{rigid} (i.e. satisfy $\Hom(X,\Sigma X)=0$). These rigid modules correspond to interesting subvarieties of $\widetilde{\mathcal M}_\Lambda$:

\begin{enumerate}
\setcounter{enumi}{1}
\item For any rigid indecomposable $X \in \ind K_\Lambda$, the divisor $u_X = 0$ in $\widetilde{\mathcal M}_\Lambda$ is isomorphic to $\widetilde{\mathcal M}_B$, for some algebra $B$ which is called the \emph{Jasso reduction} of $\Lambda$ at $H^0(X)$ (Theorem \ref{Jasso-thm}). A similar statement (Theorem \ref{Jasso-u}) holds for the $\widetilde{\mathcal U}_\Lambda$ varieties.
\end{enumerate}
\noindent
For every $X \in \ind K_\Lambda$, we define (Section \ref{sec:geoM}) a rational function $v_X$ of $n$ variables $y_1,\ldots,y_n$. Then

\begin{enumerate}
\setcounter{enumi}{2}
\item The functions $v_X\in\mathbb C(y_1,\dots,y_n)$ provide a rational parametrization of $\widetilde{\mathcal M}_\Lambda$, i.e. the coordinate ring of $\widetilde{\mathcal M}_\Lambda$ is isomorphic to the subring of $\mathbb C(y_1,\dots,y_n)$ generated by the $v_X$ (Theorem \ref{thm:RViso}). In other words, setting $u_X = v_X(y)$ solves the $\hF_M = 1$ equations and all solutions are obtained in this way (Corollary \ref{vs-solve}).
\end{enumerate}
\noindent
It is natural to consider the real, non-negative points of these varieties, $\widetilde{\mathcal M}_\Lambda^{\geq 0}$. This is closely related to the toric variety $X_\Lambda$ and corresponding~$g$-vector fan.

\begin{enumerate}
\setcounter{enumi}{3}
\item Restricting the variables $y_i$ to be real, non-negative, $y_i\geq 0$, the rational parametrization above restricts to the non-negative part of the variety, $\widetilde{\mathcal M}_\Lambda^{\geq 0}$ (Theorem \ref{thm:agree}). Equivalently, the non-negative part of $\widetilde{\mathcal M}_\Lambda$ agrees with the non-negative part of the toric variety $X_\Lambda$.

\item The boundary strata of $\widetilde{\mathcal M}_\Lambda^{\geq 0}$ are in bijection with the cones of the $g$-vector fan of $\Lambda$ (Proposition \ref{prop:strat}). This follows from the previous point, when combined with the properties of the moment map.
\end{enumerate}
\noindent
The five properties above generalize properties established in \cite{AHL} for the Dynkin case to all finite representation type algebras. Our more general setting allows us to further establish that our assignment of varieties to algebras is functorial:

\begin{enumerate}
\setcounter{enumi}{5}
\item If $A$ is a quotient of $\Lambda$, then there is a surjective map from $\widetilde {\mathcal M}_\Lambda$ to $\widetilde {\mathcal M}_A$ (Theorem \ref{quotient}). As a homomorphism of coordinate rings, this map is a monomial map. This map can also be viewed as the blowdown map, $X_\Lambda \rightarrow X_A$, of the associated toric varieties (Theorem \ref{thm:blowdown}).
\end{enumerate}

\medskip
We also establish the following property of the \emph{Rogers dilogarithm},
\[
L(x) =  -\int_0^x\frac{\log(1-t)}{t}dt + \frac{1}{2} \log(x) \log(1-x),
\]
which generalizes the dilogarithm identities proved by Chapoton \cite{Chapoton} in the Dynkin setting.

\begin{enumerate}
\setcounter{enumi}{6}
\item The sum of the Rogers dilogarithms $L(1-v_X)$ over all $X \in \ind K_\Lambda$ is a constant,~i.e., independent of $y_1,\dots,y_n$ (Theorem \ref{dil-ident}).  
\end{enumerate}

\medskip

Finally, the rational functions $v_X$, which play a key role in establishing the above properties, also have a further interest in their own right. For $X$ a rigid indecomposable in $K_\Lambda$, the \emph{tropicalization} of $v_X$, $\trop v_X$, is a piecewise linear function on $\mathbb{R}^n$. We show:

\begin{enumerate}
\setcounter{enumi}{7}
\item Evaluated on some vector $g\in\mathbb Z^n$, $\trop v_X(g)$ is the multiplicity of $X$ as a summand of a generic 2-term projective complex of weight $g$ (Theorem \ref{trop-v-is-mult}).
\end{enumerate}

\noindent
We prove this for arbitrary finite-dimensional algebras; the assumption that $\Lambda$ is of finite representation type is not used here. This theorem generalizes the recent result of Kus--Reineke for Dynkin quivers \cite{KusReineke}.

\subsection{Generalized factorization and stringy amplitudes}
The type $A$ $u$-equations lead directly to a formula for the tree-level Veneziano amplitudes for open strings, which are meromorphic functions whose residues exhibit the key property of factorization. Other physical properties of these functions, like the spectrum of their poles and their behaviour in special limiting directions, also follow from the geometry of the $\widetilde{\mathcal U}$-varieties. The $\widetilde{\mathcal M}$-varieties that we study also define functions that resemble these open string amplitudes. These ``string-y amplitudes'' are meromorphic functions that continue to exhibit many of the properties that we associate to the string amplitudes, but are a nontrivial generalization of string amplitudes, in the sense that they do not arise from a worldsheet picture. It is not the purpose of this paper to explore the physical applications of these functions. But since they are a main part of our motivation for this work, we give a brief summary here of how these stringy amplitudes are defined and their interest from a physical point of view.

The type $A$ $u$-equations, and their positive parameterization, gives a new way to understand the color-ordered tree-level open string integrals and their properties \cite{AHL,arkani2023,2025cuts} . For variables $u_{ij}$ labelled by the chords $(i,j)$ of the $n$-gon the equations are
\begin{equation}\label{intro:u}
1 = u_{ij} + \prod_{\substack{(k,\ell)\ \text{crosses}\ (i,j)}} u_{k\ell},
\end{equation}
which defines a $n-3$ dimensional variety $\mathcal{U}_{A_n}$. These equations can be viewed as cross-ratio identities. Let $(z_1,\ldots,z_n)$ be some $n$ points on the boundary of a disk with fixed cyclic order. We can take the disk to be the upper half plane, so that the $z_i$ are real. Their cross-ratios satisfy
\[
(ij|kl) + (jk|li) = 1,\qquad (ij|kl) = \frac{(z_i - z_j)(z_k - z_l)}{(z_i - z_k)(z_j - z_l)}.
\]
If we identify the $u$-variables with (positive) cross-ratios,
\[
u_{ij} = (i, j+1 \ | \ i + 1, j),
\]
then the cross-ratio identity becomes \eqref{intro:u}. These cross-ratios are $PSL_2\mathbb{R}$ invariant. So we can identify the non-negative part $\mathcal{U}_{A_n}^{\geq 0}$ with the moduli space of cyclically ordered points up to $PSL_2\mathbb{R}$: i.e. a component $\mathcal{M}_{0,n}^+(\mathbb{R})$ of the real part of the moduli space $\mathcal{M}_{0,n}$. The Veneziano open string amplitude is an integral over $\mathcal{M}_{0,n}^+(\mathbb{R})$, conventionally written as
\begin{equation}
\label{intro:KN}
\mathcal{A}_{n} =
\int
\frac{\prod_{i=1}^{n} dz_i}{PSL(2,\mathbb R)}
\;
\prod_{1\le i<j\le n}
|z_i - z_j|^{\alpha' s_{ij}}
\;\times\; \frac{1}{(z_1 - z_2) (z_2 - z_3) \cdots (z_n - z_1)},
\end{equation}
where the integral is over points $z_i$ in a fixed cyclic ordering $z_1<z_2<\cdots < z_n$. The amplitude is a function of the Mandelstam variables $s_{ij}$ for every pair of points. The integrand is written in terms of variables that are not $PSL_2\mathbb{R}$ invariant. So it is convenient to instead use the positive parameterization of $\mathcal{U}_{A_n}$. The integral then takes the form
\[
\mathcal{A}_n = 
\int_0^\infty \frac{dy_1 \cdots dy_{n-3}}{y_1 \cdots y_{n-3}} \, \prod_{(i,j)} (u_{ij})^{\alpha' X_{ij}},
\]
now as a function of some new variables $X_{ij}$ for every chord $(i,j)$. Here the $u_{ij}$ are parameterized as rational functions of the $y_i$. The poles of $\mathcal{A}_n$ are the singularities arising near where one $u_{i,j} = 0$. Thanks to the properties of the positive parameterization, the residues of the poles of $\mathcal{A}_n$ are themselves integrals of the same type, and manifest the factorization of the amplitude.

For example, consider $\mathcal{A}_5$, corresponding to the path algebra of $A_2$. There are five $u$-variables, $u_{ij}$, and we can take, for example, the following positive parameterization:
\[
u_{25} = \frac{1+y_2}{1+y_2+y_1y_2},~~ u_{35} = \frac{1}{1+y_2},~~u_{13} = \frac{y_1}{1+y_1},~~u_{14} = \frac{y_2(1+y_1)}{1+y_2+y_1y_2},~~u_{24} = \frac{1+y_2+y_1y_2}{(1+y_2)(1+y_1)}
\]
The poles in $X_{13}$ arise from near $y_1 = 0$, where $u_{13} \approx 0$. Near $y_1 = 0$, we have
\[
u_{13} = y_1 + O(y_1^2),~~ u_{14} = \frac{y_2}{1+y_2} + O(y_1), ~~ u_{35} = \frac{1}{1+y_2},~~ u_{24} = 1 + O(y_1),~~ u_{25} = 1 + O(y_1).
\]
The integral is
\[
\mathcal{A}_5 = \int \frac{dy_2}{y_2} \int \frac{dy_1}{y_1} \left(y_1 + O(y_1^2)\right)^{\alpha' X_{13}} \times (u_{14})^{\alpha' X_{14}}(u_{24})^{\alpha' X_{24}}(u_{25})^{\alpha' X_{25}}(u_{35})^{\alpha' X_{35}}.
\]
Taylor expanding the integrand in $y_1$, we can pull out the leading term, $y_1^{\alpha' X_{13} - 1}$, which gives a pole at $X_{13} = 0$:
\[
\mathcal{A}_5 \sim \frac{1}{\alpha' X_{13}} \int \frac{dy_2}{y_2} \left( \frac{y_2}{1+y_2} \right)^{\alpha' X_{14}} \left( \frac{1}{1+y_2} \right)^{\alpha' X_{35}}.
\]
The residue of the pole is an integral corresponding to the $A_1$ path algebra, or $n=4$ amplitude, involving the two chords, $(1,4)$ and $(3,5)$, that do not cross $(1,3)$. This is an example of \emph{factorization}, and corresponds to the boundary of the moduli space where the disk with $5$ marked points factorizes into a $3$-point disk times a $4$-point disk. Note that $\mathcal{A}_5$ also has poles at $X_{13} + N = 0$ (for positive integers $N$), and these are obtained from the non-leading terms in the Taylor expansion of the integrand in $y_1$.

We also have the monomial maps among $u$-variables that imply the behaviour of $\mathcal{A}_n$ in certain limits. For example, note that the $A_2$ $u$-variables above satisfy
\[
\frac{1}{1+y_1} = u_{24} u_{25},\qquad \frac{y_2}{1+y_2} = u_{14} u_{24}.
\]
So consider $\mathcal{A}_5$ in the limit that
\[
X_{24} \rightarrow X_{14} + X_{25}.
\]
In this limit,
\[
\mathcal{A}_5 \rightarrow \int \frac{dy_1}{y_1} \left( \frac{1}{1+y_1} \right)^{X_{25}} \left( \frac{y_1}{1+y_1} \right)^{X_{13}} \times  \int \frac{dy_2}{y_2} \left( \frac{1}{1+y_2} \right)^{X_{35}} \left( \frac{y_2}{1+y_2} \right)^{X_{14}},
\]
i.e. $\mathcal{A}_5$ becomes a product of two $n=4$ amplitudes in this limit. This is called \emph{splitting}.

The results of this paper suggest a large generalization of the string integrals associated to type $A$. For any finite representation type algebra $\Lambda$, we can define stringy integrals
\begin{equation}
\mathcal{A}_\Lambda = \int \frac{dy_1 \cdots dy_{n}}{y_1 \cdots y_{n}} \, \prod_{Y \in \ind K_\Lambda} (u_{Y})^{\alpha' X_Y},
\end{equation}
where the $u_Y$ satisfy the $u$-equations, \eqref{intro:newu}, and are given as positive rational functions of the positive $y_i$ variables. $\mathcal{A}_\Lambda$ can be viewed as an integral over the non-negative part of the variety $\widetilde{\mathcal{U}}_{\Lambda}^{\geq 0}$. Each $\mathcal{A}_\Lambda$ meromorphic functions of the $X_Y$ variables, and our results imply that they continue to exhibit the properties of the standard tree-level amplitude above. In particular, our results imply that the residue of $\mathcal{A}_\Lambda$ at some poles $X_Y = 0$ is
\[
\Res_{X_Y = 0} \mathcal{A}_\Lambda = \mathcal{A}_{\mathcal{B}}
\]
where $\mathcal{B}$ is the \emph{Jasso reduction} of $\Lambda$ at $Y$. Jasso reduction is then a notion of \emph{generalized factorization} of these stringy integrals that has nothing to do with a worldsheet picture. Moreover, we find monomial maps $\mathcal{M}_\Lambda \rightarrow \mathcal{M}_A$ for any quotient algebra $A$ of $\Lambda$. This implies that, taking certain limits of the variables $X_Y$, the stringy integrals behaves as
\[
\mathcal{A}_\Lambda \rightarrow \mathcal{A}_A.
\]
In other words, the stringy integrals have ``generalized splits''. In particular, we can take the quotient by the ideal generated by \emph{all} edges of the quiver, so that $A = A_1 \times A_1 \times \cdots \times A_1$, for $n$ copies of $A_1$. Then there is a limit such that $\mathcal{A}_\Lambda$ becomes a product of $n$ integrals $\mathcal{A}_{A_1}$. 

For example, consider the quiver $Q$ given by
\[
\begin{tikzpicture}[>=stealth]
 \node (1) at (0,0) {$1$};
  \node (2) at (-1.3,0) {$2$};
  \draw[->,thick] (1) -- (2) node[midway,above]{$a$};
  \draw[->,thick]
    (1) .. controls +(0.9,0.6) and +(0.9,-0.6) .. (1)
    node[midway,right] {$x$};
\end{tikzpicture}
\]
We can take the algebra $\Lambda = \mathbb{C} Q / \langle x^2 \rangle$, where $\mathbb{C} Q$ is the path algebra, and we quotient by the relation $x^2 = 0$. This is a finite representation type algebra. The category $K_\Lambda$ has $9$ indecomposable objects, which label $9$ $u$-variables. (See Section \ref{A2-loop}, where this example is worked out in more detail). The stringy integral $\mathcal{A}_\Lambda$ is then a meromorphic function of $9$ variables $X_Y$ (for $Y \in \ind K_\Lambda$), whose residues are also given by stringy integrals. Its poles exhibit the generalized factorization property. For example, take the projective module $P_2$ with support at vertex $2$. Then
\[
\Res_{X_{P_2} = 0} \mathcal{A}_\Lambda = \mathcal{A}_\mathcal{B},
\]
where $\mathcal{B}$ is the Jasso reduction of $\Lambda$ at $P_2$. $\mathcal{B}$ is the path algebra of the quiver
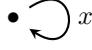
\[
\begin{tikzpicture}[>=stealth]
 \node (v) at (0,0) {$\bullet$};
  \draw[->,thick]
    (v) .. controls +(0.9,0.6) and +(0.9,-0.6) .. (v)
    node[midway,right] {$x$};
\end{tikzpicture}
\]
modulo the relation $x^2 = 0$. $K_\mathcal{B}$ has three indecomposable objects, so that $\mathcal{A}_\mathcal{B}$ is a function of three $X$ variables. It is given explicitly by
\[
\mathcal{A}_\mathcal{B} = \int \frac{dy}{y} \left(\frac{y+1}{y^2+y+1}\right)^{X_1} \left(\frac{y (y+1)}{y^2+y+1}\right)^{X_2} \left(\frac{y^2+y+1}{(y+1)^2}\right)^{X_3}
\]
This integral gives an example of generalized splitting. In the limit $X_3 \rightarrow X_1 + X_2$ we have
\[
\mathcal{A}_\mathcal{B} \rightarrow \int \frac{dy}{y}  \left( \frac{1}{1+y} \right)^{X_{1}} \left( \frac{y}{1+y} \right)^{X_{2}} = \mathcal{A}_{A_1},
\]
which corresponds to the quotient $\mathcal{B} / \langle a \rangle = \mathbb{C} A_1$.

\subsection{Introduction from a cluster algebras point of view}
The results of \cite{AHL} are couched in the language of cluster algebras. We summarize their results here, in order to clarify how our work is related to theirs, and to explain how this paper generalizes their results. A reader for whom cluster algebras are unfamiliar (or unappealing) can safely skip this discussion. We do not otherwise make use of cluster algebras in this paper.

Given a Dynkin type $\Phi$ (not necessarily simply-laced) of rank $n$, with orientation, there is an associated cluster algebra with \emph{universal coefficients}, $\mathcal{A}$. The cluster variables of $\mathcal{A}$, $x_\gamma$, are indexed by the \emph{almost positive roots} $\gamma \in \Phi_{\geq -1}$. Moreover, there is a pairing on the almost positive roots, $(\gamma\parallel\beta)\in\mathbb Z_{\geq 0}$, for $\gamma,\beta\in\Phi_{\geq -1}$ (this is the \emph{Fomin--Zelevinsky compatibility degree}).
and choose an orientation of the edges of the Dynkin diagram. These data determine a skew-symmetrizable $n\times n$ exchange matrix $B$. Further rows can be added to $B$ if desired, obtaining a larger matrix $\widetilde B$ to which we can associate a cluster algebra $\mathcal A(\widetilde B)$. One way to add further rows provides the cluster algebra of type $\Phi$ with so-called \emph{universal coefficients}; we skip the precise details of this here. The cluster variables of $\mathcal A(\widetilde B)$ are indexed by the set of \emph{almost positive roots}, $\Phi_{\geq -1}$. For any two such roots, $\gamma,\beta\in\Phi_{\geq -1}$, there is the \emph{Fomin--Zelevinsky compatibility degree}, denoted $(\gamma\parallel\beta)\in\mathbb Z_{\geq 0}$.

The affine varieties introduced in \cite{AHL} can then be described as follows. Introduce a variable $u_\gamma$ for each $\gamma\in\Phi_{\geq -1}$. Then we can consider the variety cut out by the system of equations,
$$ u_\gamma+ \prod_{\beta\in\Phi_{\geq -1}} u_\beta^{(\beta\parallel\gamma)} = 1,$$
where we have one equation for each $\gamma\in\Phi_{\geq -1}$. In particular, in type $A_n$, the almost positive roots can be identified with the diagonals of an $(n+3)$-gon, and the compatibility degree $(\beta\parallel\gamma)$ is 1 if $\beta$ and $\gamma$ cross in their interiors, and zero otherwise. This recovers the $u$-equations of Koba--Nielsen \cite{KN1, KN2, KN3}.

A second family of affine varieties, also introduced in \cite{AHL}, can be defined using the same data. These are cut out by equations we call the $\hF$-equations. In the Dynkin setting, the $\hF$- and $u$-equations turn out to define the same varieties. However, the equations look very different, and are defined as follows. For each $\beta\in \Phi_{\geq -1}$, let $x_\beta$ be the corresponding variable in the cluster algebra with universal coefficients, where the coefficient variables, also indexed by $\Phi_{\geq -1}$, are denoted $u_\beta$. For a negative simple root, $-\alpha_i$, the corresponding $x_{-\alpha_i}$ is an initial variable.
The polynomial $\hF_\beta$ is then given by
$$\hF_\beta= \left. x_\beta\, \right|_{x_{-\alpha_i} = 1}.$$
That is to say, $\hF_\beta$ is obtained from $x_\beta$ by setting all initial cluster variables equal to 1, leaving a polynomial in the coefficient variables $\{u_\gamma\}_{\gamma\in\Phi_{\geq -1}}$. (The procedure to produce $\hF_\beta$ is analogous to the procedure for defining the $F$-polynomials given in \cite{CA4}; there, however, one starts with the cluster algebra with principal coefficients.) 
The system of $\hF$-equations are simply
\[
\hF_\beta=1,\]
for each $\beta\in\Phi_{>0}$.

The positive real points of these varieties have interesting boundary stratifications. In fact, the divisors where a single variable vanishes, $u_\beta = 0$, are themselves affine varieties of the same kind.

Moreover, \cite{AHL} show in the Dynkin setting how to give a rational parametrization of these varieties. 
Suppose that $x_\gamma$ and $x_\delta$ are two cluster variables related by mutation. Then we have an ``exchange relation'', $x_\gamma x_\delta = M_1 + M_2$, where $M_1$ and $M_2$ are two monomials in the remaining variables in the cluster and also in the coefficient variables.\footnote{The ratio $M_2/M_1$ can be identified as an $\mathcal X$-cluster variable, in the sense of Fock and Goncharov.} To each $x_\gamma$ (with $\gamma\in\Phi_{\geq 0}$), \cite{AHL} define a specific cluster containing $x_\gamma$. Mutating $x_\gamma$ in this cluster gives some exchange relation, $x_\gamma x_\delta = M_1 + M_2$, and \cite{AHL} show that setting
\[
u_\gamma=\frac{M_1}{x_\gamma x_\delta}=\frac{1}{1+M_2/M_1}
\]
solves the $u$-equations (equivalently, the $\hF$-equations). In other words, these rational functions of the cluster variables parametrize the varieties $\widetilde{\mathcal M}_{\mathbb C Q}$.

The present paper greatly generalizes these results, as summarized above in Section \ref{intro:summary}. We define $u$-equations and $\hF$ equations for any algebra of finite representation type, $\Lambda$. Note that most such algebras do not correspond to a cluster algebra. However, there are suitable generalizations of all the key ingredients in the cluster algebra definitions. In place of cluster compatibility degree, we introduce a compatibility degree expressed using the dimensions of the $\Hom$ groups in $\mod \Lambda$. In place of cluster variables, we use the $F$-polynomials associated to $\Lambda$, as introduced by Derksen--Weyman--Zelevinsky. Finally, we define a suitable generalization of ``$F$-polynomials with universal coefficients'' to our setting, which we continue to denote as $\hF$.

\subsection{Motivation from number theory}
In number theory,~$u$-equations have also played a prominent role. Brown showed that~$u$-equations define affine charts of~$\overline{\mathcal{M}}_{0,n}$ and used these to prove that all periods of~$\mathcal{M}_{0,n}$ can be expressed as multiple zeta values (MZVs) \cite{Brown09}. A key feature of this proof is that~$u$-equations behave well under the natural forgetful morphisms~$\mathcal{M}_{0,n+1} \to \mathcal{M}_{0,n}$. The resulting iterated integrals bear a striking resemblance to string amplitude integrals and provide a geometric framework for understanding the appearance of MZVs in string theory amplitudes \cite{brown2021, broedel2014}. In addition, Brown's work suggests that the geometry of~$u$-equation varieties may be useful for studying relations among MZVs—such as Zagier’s height-two identity \cite{zagier2012}.

There is also a fascinating connection to special functions. The Rogers dilogarithm $L(x)$ satisfies an identity that may be viewed as a~$u$-equation for the $A_2$ root system \cite{kirillov1995}. This identity appears in special cases of Nahm's conjecture  about the modular properties of hypergeometric $q$-series. Generalisations of Rogers dilogarithm identities have been established for algebras of Dynkin type in \cite{Chapoton}. In this paper, we use our approach to extend these results. We present a natural family of Rogers dilogarithm identities, one for every finite representation type algebra (Section \ref{sec:dilog}).

\section{Algebra background}
Let~$\Lambda$ be an associative~$\mathbb C$-algebra, finite-dimensional over~$\mathbb C$.  Let~$e_1, \ldots, e_n$ be a complete set of pairwise orthogonal primitive idempotents of~$\Lambda$.  Then, as a right~$\Lambda$-module,~$\Lambda$ is a direct sum of the indecomposable projective modules:~$P_1 = e_1\Lambda, \ldots, P_n = e_n\Lambda$.  We assume that~$\Lambda$ is \emph{basic}; i.e. the~$P_i$ are pairwise non-isomorphic. We write~$\mod\Lambda$ for the abelian category of finite-dimensional right~$\Lambda$-modules, and~$\proj\Lambda$ and~$\inj\Lambda$ for the full additive subcategories of projective and injective modules, respectively. We write $S_i=P_i/\rad P_i$ for the \emph{simple} right $\Lambda$ modules. The indecomposable injective modules, $I_i$, are the injective envelopes of $S_i$. In particular, note that the socle of $I_i$ is $\soc I_i = S_i$.

We write $\hom_\Lambda(~,~)$ for the dimension of $\Hom_\Lambda(~,~)$ over $\mathbb C$, and similarly $\hom_{\mathcal C}(~,~)$ for the dimension of $\Hom_{\mathcal C}(~,~)$ in a category $\mathcal C$. Similarly, we write $\ext_\Lambda^1(~,~)$ for the dimension of $\Ext_\Lambda^1(~,~)$.
We write $DV$ for the $\mathbb C$-linear dual vector
space of a $\mathbb C$-vector space $V$.

\subsection{Almost split morphisms}
We briefly recall some key elements of Auslander--Reiten theory. The reader interested in further details is invited to consult \cite[Chapter IV]{ASS}.

Consider a short exact sequence,~$0\to L\xrightarrow{f} M\xrightarrow{g} N\to 0$, in $\mod \Lambda$. It is \emph{split} if it is isomorphic to~$0 \to L \to L\oplus N \to N \to 0$. In a split short exact sequence, $f$ is called a \emph{section} and $g$ is called a \emph{retraction}.
Obviously, in a split exact sequence, any map $h:Z\rightarrow N$ can be lifted to a map $h':Z\rightarrow M$ such that $h=gh'$. Under the same hypothesis, any map $j:L\rightarrow Z$ can be extended to a map $j':M\rightarrow Z$ with $j=j'f$.

An important class of short exact sequences are the \emph{almost split} or \emph{Auslander--Reiten} sequences. They are very close to sharing the above lifting/extension properties. 
For $N$ indecomposable, a map $g:M\rightarrow N$ is \emph{right almost split} if it is not a retraction, and for any indecomposable module $Z$ and non-isomorphism $h:Z\rightarrow N$, there exists a map $h':Z\rightarrow M$ such that $h=gh'$. Similarly, for $L$ indecomposable, a map $f:L\rightarrow M$ is \emph{left almost split} if it is not a section, and for any indecomposable module $Z$ and non-isomorphism $j:Z\rightarrow L$, there exists a map $j':M\rightarrow Z$ with $j=j'f$. 

A short exact sequence $0\to L\xrightarrow{f} M\xrightarrow{g} N\to 0$ is \emph {almost split}  if and only if $N$ is indecomposable and $f$ is left almost split, or equivalently if and only if $L$ is indecomposable and $g$ is right almost split. \cite[Theorem IV.1.13]{ASS}.
Any non-projective indecomposable module~$N$ is the right-most term of an almost split sequence,~$0 \to \tau N \to M  \to N \to 0$, unique up to isomorphism \cite[Theorem IV.3.1]{ASS}. Similarly, any non-injective indecomposable~$L$ is the left-most term of an almost split sequence,~$0 \to L \to M \to \tau^{-1} L \to 0$. This defines the \emph{Auslander-Reiten translation},~$\tau$, of $\mod\Lambda$. 

Note that the almost split sequences provide a supply of left and right almost split maps (a right almost split map to every non-projective indecomposable, and a left almost split map from every non-injective indecomposable). We will also make use of right almost split maps into indecomposable projective modules and left almost split maps out of indecomposable injective modules.
If $P$ is indecomposable projective, the inclusion
$\rad P \rightarrow P$ is right almost split. Moreover, if $I$ is indecomposable injective, then the quotient $I\rightarrow I/\soc I$ is left almost split \cite[Proposition IV.3.5]{ASS}.

The following proposition follows directly from the almost splitness of the maps mentioned above. We use this, in particular, in Section \ref{sec:split}, where we show how these identities define a useful set of dual bases of the split Grothendieck group.

\begin{proposition} \label{prop:ar} Suppose that $0\to \tau N \to M \to N \to 0$ is an almost split sequence, and $X$ is an indecomposable module. Then:
  \begin{enumerate}
  \item $\hom_\Lambda(X,\tau N)-\hom_\Lambda(X,M) + \hom_\Lambda(X,N)=0$ unless $X\simeq N$, in which case it equals $1$.
  \item $- \hom_\Lambda(X,\rad P_i) + \hom_\Lambda(X,P_i) = 0$ unless $X \simeq P_i$, in which case it equals $1$.
    \item $\hom_\Lambda(N,X)-\hom_\Lambda(M,X) + \hom_\Lambda(\tau N,X)=0$ unless $X\simeq \tau N$, in which case it equals $1$.
    \item $\hom_\Lambda(I_i,M) - \hom_\Lambda(I_i/S_i, M) = 0$, unless $M \simeq I_i$, in which case it is $1$.
    \end{enumerate} 
\end{proposition}

  The \emph{injectively stable morphisms} from $X$ to $Y$ in $\mod \Lambda$,
  denoted $\overline{\Hom}_\Lambda(X,Y)$ are by definition $\Hom_\Lambda(X,Y)$ quotiented by the subspace of those morphisms factoring through an injective $\Lambda$-module. One of the Auslander--Reiten formulas \cite[Theorem IV.2.13]{ASS} says the following:

  \begin{theorem}\label{thm:ar-formula} $\Ext^1_\Lambda(X,Y)\simeq D\overline{\Hom}_\Lambda(Y,\tau X)$. \end{theorem}

  We will need the following lemma:

  \begin{lemma} \label{lem:exthom} Let $M$ be an indecomposable $\Lambda$-module.
  Then  $\Ext^1_\Lambda(M,S_i) \simeq D\Hom_\Lambda(S_i,\tau M).$ \end{lemma}

  \begin{proof} By the Auslander--Reiten formula of Theorem \ref{thm:ar-formula},
    $\Ext^1_\Lambda(M,S_i)\simeq D\overline{\Hom}_\Lambda(S_i,\tau M)$. The only indecomposable injective which admits a non-zero morphism from $S_i$ is $I_i$. An injection from $I_i$ to $\tau M$ splits, so in that case $\tau M$, being indecomposable, would have to be isomorphic to $I_i$, which is impossible. Thus, a map from $S_i$ to $\tau M$ would necessarily factor through a proper quotient of $I_i$, and thus a quotient of $I_i/\soc I_i$. The only possible non-zero image of $S_i$ in $I_i$ being $\soc I_i$, no non-zero map from $S_i$ to $\tau M$ factors through $I_i$, and therefore $D\overline{\Hom}_\Lambda(S_i,\tau M)\simeq D\Hom_\Lambda(S_i,\tau M)$.
    \end{proof}

\subsection{The category $K_\Lambda$}\label{2-term}
Let~$K^b := K^b(\proj\Lambda)$ be the homotopy category of bounded complexes of projective modules. Write $\Sigma$ for the shift functor of~$K^b$, that shifts all degrees of an object in~$K^b$ by~$-1$.

\begin{definition}
The category,~$K_\Lambda := K^{[-1,0]}(\proj\Lambda)$, of~\emph{$2$-term complexes of projectives}, is the full subcategory of $K^b$ whose objects are complexes supported only in degrees~$-1$ and~$0$. $K_\Lambda$ is an extension-closed subcategory of~$K^{b}(\proj\Lambda)$, and so is an extriangulated category in the sense of~\cite{NakaokaPalu19}, by a result of~\cite{HerschendLiuNakaoka21}. As an extriangulated category,~$K_\Lambda$ has enough projective objects and enough injective objects.  
\end{definition}

For an object, $X$, in $K_\Lambda$, we write $X = (X^{-1} \rightarrow X^0)$, for projective modules $X^{-1}$, $X^0$ in $\mod \Lambda$. The indecomposable projective objects of $K_\Lambda$ are $P_i = (0 \rightarrow P_i)$. (Note that we write $P_i$ both for the projective indecomposables of $\mod\Lambda$ and of $K_\Lambda$.) The indecomposable injective objects are the shifted projectives, $\Sigma P_i = (P_i \rightarrow 0)$.

\begin{definition}
A \emph{conflation} in $K_\Lambda$ is a distinguished triangle of $K^b(\proj\Lambda)$ whose objects and morphisms are in $K_\Lambda$. Given such a distinguished triangle $X \rightarrow E \rightarrow Y \rightarrow \Sigma X$, we write $X \rightarrowtail E \twoheadrightarrow Y$ for the corresponding conflation. A conflation in $K_\Lambda$ is split (resp. almost split) if it is split (resp. almost split) as a distinguished triangle in $K^b(\proj\Lambda)$.
\end{definition}

We write $\ind K_\Lambda$ for the set of indecomposable objects of $K_\Lambda$. The category~$K_\Lambda$ has \emph{Auslander--Reiten--Serre duality} \cite{IyamaNakaokaPalu}. In particular, for any non-projective indecomposable object~$X \in \ind K_\Lambda$, there is an almost split conflation of the form~$\tau X \rightarrowtail E \twoheadrightarrow X$. For any non-injective indecomposable object~$Y \in \ind K_\Lambda$, there is an almost split conflation~$Y \rightarrowtail F \twoheadrightarrow \tau^{-1} Y$. (Note that we write $\tau$ both for the Auslander-Reiten translation in $\mod\Lambda$ and in $K_\Lambda$.)

\subsection{The $H^0$ functor}
The category $K_\Lambda$ is closely related to $\mod\Lambda$. Indeed, there is a functor
\[
H^0:\, K_\Lambda \rightarrow \mod \Lambda,
\]
which sends $X = (X^{-1} \xrightarrow{f} X^0)$ to its cohomology in degree zero, $H^0(X) = X^0 / {\rm Im}(f)$. In other words, we can identify objects $\mod\Lambda$ with their (minimal) projective presentations, which are objects of $K_\Lambda$. 
However, note that the injectives of $K_\Lambda$, $\Sigma P_i$, are not projective presentations of a non-zero module in $\mod\Lambda$, and $H^0(\Sigma P_i) = 0$. This leads to:

\begin{lemma}
$H^0$ induces an equivalence of~$k$-linear categories~\[K_\Lambda/\langle\Sigma \Lambda\rangle \xrightarrow{\sim} \mod\Lambda.\] Here,~$\langle\Sigma\Lambda\rangle$ is the ideal of all morphisms   that factor through an object of~$\add(\Sigma\Lambda)$.
\end{lemma}

\begin{lemma}\label{lem:split}
A conflation,~$X \rightarrowtail E \twoheadrightarrow Y$, in $K_\Lambda$ is split if and only if the corresponding exact sequence
\[
H^0 X \rightarrow H^0 E \rightarrow H^0Y
\]
is a split short exact sequence in $\mod\Lambda$.
\end{lemma}

Moreover, $H^0$ relates the Auslander--Reiten translations of $K_\Lambda$ and $\mod \Lambda$.

\begin{lemma}\label{lem:almost-split}
For a non-projective indecomposable, $X \in \ind K_\Lambda$, consider its Auslander-Reiten conflation in $K_\Lambda$, $\tau X \rightarrowtail E \twoheadrightarrow X$. Then
\begin{enumerate}
\item If $X$ is non-injective, then this conflation is sent by $H^0$ to the Auslander-Reiten sequence in $\mod\Lambda$,
\[
0\rightarrow H^0(\tau X) \rightarrow H^0E \rightarrow H^0X \rightarrow 0.
\]
In particular, $H^0(\tau X)  = \tau H^0 X$.
\item If $X$ is injective in $K_\Lambda$, i.e. $X = \Sigma P_i$, then the above conflation is sent by $H^0$ to the exact sequence
\[
I_i \rightarrow I_i / S_i \rightarrow 0,
\]
where $I_i$ is the injective module in $\mod\Lambda$. In particular, $H^0(\tau \Sigma P_i) = I_i$.
\end{enumerate}
\end{lemma}

The following form of Auslander-Reiten duality relates morphisms in the module category and the homotopy category.

\begin{lemma}\label{lemm::hom}
For ~$X,Y \in \ind K_\Lambda$,
\[
  \Hom_{K^b}(X, \Sigma Y) \cong  D\Hom_{\Lambda}(H^0Y, H^0 (\tau X)),
\]
In particular,
\[
 \hom_{K^b}(X, \Sigma Y)  = \hom_\Lambda(H^0Y, \, H^0 (\tau X) ).
\]

\end{lemma}
\begin{proof}
For~$X \neq \Sigma P_i$, we have $H^0(\tau X) = \tau H^0 X$ (Lemma \ref{lem:almost-split}) and the result can be proved using Serre duality and the Nakayama functor~$\nu$; a proof can be found in \cite[Lemma 2.6]{Plamondon13}.  For~$X = \Sigma P_i$, we have $H^0(\tau X) = I_i$. Indeed, in this case we have $\Hom_{K^b} (X,\Sigma Y) \cong \Hom_\Lambda(P_i, H^0 Y)$. And $\Hom_\Lambda(P_i,H^0Y) \simeq D \Hom_\Lambda(H^0Y,I_i)$.
\end{proof}

\subsection{The $g$-vector fan} \label{sec:gfan} 
A general reference for the material in this subsection is \cite{AIR}.
For each object $X$ in $K_\Lambda$, we define a $g$-vector, $g(X)$, in $\mathbb Z^n$. If $X = (X^{-1} \rightarrow X^0)$, then the entry $g(X)_i$ of $g(X)$ is the multiplicity of $P_i$ in $X^0$ minus the multiplicity of $P_i$ in $X^{-1}$.

An object $X$ in $K_\Lambda$ is called \emph{rigid} if $\Hom_{K^b(\proj \Lambda)}(X,\Sigma X)=0$. 
If $X$ is rigid, then $P_i$ cannot appear in both $X^{-1}$ and $X^0$ \cite[Proposition 2.5]{AIR}. So, in particular, if $X$ is rigid and non-zero then $g(X) \neq 0$. There is at most one rigid object with any given $g$-vector \cite[Theorem 5.5]{AIR}. We say that two rigid objects $X,Y$ in $K_\Lambda$ are \emph{compatible} if $\Hom_{K^b}(X,\Sigma Y)=0$ and $\Hom_{K^b}(Y,\Sigma X)=0$.

The $g$-vector fan in $\mathbb{Z}^n$ has rays generated by the $g$-vectors of indecomposable rigid objects. The cones of the $g$-vector fan are generated by the $g$-vectors of sets of pairwise compatible indecomposable rigid objects. As the name suggests, the $g$-vector fan is indeed a fan \cite[Corollary 6.7]{DIJ}. In other words, two cones in the fan intersect either in another cone of the fan. Moreover, if $\Lambda$ has only finitely many indecomposable rigid objects, then the $g$-vector fan covers $\mathbb R^n$ \cite[Proposition 4.8]{Asai}. 

\subsection{Dimension vectors}
For each object $M$ in $\mod\Lambda$, define the \emph{dimension vector}, $d(M) \in \mathbb{Z}^n$, with $d(M)_i = \dim (M e_i)$. For $X \in K_\Lambda$, we define its dimension vector as the dimension vector of $H^0(X)$, and write $d(X) = d(H^0(X))$. Write $\langle-,-\rangle$ for the standard inner product on $\mathbb{Z}^n$. Then, for $X \in K_\Lambda$ and $M\in \mod \Lambda$, we have the pairing
\[
\langle g(X), \, d(M) \rangle = \sum_{i=1}^n g(X)_i \, d(M)_i.
\]
If $X = (X^{-1}\rightarrow X^0)$, then 
\begin{equation} \label{eq:gdpairing}
\langle g(X), \, d(M) \rangle = \hom_{\Lambda}(X^0, M) - \hom_{\Lambda}(X^{-1}, M).
\end{equation}
The following is a variation on \cite[Theorem~1.4a]{AR} (See also \cite[Proposition 2.4(a)]{AIR}.) We include a short proof for the convenience of the reader.

\begin{lemma}\label{lem:gd}
For an indecomposable $X \in \ind K_\Lambda$, and $M \in \mod \Lambda$,
\[
\langle g(X), d(M) \rangle   =  \hom_{\Lambda}(H^0X,M)  -  \hom_\Lambda(M,\, H^0 (\tau X)).
\] 
In the special case that $X = \Sigma P_i$, this identity becomes
\[
\langle g(\Sigma P_i), d(M) \rangle   =  -  \hom_\Lambda(M,\, I_i),
\] 
where $I_i$ is the injective module in $\mod \Lambda$.
\end{lemma}
\begin{proof}
If $X = \Sigma P_i$, $H^0X = 0$ and $H^0(\tau X) = I_i$. But $\hom_\Lambda( P_i, M) = \hom_\Lambda (M, I_i)$. So the statement follows directly from \eqref{eq:gdpairing}, after using that $\hom_\Lambda(P_i,M) = \hom_\Lambda(M,I_i)$.

Now take $X$ non-injective. We have $X^{-1} \xrightarrow{f} X^0 \xrightarrow{q} H^0 X \rightarrow 0$. Dually, there is a short exact sequence $0 \xrightarrow{\iota} \tau H^0 X \rightarrow \nu X^{-1} \xrightarrow{\nu f} \nu X^0$, where $\nu$ denotes the Nakayama functor. Then, applying $\Hom_\Lambda(M,-)$ and $\Hom_\Lambda(-,M)$, respectively, we get the following commutative diagram
    $$
    \begin{tikzcd}[column sep = 1em]
        0 \arrow[r] & \Hom(M, \tau H^0X) \arrow[r,"\iota_*"] & \Hom(M, \nu X^{-1}) \arrow[r,"(\nu f)_*"] \arrow[d,leftrightarrow,"\simeq"] & \Hom(M,\nu X^0 \arrow[d,leftrightarrow,"\simeq"])\\
        && D\Hom(X^{-1},M) \arrow[r,"Df^*"] & D\Hom(X^0,M) \arrow[r,"Dq^*"] & D\Hom(H^0X,M) \arrow[r] & 0,
    \end{tikzcd}
    $$
   which defines a four-term exact sequence. Exactness of this implies
    $$\hom_\Lambda (M,\tau H^0X) - \hom_\Lambda(X^{-1},M) + \hom_\Lambda(X^0,M) - \hom_\Lambda (H^0X,M) = 0.$$
    Then \eqref{eq:gdpairing} gives the result.
 \end{proof}

\subsection{$F$-polynomials}
Given a vector~$\bd\in \bZ_{\geq 0}$, let~$\Gr_{\bd}(M)$ be the \emph{submodule Grassmannian} of~$M$. As a set, the elements of $\Gr_{\bd}(M)$ are in bijection with submodules of~$M$ of dimension~$\bd$. It can be viewed as a Zariski-closed subset of the product~$\prod_{i=1}^n \Gr_{d_i}(Me_i)$ of Grassmannians \cite{CalderoChapoton06}, and in this setting it is a projective variety. 

\begin{definition}\label{def:F}
Let~$M$ be a module in~$\mod\Lambda$.  Its \emph{$F$-polynomial} is
 \[
  F_M = \sum_{\bd\in \bZ_{\geq 0}} \chi\left( \Gr_{\bd}(M) \right) y^{\bd} \in \bZ[y_1, \ldots, y_n],
 \]
 where~$\chi$ is the Euler characteristic of topological spaces and~$y^{\bd} = \prod_{i=1}^n y_i^{d_i}$. Moreover, for an object~$X$ of~$K_\Lambda$, we define its $F$-polynomial as
 \[
 F_X := F_{H^0X}.
 \]

\end{definition}

\subsection{$\hF$-polynomials}

We now define a second set of polynomials, $\hF_M$, in the ring~$\bZ[u_X \, | \, X \in \ind K_\Lambda]$. For this, it is important to assume the following

\begin{assumption}[Finiteness assumption]\label{finiteness-assumption}
 The set~$\ind K_\Lambda$ of isomorphism classes of indecomposable objects of~$K_\Lambda$ is finite. Equivalently,~$\Lambda$ is of finite representation type.
\end{assumption}

\noindent
We make this assumption for the rest of the paper. 

\begin{definition}\label{def:Fhat}
Given a module $M$ in $\mod\Lambda$, its \emph{$\hF$-polynomial} is
$$
\hF_M= \sum_{\bd \in \bZ_{\geq 0}} \chi(\Gr_\bd(M)) \prod_{V\in\ind K_\Lambda} u_V^{\hom_\Lambda(H^0V,M)-\langle g(V),\bd\rangle}.
$$
And, for $X\in K_\Lambda$, we define
\[
\hF_X:=\hF_{H^0X}.
\]
\end{definition}

\medskip

The $\hF$-polynomials are closely related to the $F$-polynomials by the following map from the $y$-variables to the $u$-variables.

\begin{definition} \label{psidef} Define a homomorphism
$\Psi:\mathbb C[y_1,\dots,y_n]\rightarrow \mathbb C[u_V^{\pm 1}\,|\, V\in \ind K_\Lambda]$ that maps
\[
\Psi:\, y_i \longmapsto \prod_{V\in \ind K_\Lambda} u_V^{-g(V)_i}.
\]
\end{definition}

\noindent
In particular, $\Psi$ maps the monomial $y^{\bd}$ to $\prod_V u_V^{-\langle g(V), \, \bd\rangle}$. So

\begin{lemma}\label{reform} $\hF_M$ is given by
  $$\hF_M= \left(\prod_{V\in\ind K_\Lambda} u_V^{\hom_\Lambda(H^0V,\,M)}\right) \, \Psi (F_M).$$
  \end{lemma}

\medskip
We verify that

\begin{lemma}\label{lem:Ftilde}
 For any~$M \in \mod\Lambda$,~$\hF_M$ is indeed a polynomial. That is, $\hF_M \in \bZ[u_V \, | \, V\in \ind K_\Lambda]$.
\end{lemma}
\begin{proof}
Take any~$\bd$ with~$\chi(\Gr_{\bd}(M))$ non-zero. Take some submodule of~$M$,~$L$, of dimension~$d(L) = \bd$. Then we claim $\hom_{\Lambda}(H^0V, M) - \langle g(V) , \bd \rangle$ is non-negative for all $V\in\ind K_\Lambda$. Let~$V = (V^{-1}\to V^0)$.  Then we get an exact sequence
\[
V^{-1}\to V^0\to H^0V \to 0
\]
in $\mod\Lambda$. Moreover, applying the functor~$\Hom_{\Lambda}(-,L)$, gives the exact sequence of vector spaces,
 \[
  0 \to \Hom_{\Lambda}(H^0V, L) \to \Hom_{\Lambda}(V^0, L) \to \Hom_{\Lambda}(V^{-1}, L).
 \]
But exactness implies that
 \[
  \hom_{\Lambda}(H^0V, L) - \hom_{\Lambda}(V^0, L) + \hom_{\Lambda}(V^{-1}, L) \geq 0.
 \]
Since~$L$ is a submodule of~$M$,~$\hom_{\Lambda}(H^0V, M) \geq \hom_{\Lambda}(H^0V, L)$. And so, recalling also \eqref{eq:gdpairing},  $\hom_{\Lambda}(H^0V, M) - \langle g(V) , \bd \rangle \geq 0$.
\end{proof}

\subsection{Jasso Reduction}\label{Jasso-intro}
In this subsection, we give a quick introduction to Jasso reduction. See \cite{Jasso15} for further details. 

Let $M$ be a rigid indecomposable $\Lambda$-module. We define $^\perp \tau M$ to be the subcategory of $\mod \Lambda$ consisting of modules $Y$ with $\Hom(Y,\tau M)=0$. An object $Z$ in $^\perp \tau M$ is called Ext-projective if $\Ext^1(Z,Y)=0$ for any $Y\in {}^\perp \tau M$.  
Define $X$ to be the direct sum of the indecomposable Ext-projective modules in $^\perp \tau M$. This the called the Bongartz completion of $M$. It contains $M$ as a direct summand.

Define $B$ to be $\End(X)/\langle e_M \rangle$, where $e_M$ is the idempotent corresponding to the summand $M$. The algebra $B$ is called the Jasso reduction of $\Lambda$ at $M$. 

Let $\mathcal C= M^\perp \cap {}^\perp \tau M$, i.e., the modules admitting no morphisms from $M$ and no morphisms to $\tau M$. Then $\Hom(X,-)$ gives an exact equivalence  from $\mathcal C$ to $B$-mod \cite[Theorem 1.4]{Jasso15}.

Starting from $M$, we can define two torsion classes. One, which we denote $t$, consists of all quotients of sums of copies of $M$. The other is $^\perp \tau M$, which we will denote $T$. Observe that $t\subseteq T$. For $N$ a $\Lambda$-module, we write $t(N), T(N)$ for the maximal subobject of $N$ contained in $t,T$ respectively.
Define $J(N)=T(N)/t(N)$. Observe that $J(N)$ automatically belongs to $\mathcal C$, and if $N$ belongs to $\mathcal C$, then $J(N)\simeq N$. 

For $N\in\mod\Lambda$, write $\Gamma(N)$ for $\Hom(X,T(N)/t(N))\in\mod B$. 

There are also categories of projective presentations corresponding to $\mod \Lambda$, $\mathcal C$, and $\mod B$. These are, respectively,
$K^b(\proj \Lambda)$, $K^b(\proj \Lambda)/\thick(M)$, and $K^b(\proj B)$. There is a Verdier quotient from $K^b(\proj \Lambda)$ to $K^b(\proj \Lambda)/\thick(M)$, and there is a natural map from $K^b(\proj \Lambda)/\thick(M)$ to $K^b(\proj B)$. 
On the level of Grothendieck groups, the first map is the quotient by $[M]$ and the second map is an isomorphism. Their composition gives a map from $K_\Lambda$ to $K_B$. 

The above map from $K_\Lambda$ to $K_B$ restricts to a bijection which we denote $s$ from the objects of $\ind K_\Lambda$ which are compatible with $M$, other than $M$ itself, to  the objects of $\ind K_B$. The map $s$ can be viewed as a bijection from the rays of the $g$-vector fan of $\Lambda$ which are compatible with $M$ other than that corresponding to $M$, to the rays of the $g$-vector fan of $B$. This bijection induces an identification of the link of the ray generated by $M$ in the $g$-vector fan of $\Lambda$ with the $g$-vector fan of $B$. This identification is also induced by the quotient map from $K_0(\proj \Lambda)$ to
$K_0(\proj \Lambda)/[M] \simeq K_0(\proj B)$.

\begin{lemma}\label{square} For $Z\in \ind K_\Lambda$, we have $H^0(s(Z))=\Gamma(H^0(Z))$.
\end{lemma}

\begin{proof} \cite[Proposition 4.12]{Jasso15} gives a commutative square of bijections among silting objects in $K_\Lambda$ and $K_B$ and support $\tau$-tilting modules in $\mod \Lambda$ and $\mod B$. It follows that there is a similar commutative square:

  $$\begin{tikzpicture} \node (a) at (0,0) {$\{Z\in\ind K_\Lambda\setminus \{M\}\mid \Hom(Z\oplus M,\Sigma Z \oplus \Sigma M)=0\}$};
    \node (b) at (8,0) {$\{N\in \ind \mod \Lambda\mid \Hom(N,\tau N)=0\}$};
    \node (c) at (0,-2) {$\{W\in \ind \mathcal C\mid \Hom(W,\Sigma W)=0\}$};
    \node (d) at (8,-2) {$\{L\in \ind \mod B, \Hom(B,\tau B)=0\}$};
      \draw [->] (a)-- node[above] {$H^0$} (b); \draw [->] (a)--node[left]{$s$}(c); \draw [->](b)--node [right] {$\Gamma$} (d);
      \draw [->] (c)-- node[above] {$H^0$}(d);\end{tikzpicture}$$

  The desired result follows from the commutativity.
\end{proof}  

\begin{lemma} \label{wanted} For $N\in\mathcal C$, and $Z\in \ind K_\Lambda$, we have that
  $\hom(H^0Z,N)=\hom(H^0s(Z),\Gamma(N))$. \end{lemma}
\begin{proof} By Lemma \ref{square}, $H^0s(Z)=\Gamma(H^0(Z))$. So
  $\hom(H^0s(Z),\Gamma(N))=\hom(\Gamma(H^0(Z)),\Gamma(N))$. Since $H^0Z$ and $N$ are in $\mathcal C$, the desired result follows from the fact that $\Gamma$ is an exact functor from $\mathcal C$ to $\mod B$. \end{proof}

Note that Jasso reduction, as defined in \cite{Jasso15}, starts from $M$ a rigid indecomposable $\Lambda$-module. However, it is also possible to think of starting with $M$ an indecomposable rigid object in $K_\Lambda$. If $M=P_i\rightarrow 0$, then the Bongartz completion is $M\oplus (\Lambda /\langle e_i\rangle)$. Define the torsion class $t(M)=0$ and $T(M)$ consists of the modules with no support at vertex $i$.

\section{Varieties defined by~$\hF$-equations and $u$-equations}

\subsection{The varieties~$\cM_{\Lambda}$ and~$\widetilde{\cM}_{\Lambda}$}
The equations
\begin{equation}\label{eq:hFeqns}
\hF_X = 1,
\end{equation}
for all $X \in \ind K_\Lambda$, define the affine varieties that are the main subject of this paper.

\begin{definition}\label{def:varieties}
In the ring $\mathbb C[u^{\pm 1}_V\,|\, V\in \ind K_\Lambda]$, consider the ideal $I_\Lambda = \langle \hF_W - 1 \, | \, W\in \ind K_\Lambda \rangle$. Similarly, write $\tI_\Lambda$ for the ideal $\tI_\Lambda = \langle \hF_W - 1 \, | \, W\in \ind K_\Lambda \rangle$ in the polynomial ring $\mathbb C[u_V\smid V\in \ind K_\Lambda]$, generated by the same polynomials. This defines two affine varieties:
 \begin{center}
 \begin{tabular}{lll}
   $R_\Lambda := \bC[u^{\pm 1}_X \smid X\in \ind K_\Lambda]\big/ I_\Lambda$  & \qquad\qquad & $\widetilde{R}_\Lambda := \bC[u_X \smid X\in \ind K_\Lambda]\big/ \tI_\Lambda$ \\
  $\cM_\Lambda:= \Spec(R_\Lambda)$ & \qquad\qquad & $\widetilde{\cM}_\Lambda :=\Spec(\widetilde{R}_\Lambda).$
   \end{tabular}
 \end{center}
\end{definition}

\medskip

Thus $\cM_\Lambda$ is the localisation of $\widetilde{\cM}_\Lambda$ along the divisors $u_V = 0$. In other words, $\cM_\Lambda$ is obtained by removing the points where some $u_V = 0$. We emphasize that, in the case that $\Lambda$ is the path algebra of the Dynkin quiver $A_{n-3}$, $\widetilde{\cM}_\Lambda$ is \emph{not} the Deligne-Mumford compactification of the moduli space of points on $\mathbb{P}^1$, $\overline{{\cM}}_{0,n}$. The latter is a projective variety, whereas $\widetilde{M}_\Lambda$ is an affine open in $\overline{{\cM}}_{0,n}$. To avoid confusion, we therefore use a tilde, and not a bar, to denote these affine varieties. Note that these affine varieties are studied by Brown in \cite{Brown09}.


\subsection{The~$u$-equations and configuration spaces}
The varieties in Definition~\ref{def:varieties} generalize the cluster configuration spaces of finite type of~\cite{AHL}, along with other varieties previously considered in the physics and mathematics literature \cite{KN1,KN2,KN3,Brown09}.  We now show how the varieties defined by the equations \eqref{eq:hFeqns} are related to another set of equations called the~$u$-equations.

\begin{definition}\label{def:u-equations}
 Under the finiteness assumption~\ref{finiteness-assumption}, the \emph{$u$-equation} of an object~$X\in \ind K_\Lambda$ is the following equation in the polynomial ring~$\bZ[u_V \smid V \in \ind K_\Lambda]$:
 \[
  u_X + \prod_{Y\in \ind K_\Lambda} u_Y^{c(X,Y)} = 1,
 \]
where the \emph{compatibility degree} of $X, Y \in \ind K_\Lambda$ is
\[
c(X,Y) = \hom_{K_{\Lambda}}(X,\Sigma Y) + \hom_{K_{\Lambda}}(Y,\Sigma X).
\]
The $u$-equations define an affine variety
\[
\widetilde{\cU}_\Lambda = \Spec(\widetilde{S}_\Lambda),
\]
where~$\widetilde{S}_\Lambda := \bC[u_X \smid X\in \ind K_\Lambda]/\langle \textrm{$u$-equations} \rangle$.
\end{definition}

\begin{definition}\label{def:pospart}
The configuration space, ${\widetilde{\cU}}_\Lambda^{\geq 0}$, is the semialgebraic set of real, non-negative points of $\widetilde{\cU}_\Lambda$.
\end{definition}

Finally, the exponents $c(X,Y)$ appearing in the $u$-equations can be written in terms of the module category $\mod\Lambda$ as follows.

\begin{lemma}\label{lem:cXY}
For  $X,Y \in \ind K_\Lambda$, the compatibility degree $c(X,Y)$ is also given by
\[
	c(X,Y) = \hom_\Lambda(H^0 X, H^0(\tau Y)) + \hom_\Lambda(H^0Y,\, H^0 (\tau X)),
\]
or, equivalently, by
\[
   c(X,Y) = -\langle g(X), d(H^0Y) \rangle + \hom_{\Lambda}(H^0X,H^0Y) + \hom_{\Lambda}(H^0X,H^0(\tau Y)).
\]
\end{lemma}
\begin{proof}
Recall from Lemma \ref{lemm::hom} that $\hom_{K_\Lambda}(Y,\Sigma X) = \hom_\Lambda(H^0 X, H^0 (\tau Y))$. The second formula follows from Lemma \ref{lem:gd}.
\end{proof}

\subsection{Examples}
The literature includes several families of examples of~$u$-equations, that we summarize below. Note that Definition \ref{def:u-equations}, above, generalizes these examples to the case of any finite representation type algebra. For new examples of~$u$-equations, see Section~\ref{section-examples}.

\subsubsection{Dynkin types}
If the algebra~$\Lambda$ is the path algebra~$\bC Q$ of a quiver~$Q$ of Dynkin type~$A, D$ or~$E$, then the~$u$-equations defining~$\widetilde{\cU}_\Lambda$ are the same as the ones defining the variety~$\widetilde{\cM}$ of Arkani-Hamed, He and Lam~\cite{AHL}.  This is a consequence of the additive categorification of cluster algebras, which was worked out for Dynkin quivers in~\cite{CalderoChapoton06}. In particular, for type~$A_{n-3}$,~$\widetilde{\cU}_\Lambda$ is the affine variety~$\widetilde{\cM}_{0,n}$ of Brown~\cite{Brown09}. For example, for~$A_2$,~$\ind K_\Lambda$ has five modules: the projectives~$P_1, P_2$, the shifted projectives~$\Sigma P_1, \Sigma P_2$, and the simple module~$S_2$. Then the~$u$-equations for this algebra are
\begin{center}
 \begin{tabular}{lclcl}
  $u_{P_1} + u_{S_2}u_{\Sigma P_1} = 1$ & & $u_{P_2} + u_{\Sigma P_1}u_{\Sigma P_2}=1$ && $u_{S_2} + u_{P_1}u_{\Sigma P_2} = 1$\\
  $u_{\Sigma P_1} + u_{P_1}u_{P_2} = 1$ && $u_{\Sigma P_2} + u_{P_2}u_{S_2} = 1$
 \end{tabular}
\end{center}
See Section~\ref{ex:A2} for more details. We also work out type~$A_3$ in detail in Section~\ref{ex:A3}.

\subsubsection{The Grassmannian}\label{grassmannian} The CEGM~$u$-equations are defined in \cite{early2022}. As will be discussed in more detail elsewhere \cite{EPT}, they are the~$u$-equations for the path algebra,~$\Lambda_{k,n}$, of the quiver
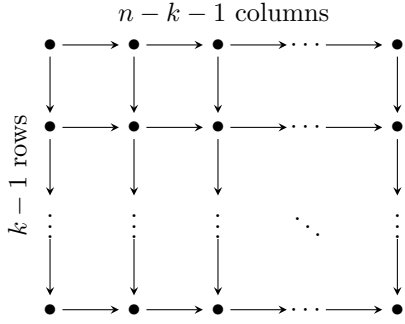
\begin{center}
\begin{tikzpicture}[
    >=stealth,
    plain/.style={inner sep=0pt, draw=none},
    ->, shorten >=1pt
]
\matrix (m) [
  matrix of nodes,
  nodes={inner sep=2pt},
  column sep=0.8cm,
  row sep=0.8cm,
  column 4/.style={nodes={plain}},
  row 3/.style={nodes={plain}}
] {
  $\bullet$ & $\bullet$ & $\bullet$ & $\cdots$ & $\bullet$ \\
  $\bullet$ & $\bullet$ & $\bullet$ & $\cdots$ & $\bullet$ \\
  $\vdots$  & $\vdots$  & $\vdots$  & $\ddots$ & $\vdots$  \\
  $\bullet$ & $\bullet$ & $\bullet$ & $\cdots$ & $\bullet$ \\
};
\foreach \i in {1,2,4}{
  \draw (m-\i-1) -- (m-\i-2);
  \draw (m-\i-2) -- (m-\i-3);
  \draw (m-\i-3) -- (m-\i-4);
  \draw (m-\i-4) -- (m-\i-5);
}
\foreach \j in {1,2,3,5}{
  \draw (m-1-\j) -- (m-2-\j);
  \draw (m-2-\j) -- (m-3-\j);
  \draw (m-3-\j) -- (m-4-\j);
}
\node[plain, above=8pt] at ($(m-1-1)!0.5!(m-1-5)$) {$n-k-1$ columns};
\node[plain, left=12pt, rotate=90] at ($(m-1-1)!0.4!(m-3-1)$) {$k-1$ rows};
\end{tikzpicture}
\end{center}
with the relations that the composition of any right-arrow with any down-arrow vanishes. For example, for the case of $k=3$, $n=6$, the indecomposable modules can be labelled by non cyclically-consecutive 3-tuples of $1,2,3,4,5,6$, and the~$u$-equations are:
\begin{align*}
u_{124}+u_{135} u_{136} u_{235} u_{236} &= 1 
&\qquad u_{125}+u_{136} u_{146} u_{236} u_{246} u_{346} &= 1 \\[6pt]
u_{134} + u_{235} u_{236} u_{245} u_{246} u_{256} &= 1
&\qquad u_{135}+ u_{124} u_{146} u_{236} u_{245} u_{256} u_{346} u_{246}^2 &= 1 \\[6pt]
u_{136} + u_{124} u_{125} u_{245} u_{246} u_{256} &= 1
&\qquad u_{145} +u_{246} u_{256} u_{346} u_{356} &= 1 \\[6pt]
u_{146} +u_{125} u_{135} u_{235} u_{256} u_{356} &= 1
&\qquad u_{235} +u_{124} u_{134} u_{146} u_{246} u_{346} &= 1 \\[6pt]
u_{236} +u_{124} u_{125} u_{134} u_{135} &= 1
&\qquad u_{245} +u_{134} u_{135} u_{136} u_{346} u_{356} &= 1 \\[6pt]
u_{246}+ u_{125} u_{134} u_{136} u_{145} u_{235} u_{356} u_{135}^2 &= 1
&\qquad u_{256}+ u_{134} u_{135} u_{136} u_{145} u_{146} &= 1 \\[6pt]
u_{346}+ u_{125} u_{135} u_{145} u_{235} u_{245} &= 1
&\qquad u_{356}+ u_{145} u_{146} u_{245} u_{246} &= 1
\end{align*}

\medskip

\subsection{The varieties $\cM$ and $\cU$}
The main result of this section is that the~$\hF$-equations \eqref{eq:hFeqns} imply the~$u$-equations. In other words, $\widetilde {\cU}_\Lambda \supseteq \widetilde {\cM}_\Lambda$. The main tool in order to prove this is a kind of ``exchange relation'' for the $\hF$-polynomials.

\begin{proposition}\label{prop:F-tilde-imply-u}
 Let~$X\in \ind K_\Lambda$.
 \begin{enumerate}
  \item If~$X$ is neither projective nor injective, with almost split conflation~~$\tau X \rightarrowtail E_X \twoheadrightarrow X$, then~$\hF_X \hF_{\tau X} = u_X \hF_{E_X} + \prod_{W \in \ind K_\Lambda} u_W^{c(W,X)}$.
  \item If~$X = P_i$, then~$\hF_X = u_X \hF_{\rad P_i} + \prod_{W \in \ind K_\Lambda} u_W^{c(W,X)}$.
  \item If $X = \Sigma P_i$, then~$\hF_{\tau X} = u_X \hF_{I_i/S_i} + \prod_{W \in \ind K_\Lambda} u_W^{c(W,X)}$.
\end{enumerate}
\end{proposition}

From Proposition \ref{prop:F-tilde-imply-u}, it follows immediately that

\begin{corollary}\label{coro:F-tilde-imply-u}
 If~$\hF_X = 1$ for all~$X\in \ind K_\Lambda$, then~$u_X + \prod_{W\in \ind K_\Lambda} u_W^{c(W,X)} = 1$ for all~$X\in \ind K_\Lambda$.
\end{corollary}

\medskip

To prove Proposition~\ref{prop:F-tilde-imply-u}, we need some relations between different~$F$-polynomials.

\begin{lemma}\label{lem:projective}
For an indecomposable projective $P_i \in \mod\Lambda$,~$F_{P_i} = F_{\rad P_i} + y^{d(P_i)}$. For an indecomposable injective $I_i \in \mod\Lambda$,~$F_{I_i} = y_i F_{I_i / S_i} + 1$.
\end{lemma}
\begin{proof}
If a submodule of $P$ is not $P$ itself, then it is a submodule of the radical, $\rad P$. $I_i$ contains the zero submodule, and all other submodules contain $\soc I_i=S_i$. 
\end{proof}

\begin{lemma}\label{theo::fpolynomials}  For~$V \rightarrowtail X \twoheadrightarrow W$ a conflation in~$K_\Lambda$:
 \begin{enumerate}
  \item  if the conflation is split, then~$F_VF_W = F_X$,
  \item if the conflation is almost split, and $W$ is not injective, then~$F_VF_W = F_X + y^{d(W)}$.
 \end{enumerate}
\end{lemma}

\begin{proof} Recall that for $X \in K_\Lambda$, $F_X := F_{H^0X}$. The identities then follow from the corresponding identities in $\mod\Lambda$, together with Lemma \ref{lem:almost-split}, which relates short exact sequences in $\mod\Lambda$ and conflations in $K_\Lambda$. Indeed, (1) holds for split short exact sequences in $\mod\Lambda$ by \cite[Prop. 2.13]{Plamondon18} (see also \cite{CalderoChapoton06, DominguezGeiss14}). And,  by Lemma \ref{lem:split}, $H^0$ sends a split conflation to a split short exact sequence in $\mod\Lambda$. Whereas (2) holds for an almost split short exact sequence in $\mod\Lambda$, which is the main result of \cite{DominguezGeiss14}. And, since $W$ is non-injective, $H^0$ sends an almost split conflation to an almost split short exact sequence (Lemma \ref{lem:almost-split}).
\end{proof}

  \smallskip

\begin{proof}[Proof of Proposition~\ref{prop:F-tilde-imply-u}]  (1) Take~$X\in \ind K_\Lambda$ neither projective nor injective.  By Lemma~\ref{theo::fpolynomials}(2), we have~$F_XF_{\tau X} = F_{E_X} + y^{d(X)}$. We can turn this into a relation among $\hF$-polynomials. Using the homomorphism $\Psi$ and Lemma \ref{reform},
 \begin{multline}
  \hF_X\hF_{\tau X} = \Psi (F_{E_X})\prod_{W\in \ind K_\Lambda} u_W^{\hom_{\Lambda}(H^0W,H^0X) + \hom_{\Lambda}(H^0W,\tau H^0X)} \notag \\  +  \prod_{W\in \ind K_\Lambda} u_W^{-\langle g(W), d(X) \rangle + \hom_{\Lambda}(H^0W,H^0X) + \hom_{\Lambda}(H^0W,\tau H^0X) } \notag
  \end{multline}
  The first term on the right hand side is equal to~$u_X\hF_{E_X}$, by Proposition \ref{prop:ar}, which implies that
  \[
 \hom_{\Lambda}(H^0W,H^0X) + \hom_{\Lambda}(H^0W,\tau H^0X)
  = \begin{cases}
                                  \displaystyle    \hom_{\Lambda}(H^0W, H^0E_X)     & \textrm{if $X \neq W$} \\
                            \displaystyle         \hom_{\Lambda}(H^0W, H^0E_X) + 1 & \textrm{if $X = W$.}
                                       \end{cases}
  \]
And the second term is equal to~$\prod_{W \in \ind K_\Lambda} u_W^{c(W,X)}$, by Lemma \ref{lem:cXY}.

(2) Take $X=P_i$. By Lemma \ref{lem:projective}, $F_X = F_{\rad H^0X} + y^{d(X)}$ and, using $\Psi$ and Lemma \ref{reform}, 
\[
\hF_X =    \Psi (F_{\rad H^0X})\prod_{W\in \ind K_\Lambda} u_W^{\hom_{\Lambda}(H^0W, H^0X)} + \prod_{W\in \ind K_\Lambda} u_W^{-\langle g(W), d(X) \rangle +  \hom_{\Lambda}(H^0W, H^0X)}.
\]
The first term is equal to~$u_X\hF_{\rad P_i}$, by Proposition \ref{prop:ar}. The second term is equal to~$\prod_{W \in \ind K_\Lambda} u_W^{c(W,X)}$, by Lemma \ref{lem:cXY}, noting that $\tau H^0X = 0$.

(3) Take $X = \Sigma P_i$ injective. By Lemma \ref{lem:projective}, $F_{\tau X} = y_i F_E + 1$, where $H^0(\tau X) = I_i$ and $H^0E = I_i/S_i$. Using $\Psi$ and Lemma \ref{reform}, this becomes
\[
\hF_{\tau X} = \Psi(F_E) \prod_{W \in \ind K_\Lambda} u_W^{\hom_\Lambda(H^0W, H^0(\tau X)) - g(W)_i} + \prod_{W\in \ind K_\Lambda} u_W^{\hom_\Lambda(H^0W, I_i)}.
\]
The second term is equal to $\prod_{W\in \ind K_\Lambda} u_W^{c(W,X)}$, by Lemma \ref{lem:cXY}, noting that $H^0X = 0$. The first term is equal to $u_X \hF_E$, because
  \begin{equation}\label{eq:injectivewrongway}
 \hom_\Lambda(H^0W,H^0(\tau X)) - g(W)_i = \begin{cases}
                                  \displaystyle     \hom_{\Lambda}(H^0W, H^0E)     & \textrm{if $X \neq W$} \\
                            \displaystyle            \hom_{\Lambda}(H^0W, H^0E) + 1 & \textrm{if $X = W$.}
                                       \end{cases}
  \end{equation}
This follows from the short exact sequence $0 \rightarrow S_i \rightarrow I_i \rightarrow I_i/S_i \rightarrow 0$, which gives
\[
0 \rightarrow \Hom_\Lambda(H^0W, S_i) \rightarrow \Hom_\Lambda(H^0W, I_i) \rightarrow \Hom_\Lambda(H^0W, I_i/S_i) \rightarrow \Ext^1(H^0W,S_i) \rightarrow 0.
\]
But, by Lemma \ref{lem:exthom}, $\mathrm{ext}^1(M,S_i) = {\hom}(S_i,\tau M)$, so that
\[
\hom_\Lambda(H^0W,I_i/S_i) - \hom_\Lambda(H^0W,I_i) = \hom_\Lambda(H^0W,S_i) - \hom_\Lambda(S_i,\tau H^0W),
\]
which implies \eqref{eq:injectivewrongway}, by Lemma \ref{lem:gd}) and observing that $H^0(\Sigma P_i) = 0$, but $g(\Sigma P_i)_i = -1$.
\end{proof}

We can prove a converse to Corollary~\ref{coro:F-tilde-imply-u} in some special cases.

\begin{proposition}\label{prop:u-equations-imply-F-hat-sometimes}
 Assume that the Auslander--Reiten quiver of~$K_\Lambda$ has no oriented cycle.
 Then the~$u$-equations imply that~$\widehat{F}_X = 1$ for all~$X\in\ind K_\Lambda$.  In particular,~$\widetilde{\cM}_\Lambda = \widetilde{\cU}_\Lambda$.
\end{proposition}
\begin{proof}
 We can assume that the algebra~$\Lambda$ is connected.  Unless~$\Lambda = \bC$, in which case the statement is trivially true, the Auslander--Reiten quiver~$\Gamma$ of~$K_\Lambda$ is connected.  Let~$\Gamma^+$ be the quiver obtained by adding arrows~$\tau X \to X$ for each non-projective~$X$.  Then~$\Gamma^+$ does not have oriented cycles.

We prove the result by an inductive procedure as follows: suppose that we have a non-empty subset~$\cI$ of~$\ind K_\Lambda$ consisting of objects~$Y$ for which~$\widehat{F}_Y = 1$ under the~$u$-equations. We start the procedure by letting~$\cI$ be the set of shifted projectives.  If~$\cI \neq \ind K_\Lambda$, then we can find an indecomposable object~$X$ not in~$\cI$ such that~$\widehat F_X = 1$ as follows.  Let~$\bar\Gamma^+$ be the quiver obtained by removing all vertices of~$\Gamma^+$ corresponding to objects in~$\cI$ and all arrows connected to them.  Since this quiver has no oriented cycle, it has a sink~$X$.  If~$X$ is a shifted projective, then it is already in $\cI$, contrary to our hypothesis.
  Otherwise, it sits in a mesh
 {\small
 \[\begin{tikzcd}
	& {E_1} \\
	{X} & {E_2} & \tau^{-1} X \\
	& {E_r}
	\arrow[from=1-2, to=2-3]
	\arrow[from=2-1, to=1-2]
	\arrow[from=2-1, to=2-2]
	\arrow[from=2-1, to=3-2]
	\arrow[from=2-2, to=2-3]
	\arrow["\vdots"{description}, draw=none, from=2-2, to=3-2]
	\arrow[from=3-2, to=2-3]
\end{tikzcd}\]
}
and by Proposition~\ref{prop:F-tilde-imply-u}, we have that
\[
 \widehat F_X \widehat F_{\tau^{-1} X} = u_{\tau^{-1} X} \prod_{i=1}^r \widehat F_{E_i} + \prod_{W\in \ind K_\Lambda} u_W^{c(W, \tau^{-1}X)}.
\]
By assumption on~$X$, the objects~$E_i$ and~$\tau^{-1}X$ are all in~$\cI$.  Therefore, we get that
\(
 \widehat F_X = u_{\tau^{-1} X}  + \prod_{W\in \ind K_\Lambda} u_W^{c(W, \tau^{-1}X)},
\)
which is equal to~$1$ under the~$u$-equations.  Thus we can add~$X$ to~$\cI$ and apply induction.  The result is proved.
\end{proof}

\begin{example}
Here are some families of examples for which Proposition~\ref{prop:u-equations-imply-F-hat-sometimes} applies:
 \begin{enumerate}
  \item if~$Q$ is any orientation of a Dynkin quiver of type~$A$, $D$ or~$E$ and~$\Lambda = \bC Q$ is its path algebra;

  \item if~$\Lambda$ is the path algebra of~$1\to 2\to \cdots \to n$ modulo any ideal generated by paths of length at least~$2$;

  \item if~$\Lambda$ is a tilted algebra of Dynkin type~$A$,$D$ or~$E$.  Indeed, in this case,~$\mod\Lambda$ is equivalent to the heart of some~$t$-structure in~$\cD^b(\mod \bC Q)$, with~$Q$ a quiver of the corresponding Dynkin type, and thus cannot have oriented cycles by Happel's theorem on the structure of~$\cD^b(\mod \bC Q)$ \cite{Happel}.
 \end{enumerate}

\end{example}

\begin{example}\label{ex:U=M}
Using a computer, we verified the equality~$\widetilde \cM_\Lambda = \widetilde \cU_\Lambda$ for the following algebras: $\bC[x]/(x^d)$ for~$d=1,2,\ldots, 9$; a $2$-cycle~$1\xrightarrow{a}2\xrightarrow{b}1$ with relation~$ab = 0$; a~$4$-cycle where all paths of length two vanish; and the quiver
\[\begin{tikzcd}
	1 & 2
	\arrow["a", shift left, from=1-1, to=1-2]
	\arrow["b", shift left, from=1-2, to=1-1]
	\arrow["x"', from=1-2, to=1-2, loop, in=30, out=330, distance=5mm]
\end{tikzcd}\]
where all path of length two vanish.
\end{example}

\section{Split Grothendieck groups}\label{sec:split}

In this section, we give bases for the split Grothendieck groups of $\mod \Lambda$ and also study the split Grothendieck group of $K_\Lambda$.

\begin{definition}
 For any extriangulated category~$\cC$, its \emph{split Grothendieck group} is 

 \[
  K_0^{\spl}(\cC) := \bigoplus_{A \in \ind\cC} \bZ\cdot[A] \big/ \langle [X]-[E]+[Y] \ | \ X \rightarrowtail E \twoheadrightarrow Y \textrm{ split conflation} \rangle. 
 \]
In the case of the module category, $\cC = \mod \Lambda$, split conflation means split short exact sequence.
\end{definition}

Consider first the split Grothendieck group of $\mod \Lambda$.   
The set of classes $[M]$ for all indecomposable modules is a basis of $K_0^{\spl}(\mod\Lambda)$, which we call the standard basis. We will give two other bases of $K_0^{\spl}(\mod \Lambda)$, each of which is dual in a suitable sense to the standard basis.

It is natural to consider classes associated to almost split short exact sequences. For any non-projective indecomposable object~$M$ of~$\mod\Lambda$, we write
\[
r_M = [\tau M]-[E]+[ M] 
\]
for the class associated to its Auslander-Reiten sequence, $0 \rightarrow \tau M \rightarrow E \rightarrow M \rightarrow 0$. In the case that $M = P_i$, we write
\[
r_{P_i} = -[\rad P_i] + [P_i],
\]
for the class associated to the right almost split inclusion. Similarly, for any non-injective indecomposable~$M$ of~$\mod\Lambda$, write
\[
\ell_M = [M] - [E] + [\tau^{-1} M],
\]
for the class associated to $0 \rightarrow M \rightarrow E \rightarrow \tau^{-1}M \rightarrow 0$. In the case $M = I_i$, write
\[
\ell_{I_i} = [I_i] - [I_i/S_i],
\]
for the class associated to the left almost split quotient. Note that  for non-projective $M$, $\ell_{\tau M} = r_M$.

There is a natural pairing on $K_0^{\spl}(\mod \Lambda)$ given by
\[
([M],[N]) = \hom_\Lambda(M,N).
\]
With respect to this pairing we have

\begin{lemma}\label{dual-bases}
The set of classes $r_M$ are a basis for $K_0^{\spl}(\mod\Lambda)$, and this is right-dual to the standard basis, in the sense that 
\[
([M],r_N) = \begin{cases} 0 & ~ M\neq N, \\ 1 & ~ M = N. \end{cases}
\]
Similarly, the set of classes $\ell_M$ are also a basis, and this basis is left-dual to the standard basis:
\[
(\ell_M, [N]) = \begin{cases} 0 & ~ M\neq N, \\ 1 & ~ M = N. \end{cases}
\]
In particular, the pairing $(~,~)$ is non-degenerate.
\end{lemma}
\begin{proof}
As already mentioned, the set of classes $[M]$ for all indecomposables $M \in \mod\Lambda$ are a basis of $K_0^{\spl}(\mod \Lambda)$. On the other hand, $([M],r_N)$ and $(\ell_M,[N])$ satisfy the stated properties by Proposition \ref{prop:ar}. But this shows that the $r_M$ (resp. $\ell_M$) also define bases of $K_0^{\spl}(\mod \Lambda)$.
\end{proof}

The following lemma is immediate (see for instance~\cite{Auslander}):
\begin{lemma}\label{lem:basisexp}
Any class $x \in K_0^{\spl}$ admits basis expansions
\[
x = \sum_{M \in \ind\Lambda} (M,x) r_M,\qquad x = \sum_{M \in \ind\Lambda} (x,M) \ell_M
\]
into the basis of $r_M$ and the basis of $\ell_M$ classes, respectively.
\end{lemma}

We now consider the split Grothendieck group $K_0^{\spl}(K_\Lambda)$. Again, the
standard basis for $K_0^{\spl}$ consists of the isomorphism classes of indecomposable objects, i.e., $\ind K_\Lambda$.

It will be useful to extend $H^0$ to a map from $K_0^{\spl}(K_\Lambda)$ to
$K_0^{\spl}(\mod \Lambda)$, sending $[X]$ to $[H^0(X)]$.

For non-projective $X \in \ind K_\Lambda$, let $\tau X \rightarrowtail E \twoheadrightarrow X$ be its almost split conflation, and write
\[
b_X = [X]-[E]+[\tau X] \in K_0^{\spl}(K_\Lambda).
\]
For $X$ indecomposable projective, define $b_X = [X]-[\rad X]$ (where we use the slight abuse of notation that~$\rad X$ denotes the minimal projective resolution of the radical of the projective module~$X$).

The assignment of a~$g$-vector~$g(X)\in \bZ^n$ to each $X \in \ind K_\Lambda$ (see Section \ref{sec:gfan}) can be extended to a linear map~$g:K_0^{\spl}\big(K_\Lambda\big) \to \bZ^n$.  The surjection $g$ can be regarded as the natural surjection of $K^{\spl}_0(K_\Lambda)$ onto the Grothendieck group $K_0(K_\Lambda) \simeq \bZ^n$, which has a basis given by the classes $[P_i]$ of the projectives. The following characterises the kernel of $g$.

\begin{theorem}[Theorem~4.13 of~\cite{PPPP}]\label{theo::groth}
The map~$g:K_0^{\spl}\big(K_\Lambda\big) \to \bZ^n$ is a split epimorphism.  The set
\[
\{ b_X \smid X \in \ind K_\Lambda \textrm{ and } X\neq P_i \}
\]
is a basis of $\ker g$ as a free abelian group.
\end{theorem}

\section{Geometry of $\mathcal M_\Lambda$}\label{sec:geoM}

In this section, we show that $\mathcal M_\Lambda$ is irreducible by realizing it as a torus quotient of an open set in affine space, and we give several different descriptions of its coordinate ring.

Define a grading on $\mathbb C[u_M^\pm\mid M\in \ind K_\Lambda]$ with grading group
$K_0^{\spl}(\Lambda)$ by 
$d(u_M)=H^0(b_M)$.

\begin{lemma} $d(\Psi(y_i))=0 \in K_0^{\spl}(\mod\Lambda)$. \end{lemma}
\begin{proof}  The degree $d(\Psi(y_i))$ is, by definition,  $-\sum_{X\in\ind K_\Lambda} g(X)_i H^0(b_X)$. For $M\in\ind \Lambda$, the multiplicity of $[M]$ in $d(\Psi(y_i))$ is the $i$-th component of $-g(b_{\tau_{K_\Lambda}^{-1}(M)})$. Theorem \ref{theo::groth} says that $b_{\tau_{K_\Lambda}^{-1}(M)}$ is in the kernel of $g$, and the result follows.  
\end{proof}

Now we can establish the following lemma:

\begin{lemma} \label{deg-F} For $M\in \mod \Lambda$, we have that $\hF_M$ is homogeneous with respect to $d$, and its degree is $[M]$. \end{lemma}

\begin{proof} First, by the previous lemma, $d(\Psi(F_M))=0$. It follows that
  $\hF_M$ is homogeneous with degree
  $\sum_{L\in \ind \Lambda}\hom(L,M)H^0(b_L)$. To determine the coefficient of $[N]$ in this sum, let $N \rightarrowtail E \twoheadrightarrow \tau^{-1} N$ be the almost split conflation in $K_\Lambda$ beginning at $N$. Then the coefficient of $[N]$ in the degree of $\hF_M$ is $(\ell_N,M)$. This is 1 if $N=M$ and zero otherwise, which proves the lemma.
  \end{proof}

This grading on $\mathbb C[u_M^\pm \mid M\in \ind K_\Lambda]$ defines an action of the torus $T=\Spec \mathbb C[K_0^{\spl}(\mod\Lambda)]$ on $\Spec \mathbb C[u_M^\pm\mid M\in \ind K_\Lambda]$. Because the $\hF_M$ are homogeneous, the $T$-action restricts to an action on $X^\circ=\Spec\mathbb C[u_M^\pm\mid M\in \ind K_\Lambda][\hF_M^{-1}\mid M\in \ind \Lambda]$, which is the locus in $\Spec \mathbb C[u_M^\pm\mid M\in \ind K_\Lambda]$ where the $\hF_M$ are non-zero.

We now have the following theorem:
\begin{theorem}\label{product} There is an isomorphism $\mathbb C[u_M^\pm,\hF_M^{-1}] \cong \mathbb C[\hF_M^\pm\mid M\in \ind \Lambda] \otimes_{\mathbb C} \mathbb C[u_M^\pm,\hF_M^{-1}]_0$ which induces an isomorphism~$R_\Lambda \cong \mathbb C[u_M^\pm,\hF_M^{-1}]_0$.\end{theorem}
\begin{proof} $\mathbb C[\hF_M^\pm]$ is a subring of $\mathbb C[u_M^\pm,\hF_M^{-1}]$,
  and contains exactly one Laurent monomial (in the $\hF_M$'s) in each degree in $K_0^{\spl}(\mod\Lambda)$. Thus, any homogeneous element of $\mathbb C[u_M^\pm,\hF_M^{-1}]$ can be uniquely written as a Laurent monomial in $\mathbb C[\hF_M^\pm]$ times an element of $\mathbb C[u_M^\pm,\hF_M^{-1}]_0$, where the subscript zero indicates the subring of degree 0.  We therefore have an isomorphism
  $$\mathbb C[u_M^\pm,\hF_M^{-1}] \xrightarrow{\sim} \mathbb C[\hF_M^\pm] \otimes_{\mathbb C} \mathbb C[u_M^\pm,\hF_M^{-1}]_0$$
  which sends~$\widehat{F}_M^{\pm 1}$ to~$\widehat{F}_M^{\pm 1}\otimes 1$.  The ideal~$I_\Lambda$ is sent on the right-hand side to the ideal generated by the~$(\widehat{F}_M -1) \otimes 1$.  Taking the quotient, we get that~$R_\Lambda \cong \mathbb C[u_M^\pm,\hF_M^{-1}]_0$, proving the theorem.
\end{proof}

\begin{corollary}
  \label{torus-quotient}
  $\mathcal M_\Lambda$ is isomorphic to $X^\circ/T$, and is thus an irreducible variety. \end{corollary}

\begin{proof} $\mathcal M_\Lambda=\Spec(R_\Lambda)$. We have showed that
  $R_\Lambda= \mathbb C[u_M^\pm,\hF_M^{-1}]_0$. The coordinate ring of
  $X^\circ$ is $\mathbb C[u_M^\pm,\hF_M^{-1}]$, so $ \mathbb C[u_M^\pm,\hF_M^{-1}]_0$ is the coordinate ring of
  $X^\circ/T$. \end{proof}

For $M\in\ind K_\Lambda$, let $$\wu_M=\frac{u_M \hF_E}{\hF_{\tau M} \hF_M},$$
where $\tau M \rightarrow E \rightarrow M$ is the Auslander-Reiten conflation ending at $M$, unless $M$ is projective, in which case set $E=\rad M$ and
$\tau M=0$. 

\begin{lemma} $\wu_M$ is homogeneous of degree zero (and thus $T$-invariant). \end{lemma}
\begin{proof}
  This is immediate from the fact that $d(u_M)=H^0(b_M)$, together with
  Lemma \ref{deg-F}.\end{proof}

\begin{theorem}
  \label{wu-theorem}
  The ring~$\mathbb C[u_M^\pm,\hF_M^{-1}]_0$ is generated by the~$\widehat{v}_M^\pm$.  
  \end{theorem}

\begin{proof}
  The isomorphism~$\mathbb C[u_M^\pm,\hF_M^{-1}] \xrightarrow{\sim} \mathbb C[\hF_M^\pm] \otimes_{\mathbb C} \mathbb C[u_M^\pm,\hF_M^{-1}]_0$ of Theorem~\ref{product} sends each~$u_M$ to~$ \frac{\widehat{F}_{\tau M}\widehat{F}_M}{\widehat{F}_E}\otimes \widehat{v}_M$.  Taking the quotient by~$I_\Lambda$ yields the result.
  \end{proof}

We will now obtain another expression for $R_\Lambda$ as a subring of the rational function field in $n$ variables, $\mathbb C(y_1,\dots,y_n)$.

\begin{proposition} 
  \label{u-proposition}
    The map $\psi$ from $\mathbb C[u_M^\pm,\hF_M^{-1}]$ to
  $\mathbb C[y_1^\pm,\dots,y_n^\pm, F_M^{-1}]$, sending:
  \begin{align*} &u_M \to 1 \textrm{ for $M\in \ind\Lambda$}\\
    &u_{\Sigma P_i} \to y_i \end{align*}
  sends $\hF_M$ to $F_M$, and is an isomorphism
  from $\mathbb C[u_M^\pm,\hF_M^{-1}]_0$ to $\mathbb C[y_1^\pm,\dots,y_n^\pm, F_M^{-1}]$.\end{proposition}

\begin{proof}
It is easy to check from the definition of $\hF_M$ that $\psi$ does indeed send $\hF_M$ to $F_M$, and thus is a well-defined map of rings from $\mathbb C[u_M^\pm,\hF_M^{-1}]$ to $\mathbb C[y_1^\pm,\dots,y_n^\pm, F_M^{-1}]$.

  The argument now is similar to the argument for Theorem \ref{product}.
  This time, we use the fact that the degrees of $u_M$ for $M\in \ind\Lambda$ form a basis for $K^{\spl}_0(\mod \Lambda)$ to write
  $$\mathbb C[u_M^\pm, \hF_M^{-1}] \xrightarrow{\sim} \mathbb C[u_M^\pm\mid M\in\ind\Lambda] \otimes_{\mathbb C} \mathbb C[u_M^\pm, \hF_M^{-1}]_0.$$
  The ideal on the left-hand side generated by the~$u_M-1$ for~$M\in\ind\Lambda$ is sent to the ideal generated by the~$(u_M-1)\otimes 1$, so taking the quotient gives
  \[
   \mathbb C[u_M^\pm,\hF_M^{-1}]/\langle u_M-1 \mid M\in\ind\Lambda\rangle
  \simeq \mathbb C[u_M^\pm,\hF_M^{-1}]_0.
  \]
  But the left hand-side is isomorphic to $\mathbb C[y_1,\dots,y_n,F_M^{-1}]$ under the map $\psi$, which yields the result.
  \end{proof}

Define $v_M=\psi(\wu_M)$.
  Explicitly, the $v_M$ are given in terms of $F$-polynomials as
\[
v_M = \frac{F_{E}}{F_M\,F_{\tau M}},\qquad v_{P_i} = \frac{F_{\rad\,P_i}}{F_{P_i}},\qquad v_{\Sigma P_i} = y_i\,\frac{F_{I_i/S_i}}{F_{I_i}},
\]
for $M \neq P_i, \Sigma P_i$, and where we write $0\rightarrow \tau M \rightarrow E \rightarrow M \rightarrow 0$ for the AR-sequence ending with $M$.

\begin{remark}\label{rema::uniform-v}
 The~$v_X$ can be described without splitting into three cases by observing that
 \[
  v_X = \frac{F_XF_{\tau X} - y^{\dimv H^0X}}{F_XF_{\tau X}}.
 \]

\end{remark}

We can now prove an analogue of Theorem \ref{wu-theorem}, regarding $\mathbb C[y_1^\pm,\dots,y_n^\pm, F_M^{-1}]$ (which is isomorphic to $R_\Lambda$ by Proposition \ref{u-proposition}). 

\begin{proposition}\label{another} $\mathbb C[y_1^\pm,\dots,y_n^\pm, F_M^{-1}]$ is generated by the $v_M^\pm$, with $M\in\ind K_\Lambda$. \end{proposition}
\begin{proof}
By Theorem \ref{wu-theorem}, $\mathbb C[u_M^\pm,\hF_M^{-1}]_0$ is generated by $\wu_M^{\pm}$. Applying
  the map $\psi$ from Proposition~\ref{u-proposition}, the image under $\psi$ of $\mathbb C[u_M^\pm,\hF_M^{-1}]_0$ is generated by the images under $\psi$ of the $\wu_M^{\pm}$, which are the $v_M^\pm$. \end{proof}

We write $\Phi$ for the map from $\mathbb C[u_M\mid M\in\ind K_\Lambda]$ to~$\mathbb C(y_1,\dots,y_n)$ sending $u_M$ to $v_M$.

\begin{theorem}\label{six-three}
 The map~$\Phi$ induces an isomorphism~$\bar\Phi:R_\Lambda \to \bC[y_1^{\pm 1}, \ldots, y_n^{\pm 1}; F_M^{-1}\mid M\in \ind K_\Lambda]$.  In particular, the ideal~$I_\Lambda$ is prime.
\end{theorem}

\begin{proof}
 The map~$\bar\Phi$ is the composition of the map~$R_\Lambda \to \bC[u_M^{\pm}, \widehat{F}_M^{-1}]_0$ from Theorem~\ref{product} and the map~$\psi$ from Proposition~\ref{u-proposition}, so the result holds.  
\end{proof}

\begin{corollary} \label{vs-solve} The rational functions $v_M$ satisfy the $\hF=1$ equations and the $u$-equations.
\end{corollary}

\begin{proof} The isomorphism of Theorem \ref{six-three} from $R_\Lambda=\mathbb C[u_M\mid M\in \ind K_\Lambda]/\langle \hF_M=1\mid M\in \ind\Lambda\rangle$ to $\mathbb C[v_M]$ sending $u_M$ to $v_M$, implies that the $v_M$ satisfy the $\hF=1$ equations. Corollary \ref{coro:F-tilde-imply-u} shows that they therefore also satisfy the $u$-equations. \end{proof}

\begin{corollary}
There is an isomorphism from an open set in $\mathbb C^n$ to  $\mathcal M_\Lambda$, sending $(y_1,\dots,y_n)$ to $(v_M(y_1,\dots,y_n))_{M\in\ind K_\Lambda}$.  In particular,~$\mathcal M_\Lambda$ is irreducible.
\end{corollary}

\begin{proof}
This is the isomorphism of varieties corresponding to the isomorphism of rings from Theorem \ref{six-three}. \end{proof}

\section{Identities relating $v_X$ and $F_M$} \label{parametrization-section}
In this section, we prove some equations relating the rational functions $v_X$ and the polynomials $F_M$. These identities will be needed in Section \ref{geometry}. They can further be used to check directly that the $v_X$ satisfy the $\hF_M=1$ equations.

We write $F$ for the homomorphism $F:K_0^{\spl}\big(\mod\Lambda\big) \rightarrow \bQ(y_1, \ldots, y_n)$ that maps classes in the split Grothendieck group to their $F$-polynomials:
\[
F:\,  \sum_i\lambda_i[A_i]  \longmapsto  \prod_i F_{A_i}^{\lambda_i},
\]
for integers $\lambda_i$. By~\cref{theo::fpolynomials}(1), this is indeed well defined. Moreover, note that
\[
F(r_{X}) = v_X^{-1},\qquad y_i F(\ell_{I_i}) = v_{\Sigma P_i}^{-1},
\]
for $X \neq \Sigma P_i$. Then the basis expansions in $K_0^{\spl}$ given by Lemma \ref{lem:basisexp} imply the following identities:

\begin{lemma}\label{lem:Fexpand}
  For any module $M$ in $\mod\Lambda$,
  \begin{align*}
  \frac{1}{F_M}&=\prod_{N\in\ind\Lambda}v_N^{\hom_\Lambda(N,M)}\\
  \frac{y^{d(M)}}{F_M}&=\prod_{X\in\ind K_\Lambda} v_X^{\hom_\Lambda(M,H^0(\tau X))}
  \end{align*}
\end{lemma}
\begin{proof} 
  The first equation follows by applying Lemma \ref{lem:basisexp} to expand $x=[M]$ in the $r_N$ basis, and then applying $F$ to the resulting identity in $K_0^{\spl}(\mod\Lambda)$.
  Similarly expanding $x=[M]$ in the $\ell_N$ basis and applying $F$, we get
  $$\frac 1{F_M}=\prod_{N\in\ind\Lambda} F(\ell_N)^{-\hom_\Lambda(M,N)}=\prod_{X\in\ind K_\Lambda} F(\ell_{H^0(\tau X)})^{-\hom_\Lambda(M,H^0(\tau X))}$$
    Now, $F(\ell_{H^0(\tau X)})=v_X^{-1}$ unless $X=\Sigma P_i$, in which case
    $y_iF(\ell_{H^0(\tau \Sigma P_i)})=v_{\Sigma P_i}^{-1}$. Thus, to be left with only powers of the $v_X$ on the righthand side, we need a compensating monomial on the lefthand side, which is just $y^{d(M)}$ because $\hom(M,H^0(\tau \Sigma P_i))=d(M)_i$. This proves the second equation. \end{proof}

\begin{lemma}\label{lem:yexpand}
For any module $M$ in $\mod\Lambda$, with dimension vector $d(M)$,
\[
y^{d(M)} = \prod_{Y\in\ind K_\Lambda} v_Y^{-\langle g(Y), d(M) \rangle}.
\]
\end{lemma}
\begin{proof}
By the identities for $1/F_{M}$ and $y^{d(M)} / F_{M}$ (Lemma~\ref{lem:Fexpand}), we find
\[
y^{d(M)} = \prod_{Y\in\ind K_\Lambda} v_Y^{\hom_{\Lambda}(M,H^0 \tau Y)-\hom_{\Lambda}(H^0Y,M)}.
\]
Then the result follows from Lemma~\ref{lem:gd}.
\end{proof}

The two lemmas above can be used to give a direct proof of Corollary \ref{vs-solve}.

\section{Tropical $u$-variables}\label{sec:trop}

In this section, we prove some results on the tropicalization of the rational functions $v_M$ for $M\in\ind K_\Lambda$, as defined just before Proposition~\ref{another}. These are of independent interest, generalizing results of \cite{KusReineke} for Dynkin quivers; we will also use them in Section \ref{geometry}. Because of the independent interest of these results, for most of this section we drop the assumption that our finite-dimensional algebra $\Lambda$ is of finite representation type. 

For $g\in \mathbb Z^n$, we can consider two-term complexes of projectives
$P^{-1}\rightarrow P^0$, with $P^{0}=\bigoplus_{g_i>0} P_i^{g_i}$, $P^{-1}=\bigoplus_{g_i<0} P_i^{-g_i}$. We call such complexes \emph{presentations of weight $g$}.
The space of such complexes is just an affine space, so in particular it is irreducible, and we can look for generic behaviours. In particular, we will be interested in the multiplicity of a given $X\in\ind K_\Lambda$ as an indecomposable summand of a generic presentation of weight $g$. Also, for $M$ a $\Lambda$-module,
define $\hom(g,M)$ to be the dimension of $\Hom$ from the cokernel of a generic presentation of weight $g$ to $M$.

Any $g\in\mathbb Z^n$ admits a canonical decomposition, similar to Kac's canonical decomposition for representations \cite[Section 4]{DF}: the canonical decomposition of $g=g_1+\dots+g_s$ if a generic presentation of weight $g$ decomposes as a direct sum of presentations of weights $g_1,\dots,g_s$. 

Tropicalization of a polynomial drops coefficients, replaces product by sum and sum by maximum. The tropicalization of a monomial is therefore a linear function, and the tropicalization of a polynomial is the maximum of the linear functions corresponding to the terms of the polynomial. Generally tropicalization is only considered for polynomials with non-negative coefficients, but it will be convenient for us to neglect that condition. 

\begin{lemma} $\trop F_M(x)= \max_{L\subset M}\dimv L \cdot x,$ where the maximum is taken over subrepresentations of $M$. \end{lemma}

\begin{proof} This is not immediate from our definition, because of the Euler characteristics in $F_M$. Suppose that $L$ is a submodule of $M$ of dimension $\mathbf e$ but $\chi(\Gr_{\mathbf e}(M))=0$. Then the term
  $\mathbf e \cdot y$ appears in the maximum on the righthand side but not on the lefthand side. However, let $P$ be the convex hull of the dimension vectors of the subrepresentations of $M$. Clearly, the maximum on the righthand side of the equation in the statement in the lemma can be restricted to the vertices of $P$, and \cite{BKT} shows that for $\mathbf e$ a vertex of $P$, we have $\Gr_{\mathbf e}(M)=\{\textrm{pt}\}$, so the Euler characteristic is 1. Thus, the terms appearing on the righthand side but but missing from the lefthand side (if any) correspond to points contained within the convex hull of the vertices of $P$, and these points will never provide the maximum.\end{proof}

Fei shows \cite[Theorem 3.6]{FeiII} the following theorem:

\begin{theorem}\label{fei} 
  For any representation $M$ of $\Lambda$, and any $g\in\mathbb Z^n$, there exists a positive integer $m$ such that
  $\hom(mg,M)=m(\trop F_M)(g)$. Further, if the statement holds for $m$, it also holds for any multiple of $m$.
\end{theorem}
\noindent
If a vector $g\in\mathbb Z^n$ lies in the support of the $g$-vector fan, then the generic presentation of weight $mg$ is just $m$ copies of the generic presentation of weight $g$, so we can take $m=1$ in Fei's theorem.

\medskip

The following Theorem generalizes a recent result of Kus--Reineke for Dynkin quivers \cite{KusReineke} to arbitrary finite-dimensional algebras.

\begin{theorem}\label{trop-v-is-mult} For $g\in\mathbb Z^n$ and $X\in\ind K_\Lambda$, we have that $-\trop v_X(g)$ equals
  the multiplicity
  of $X$ in a generic presentation of weight $g$. \end{theorem} 

\begin{proof}
  By replacing $g$ successively by each of the summands in its canonical decomposition, we can suppose that the generic presentation of weight $g$ is indecomposable. We split into three cases, based on whether a generic presentation of weight $g$ is
  \begin{enumerate}
    \item the minimal presentation of an indecomposable $\tau$-rigid module, \item the minimal presentation of an indecomposable non $\tau$-rigid module, or  \item a complex of the form $P_i\rightarrow 0$ (i.e., a shifted projective). \end{enumerate}

  We further divide into cases based on whether $X$ is
  \begin{enumerate}
  \item [(a)] a minimal presentation of a non-projective module $M$, with Auslander--Reiten sequence $$0\rightarrow \tau M \rightarrow E \rightarrow M\rightarrow 0,$$
  \item [(b)] a complex of the form $0\rightarrow P_j$ (i.e., a minimal presentation of a projective module $P_j$), or
    \item [(c)] a complex of the form $P_j\rightarrow 0$.
  \end{enumerate}

First, consider case 1. Let $N$ be the cokernel of a generic representation of weight $g$, which is well-defined up to isomorphism since we are in case 1.   

In case 1(a),
$$\trop(v_X)(g)= -\trop F_{\tau M}(g) + \trop F_E(g) - \trop F_M(g).$$
Applying Theorem \ref{fei} (and choosing $m=1$, as we may since $g$ is in the support of the $g$-vector fan), we obtain
  \begin{align*} -\trop v_X (g)&= \hom(N,\tau M)-\hom(N,E)+\hom(N,M)\\&=(N,r_M)\end{align*}
  By Lemma \ref{lem:basisexp}, $(N,r_M)$ equals the multiplicity of $M$ as a summand of $N$. (Note that nothing up to Lemma \ref{lem:basisexp} in Section \ref{sec:split} relied on the finite representation type assumption.) This proves the desired statement in this case.

Case 1(b), where $X$ is an indecomposable projective $P_j$, is very similar. The difference is that the Auslander--Reiten sequence ending at $M$ is replaced by $\rad P_j \rightarrow P_j$. However,
  $-\trop v_X (g)=(N,r_{P_j})$, and we draw the same conclusion.

In case 1(c), we have $X=\Sigma P_j$. Then
  \begin{align*} -\trop(v_X)(g) &= \trop F_{I_j}(g) - \trop y_jF_{I_j/S_j}(g)\\
    &= \hom(N,I_j)-\hom(N,I_j/S_j)-g_j\end{align*}
    We have a long exact sequence
    $$0 \rightarrow \Hom(N,S_j) \rightarrow \Hom(N,I_j) \rightarrow \Hom(N,I_j/S_j) \rightarrow \Ext^1(N,S_j) \rightarrow \Ext^1(N,I_j)=0$$
    So $$\hom(N,I_j)-\hom(N,I_j/S_j)=\hom(N,S_j)-\ext^1(N,S_j)$$.
    Now $\hom(N,S_j)$ is just $\max(g_j,0)$, while $\ext(N,S_j)=-\max(-g_j,0)$,
    so $\hom(N,S_j)-\ext^1(N,S_j)=g_j$, so $-\trop(v_X)(g)=0$ in this case, as desired.

Now consider case 2. In this case, the multiplicity of any particular $X$ in a generic presentation of weight $g$ is zero, because, while $X$ might possibly appear as some presentation of weight $g$, it will not be the generic presentation, since $g$ is not $\tau$-rigid. We must therefore show that, in this case,
    $\trop v_X(g)=0$.

 In case 2(a),    
$$\trop(v_X)(g)= -\trop F_{\tau M}(g) + \trop F_E(g) - \trop F_M(g).$$
  Applying Theorem \ref{fei}, there is a positive integer $m$ such that
  $$-m\trop v_X (g)= \hom(mg,\tau M)-\hom(mg,E)+\hom(mg,M).$$
  Let $N$ be the cokernel of some presentation of weight $mg$. Then
  $\hom(N,\tau M)-\hom(N,E)+\hom(N,M)$ is $(N,r_M)$, so by Lemma 
  \ref{lem:basisexp}, it equals the multiplicity of $M$ as a summand of $N$.
  Since $mg$ is not $\tau$-rigid, for a generic choice of $N$, this multiplicity will be zero, so $-m\trop v_X(g)=0$, as desired.

Cases 2(b) and 2(c) are again very similar to case 2(a).

Finally, consider case 3, where $g=-e_i$. In case 3(a),
$$\trop(v_X)(g)=-\trop F_{\tau M}(-e_i) + \trop F_E(-e_i) - \trop F_M(-e_i)$$
The tropical evaluation of any $F$-polynomial at $-e_i$ is zero, so $\trop(v_X)(g)=0$. Case 3(b) is disposed of in the same way.
For case 3(c), 
\[
 -\trop(v_X)(g) = \trop F_{I_j}(-e_i) - \trop y_jF_{I_j/S_j}(-e_i) =\delta_{ij},
 \]
as desired. 
\end{proof}

We now return to studying algebras $\Lambda$ of finite representation type. In this case, generic presentations of modules are completely encoded by the $g$-vector fan. Specifically, for any $g\in\mathbb Z^n$, we express it as a sum of the generators of the mininal cone containing $g$ in the $g$-fan:
$$g= c_1 g_{X_1}+\dots+c_r g_{X_r}.$$
Then the multiplicity of $X_i$ in a generic representation of weight $g$ is $c_i$, and the multiplicity of any other complex from $\ind K_\Lambda$ is zero. This implies the following corollary, which completely characterizes $\trop v_M$.

\begin{corollary} For $\Lambda$ of finite representation type, $\trop v_M$ is linear on each cone of the $g$-vector fan. If $N\in\ind K_\Lambda$ is rigid, then $\trop v_M(g_N)=-\delta_{MN}$, where $\delta_{MN}=1$ is $M$ and $N$ are isomorphic, and zero otherwise. \end{corollary}

In particular, note that if $M$ is not rigid, then $\trop v_M$ is identically zero.

We now establish a lemma about the $\widehat F_N$ polynomials. 

\begin{lemma}\label{rigid-cant-divide} Let $M\in \ind \Lambda$ be rigid, and $N\in \ind \Lambda$. Then $\widehat F_N$ is not divisible by $u_M$. \end{lemma}

Note that this lemma does not always hold without the rigidity hypothesis on $M$. See the example in Section \ref{A2-loop}, where $\hF_{P_1}$ is divisible by $u_{S_1}$. 

\begin{proof}
  By Theorem \ref{fei}, there exists an $m$ such that 
$(\trop F_N)(mg_M)=\hom(mg_M,N)$. Since $M$ is rigid, a generic presentation of $g$-vector $mg_M$ is isomorphic to $m$ copies of $M$, so we conclude that
  $(\trop F_N)(g_M)=\hom(M,N)$.

  Recall the expression given for $\hF_M$ in Lemma \ref{reform} as a monomial prefactor times $\Psi(F_M)$. 
 By definition, $-(\trop F_N)(g_M)$ is the minimum degree of $u_M$ over all terms of $\Psi F_M$, and then since $(\trop F_N)(g_M)=\hom(M,N)$, the monomial prefactor of Lemma \ref{reform} ensures that the minimum degree of $u_M$ over all the terms of $\hF_M$ is zero.  
\end{proof}

\section{Divisors and Jasso Reduction}
\label{Jasso}

In order to understand the locus in $\widetilde{\cM}_\Lambda$ where all the $u_M$  are non-negative, it is important to consider the stratification of $\widetilde{\cM}_\Lambda$ according to which of the $u_M$ are zero. In order to understand the strata, we need to understand what happens when we add $u_M$ into the ideal $\widetilde I_\Lambda$.  It turns out that 
$\mathbb C[u_Z\mid Z\in\ind K_\Lambda]/\big\langle\widetilde I_\Lambda+\langle u_M\rangle\big\rangle$ is isomorphic to the ring defined by the $\widehat F$-equations for an algebra of rank one less, the Jasso reduction of $\Lambda$ at $M$, which we denote by $B$. Recall that we have defined the Jasso reduction, and explained some of its properties, in Section \ref{Jasso-intro}. We use the notation from that section, notably, the torsion classes $t,T$ in $\mod \Lambda$, the map $\Gamma$ from $\mod \Lambda$ to $\mod B$, and the bijection $s$ from objects of $\ind K_\Lambda$ compatible with $M$ other than $M$ itself to objects of $\ind K_B$.

In what follows, we will write~$u_Z$ for the variables defining~$\widetilde{R}_\Lambda$ and~$\uu_Y$ for the ones defining~$\widetilde{R}_B$.
Let $\mu$ be the morphism from  $\bC[u_Z \ | \ Z\in \ind K_\Lambda]$ to $\bC[\uu_Y \ | \ Y \in \ind K_B]$ which sends $u_Z$ to 1 if $Z$ is incompatible with $M$, and sends $u_Z$ to $\uu_{s(Z)}$ if $Z$ is compatible with $M$ and not isomorphic to $M$, and sends $u_M$ to zero. 

Our principal goal in this section is to establish the following theorem:

\begin{theorem}\label{Jasso-thm}
Let $M\in\ind K_\Lambda$ be rigid.  The morphism~$\mu$ defined above induces an isomorphism from
  $\widetilde R_\Lambda / \langle u_M \rangle$ to $\widetilde R_{B}$.
\end{theorem}

\begin{lemma}\label{pi-hat-F} Let $N\in \ind \Lambda$. Then $\mu\widehat F_N=\widehat F_{\Gamma(N)}$.
  \end{lemma}

\begin{proof} 
Let $N$ be an indecomposable $\Lambda$-module. 
Consider $\widehat F_N|_{u_M\leftarrow 0}$.
By Lemma \ref{rigid-cant-divide}, the minimum exponent of $u_M$ appearing in $\widehat F_N$ is zero. This means that the lattice points of the Newton polytope of $N$ which maximize $\langle g_M,\cdot \rangle$ correspond to terms in $\widehat F_N$ with degree of $u_M$ of zero, while all the other lattice points have $u_M$ appearing with strictly positive degree.
Let us denote $\uQ$ the face of the Harder--Narasimhan polytope of $N$ on which $\langle g_M,\cdot\rangle$ is maximized.  
Thus,
$\widehat F_N\mid_{u_M\leftarrow 0}$ is the polynomial obtained by summing the terms in $\widehat F_N$ corresponding to lattice points in $\uQ$.

As \cite{BKT} explains, the Harder--Narasimhan polytope of $\Gamma(N)$ can be identified with $\uQ$.
Further, \cite[Section 6]{FeiI} explains that, after that identification, the coefficients of $\widehat F_{N}$ corresponding to lattice points on $\uQ$ agreee with the corresponding coefficients of $\widehat F_{\Gamma(N)}$. Write $\up$ for the bottommost point of $\uQ$, which is $\dimv\, t(N)$. 

Now:
\begin{align*}\hF_{\Gamma(N)}&= \sum_{\bd \in \uQ\cap \mathbb Z^n} \chi(\Gr_{\bd-\up}(\Gamma(N))\prod_{V\in \ind K_B}\uu_V^{\hom(H^0V,\Gamma(N)) - \langle g_V,\bd-\up\rangle} \\
  &= \sum_{\bd \in \uQ\cap \mathbb Z^n} \chi(\Gr_{\bd-\up}(\Gamma(N))\prod_{Z\in \ind K_\Lambda,\, c(Z,M)=0}\uu_{s(Z)}^{\hom(H^0s(Z),\Gamma(N)) - \langle g_{s(Z)},\bd-\up\rangle}  \end{align*}

Now, $g_{s(Z)}$ can be understood as $g_Z$ in the quotient by $g_M$. 
Since $\bd-\up$ is perpendicular to $g_M$, we can replace $g_{s(Z)}$ by $g_Z$. Further, we can replace $\hom(H^0s(Z),\Gamma(N))$ by $\hom(H^0Z,T(N)/t(N))$ by Lemma \ref{wanted}. 

  By \cite[Section 6]{FeiI}, $\Gr_{\bd-\up}(\Gamma(N))$ is isomorphic to $\Gr_{\bd}(N)$. We obtain:
  $$\hF_{\Gamma(N)}=\sum_{\bd \in \uQ\cap \mathbb Z^n} \chi(\Gr_{\bd}(N))\prod_{Z\in \ind K_\Lambda,\, c(Z,M)=0}\uu_{s(Z)}^{\hom(H^0Z,T(N)/t(N)) - \langle g_{Z},\bd-\up\rangle}$$

  On the other hand,
  $$\mu\hF_N= \sum_{\bd \in \uQ \cap \mathbb Z^n}\chi(\Gr_{\bd}(N))\prod_{Z\in \ind K_\Lambda,\, c(Z,M)=0}\uu_{s(Z)}^{\hom(H^0Z,N) - \langle g_{Z},\bd\rangle}$$

  To get these two expressions to agree, it remains to show that
  $$\hom(H^0Z,N)-\langle g_Z,\bd\rangle=\hom(H^0Z,T(N)/t(N))-\langle g_Z,\bd-\up\rangle.$$

  Consider the exact sequence
  $$ 0 \rightarrow t(N) \rightarrow N \rightarrow N/t(N) \rightarrow 0$$

  Apply $\Hom(Z,-)$ to it. Since $t(N)$ is a quotient of $M$, and $Z$ is compatible with $M$, $\Ext^1(Z,t(N))=0$. Thus, we have a short exact sequence:

  $$ 0 \rightarrow \Hom(Z,t(N))\rightarrow \Hom(Z,N) \rightarrow \Hom(Z,N/t(N))\rightarrow 0$$

  Since $Z\in {}^\perp{(\tau M)}=T$, the image of any map from $Z$ to $N$ is contained in $T(N)$, and thus the image of any map from $Z$ to $N/t(N)$ is contained in $T(N)/t(N)$. So we have:

  $$ 0\rightarrow \Hom(Z,t(N)) \rightarrow \Hom(Z,N) \rightarrow \Hom(Z,T(N)/t(N)) \rightarrow 0$$

  So
\begin{align*}
  \hom(H^0(Z),N)-\hom(H^0(Z),T(N)/t(N))-\langle g_Z,\up\rangle&=
\hom(H^0(Z)),t(N))-\langle g_Z,\dimv\, t(N)\rangle\\&=\hom(t(N),\tau H^0(Z))
\end{align*}
and since $\Hom(M,\tau H^0(Z))=0$, this is zero, which is what we needed to show.
  \end{proof}

\begin{proof}[Proof of Theorem \ref{Jasso-thm}]
  By Lemma \ref{pi-hat-F}, for every indecomposable $N\in \mod \Lambda$, $$\mu \hF_N-1=\hF_{J(N)}-1\in \widetilde I_{B}.$$ This shows that $\mu$ does induce a map from
$\widetilde R_\Lambda / \langle u_M \rangle$ to $\widetilde R_{B}$.

To show that the map induced by $\mu$ is an isomorphism, it suffices to show that, for every indecomposable $Y \in \mod B$, there exists a $\widehat Y\in \mod \Lambda$ such that $\Gamma(\widehat Y)=Y$, since this will show that $\widetilde I_B$ is generated by the image under $\mu$ of the generators of $\widetilde I_\Lambda$. The existence of a suitable $\widehat Y$ follows from the fact that $\Gamma$ is an equivalence of additive categories from a subcategory of $\mod \Lambda$ to $\mod B$. 
\end{proof}

The statement corresponding to Theorem \ref{Jasso-thm}
for the variety~$\cU_\Lambda$ (defined by the vanishing of the $u$-equations) also holds.  In the next statement, we denote by~$J_M:K^b(\proj\Lambda)/\thick(M) \to K^b(\proj B)$ the silting reduction functor.

\begin{theorem}\label{Jasso-u} Let $M\in\ind K_\Lambda$ be rigid.
  Then there is an isomorphism from
  $\widetilde{S}_\Lambda / \langle u_M \rangle$ to $\widetilde S_{B}$ which sends $u_N$ to $1$ if $N$ is incompatible with $M$, and sends $u_N$ to $\uu_{J_M(N)}$ if $N$ is compatible with $M$.
\end{theorem}
\begin{proof}
 Let~$\cZ_M$ be the full subcategory of~$K_\Lambda$ consisting of objects~$N$ which satisfy that~$\Ext^1(N,M) = 0 = \Ext^1(M,N)$.  By \cite[Lemma 3.13]{Garcia}, the additive quotient category~$\cZ_M/[M]$ is equivalent to~$K^{[-1,0]}(\proj B)$ (and the equivalence is induced by the silting reduction functor $J_M$).  Moreover, by~\cite[Lemma 3.4]{IyamaYang}, if~$X,Y\in \cZ_M$, then~$\Ext^1(X,Y) \cong \Ext^1(J_M(X), J_M(Y))$. Thus the morphism of rings
 \[
  \phi_M : \bC[u_X \ | \ X\in \ind K_\Lambda] \longrightarrow \bC[\uu_Y \ | \ Y \in \ind K_B]
 \]
 sending~$u_M$ to~$0$, $u_N$ to~$1$ if~$N$ is incompatible with~$M$ and~$u_N$ to~$u_{J_M(N)}$ if~$N$ is compatible to~$M$ is surjective and sends~$u$-equations to~$u$-equations.  Thus it induces a surjective morphism
 \[
  \phi_M : \widetilde{S}_\Lambda/\langle u_M \rangle \longrightarrow \widetilde{S}_{B}.
 \]
 Moreover, the morphism
 \[
  \psi_M : \bC[\uu_Y \ | \ Y \in \ind K_B] \longrightarrow \bC[u_X \ | \ X\in \ind K_\Lambda]
 \]
 sending~$\uu_{J_M(X)}$ to~$u_X$ sends the~$\uu$-equations to~$u$-equations after sending~$u_M$ to~$0$ and~$u_N$ to~$1$ for all~$N$ incompatible with~$M$.  Thus it induces a surjective morphism
 \[
  \psi_M : \widetilde{S}_{B}  \longrightarrow  \widetilde{S}_\Lambda/\langle u_M \rangle.
 \]
 Since~$\psi_M$ is clearly a right-inverse to~$\phi_M$, it is injective.  Thus it is an isomorphism, and so is~$\phi_M$.
\end{proof}

\section{Geometry of $\widetilde {\cM}_\Lambda$}\label{geometry}

In this section, we put together the results from Section \ref{sec:geoM} on the geometry of $\cM_\Lambda$ together with results from Sections \ref{sec:trop} and 
\ref{Jasso} to provide a description of the geometry of $\widetilde{\cM}_\Lambda$.

Recall that the $v_M$ for $M\in\ind K_\Lambda$ were defined in section \ref{sec:geoM}, and that $R_\Lambda=\mathbb C[u_M^{\pm1}\mid M\in \ind K_\Lambda]/I_\Lambda$, the coordinate ring of $\cM_\Lambda$, was shown in Proposition \ref{another} to be isomorphic to the subring of $\mathbb C(y_1,\dots,y_n)$ generated by $v_M^{\pm 1}$.

We will now prove a similar statement for $\widetilde R_\Lambda=\mathbb C[u_M\mid M\in\ind K_\Lambda]/\widetilde I_\Lambda$, the coordinate ring of $\widetilde {\cM}_\Lambda$.
To pass from the statement for $R_\Lambda$ to the statement for $\widetilde R_\Lambda$, we will need the following lemma.

\begin{lemma}\label{lem:localisation-in-ideal}
 Let~$R$ be a commutative Noetherian ring, and let~$I$ be an ideal of~$R$ contained in a prime ideal~$P$.  Let~$x_1, \ldots, x_n\in R\setminus P$.  Suppose that~$I_{x_1, \ldots, x_n} = P_{x_1, \ldots, x_n}$ and~$I+\langle x_i\rangle = P+\langle x_i \rangle$ for each~$i\in\{1, \ldots, n\}$.  Then~$I=P$.
\end{lemma}
\begin{proof}
To simplify notations, let~$S = \{x_1, \ldots, x_n\}$.
The equality~$I_S = P_S$ implies that, for each~$y\in P$, there is an~$r\geq 1$ such that~$(x_1\cdots x_n)^ry \in I$.  Since~$R$ is Noetherian,~$P$ is finitely generated, so there is a~$r\geq 1$ such that~$(x_1\cdots x_n)^r P \subseteq I$.  In particular,~$P/I$ is annihilated by~$(x_1\cdots x_n)^r$ inside~$R/I$.

Now say that~$P = \langle y_1, \ldots, y_m\rangle$.  For each~$j$, we have that~$y_j\in P+\langle x_1\rangle = I+\langle x_1\rangle$, so we can write~$y_j = x_1g_{1,j} + h_{1,j}$, with~$h_{1,j}\in I$.  But then~$x_1g_{1,j} = y_j - h_{1,j}\in P$, which is prime, so since~$x_1\notin P$, we get that~$g_{1,j}\in P$.  Projecting into~$P/I$, we get that~$P/I$ is generated by the~$x_1g_{1,j}$.  Thus~$P/I = x_1(P/I)$. Repeating by replacing~$x_1$ with any~$x_i$, we get that~$P/I = x_i(P/I)$.  But then~$P/I = (x_1\cdots x_n)^r P/I = 0$.  Thus~$P=I$.
\end{proof}

Let us write $\wU_\Lambda$ for the subring of $\mathbb C(y_1,\dots,y_n)$ generated by the $v_M$.

\begin{theorem}\label{thm:RViso}
  $\widetilde R_\Lambda$ is isomorphic to $\wU_\Lambda$.  In particular,~$\widetilde R_\Lambda$ is an integral domain,~$\widetilde I_\Lambda$ is a prime ideal, and~$\widetilde{\cM}_\Lambda$ is irreducible.
\end{theorem}

   \begin{proof}

We note that~$(\tI_\Lambda)_{u_X, X\in\ind K_\Lambda} = I_\Lambda$, which is prime by Theorem~\ref{six-three}, since~$I_\Lambda = \ker \Phi$ and the image of~$\Phi$ is an integral domain.  Let~$\iota:\bC[u_X \ | \ X\in \ind K_\Lambda] \to \bC[u^{\pm 1}_X \ | \ X\in \ind K_\Lambda]$ be the localisation map, and let~$P_\Lambda = \iota^{-1}(I_\Lambda)$.  Note that~$P_\Lambda$ is the kernel of the map~$\bC[u_X \ | \ X\in \ind K_\Lambda] \to \wU_\Lambda$ sending each~$u_X$ to~$v_X$, since this map factors as~$\Phi\iota$.  To prove the theorem, it suffices to show that~$P_\Lambda = \tI_\Lambda$.  The proof is by induction on~$n$.

Note that~$P_\Lambda$ is prime, since~$I_\Lambda$ is.  Also,~$P_\Lambda$ contains~$\tI_\Lambda$, but does not contain any of the~$u_X$; otherwise,~$u_X$ would be in~$I_\Lambda$, which is the kernel of~$\Phi$, but~$\Phi(u_X) = v_X \neq 0$. Moreover,~$(P_\Lambda)_{u_X, X\in \ind K_\Lambda} = I_\Lambda$.  Thus the first conditions of Lemma~\ref{lem:localisation-in-ideal} applies to~$\tI_\Lambda \subseteq P_\Lambda$.

To apply Lemma~\ref{lem:localisation-in-ideal}, it remains to be shown that~$P_\Lambda+\langle u_X\rangle = \tI_\Lambda + \langle u_X \rangle$ for each~$X\in \ind K_\Lambda$.    First, if~$X$ is not rigid, then in the~$u$-equation~$u_X + \prod_{W\in \ind K_\Lambda} u_W^{c(W,X)} = 1$, the left hand side is divisible by~$u_X$.  Thus~$1\in \tI_\Lambda + \langle u_X \rangle$, so the ideal is the whole ring.  Since this ideal is contained in~$P_\Lambda + \langle u_X \rangle$, they are both equal to the whole ring, so~$\tI_\Lambda + \langle u_X \rangle = P_\Lambda + \langle u_X \rangle$.

Before we continue, we will need the following fact: the ideal~$P_\Lambda$ has coheight~$n$.  To see this, recall that~$\iota$ induces an isomorphism of posets between prime ideals of~$\bC[u_X \ | \ X\in \ind K_\Lambda]$ not containing monomials in the~$u_X$'s and prime ideals of~$\bC[u_X^{\pm 1} \ | \ X\in \ind K_\Lambda]$.  The inverse of this bijection is given by localizing the~$u_X$'s.   This bijection sends~$P_\Lambda$ to~$I_\Lambda$.  Moreover, since the algebras involved are finitely generated~$\bC$-algebras, the bijection sends maximal ideals to maximal ideals; thus the bijection preserves the height and coheight of ideals.   By Theorem~\ref{six-three}, $I_\Lambda$ has coheight~$n$.  Thus~$\coht P_\Lambda = \coht I_\Lambda = n$.

Assume now that~$X$ is rigid.  If~$n=1$, then no other indecomposable object~$Y$ is such that~$c(X,Y)=0$.  Thus sending~$u_X$ to zero brings~$u_Y$ to 1.  Therefore~$\bC[u_W \ | \ W \in \ind K_\Lambda]/(\tI_\Lambda + \langle u_X \rangle) = \widetilde{R}_\Lambda / \langle u_X \rangle = \bC$, so the coheight of~$\tI_\Lambda + \langle u_X \rangle$ is zero.  As shown above, the coheight of~$P_\Lambda$ is~$n=1$, so by Krull's principal ideal theorem, any minimal prime ideal~$Q$ containing~$P_\Lambda+\langle u_X \rangle$ has coheight at most~$n-1=0$ provided we prove that it is not equal to the whole ring, which we do at the end of this proof for any~$n>1$. In that case, since~$\tI_\Lambda + \langle u_X \rangle \subset P_\Lambda+\langle u_X \rangle \subset Q$, comparing the coheights tells us that all three ideals are equal.

Assume now that~$n>1$.  By induction hypothesis,~$\tI_\Lambda + \langle u_X \rangle$ is prime of coheight~$n-1$.  By Krull's principal ideal theorem, any minimal prime ideal~$Q$ containing~$P_\Lambda+\langle u_X \rangle$ has coheight at most~$n-1$ (again provided we can prove that it is not equal to the whole ring, as proved below), so again the inclusions~$\tI_\Lambda + \langle u_X \rangle \subset P_\Lambda+\langle u_X \rangle \subset Q$ implies equality of all three ideals.

Finally, we prove that~$P_\Lambda+\langle u_X \rangle$ is not equal to the whole ring.  This is equivalent to proving that there is a point in the variety~$V(P_\Lambda)$ whose~$u_X$ coordinate vanishes.  Let~$N$ be the number of objects in~$\ind K_\Lambda$.  Since~$P_\Lambda = \iota^{-1}(I_\Lambda)$ where~$\iota$ is the above localisation map, we can view the variety~$V(P_\Lambda)$ as the Zariski closure of~$V(I_\Lambda)$ inside~$\bC^N$.  We now use the isomorphism of Theorem~\ref{product}
\[
 \mathbb C[u_M^\pm,\hF_M^{-1}] \cong \mathbb C[\hF_M^\pm\mid M\in \ind \Lambda] \otimes_{\mathbb C} \mathbb C[u_M^\pm,\hF_M^{-1}]_0
\]
sending each~$u_M$ to~$\widehat{v}_M$ which allows us to identify~$\cM_\Lambda$ with the subset of~$\bC^N$ where the~$\widehat{F}_M$ are equal to~$1$ and the~$u_M$ don't vanish.  Take a continuous map~$\gamma:[0,1]\to \bC^N$ such that~$\gamma(0)$ has a zero at coordinate~$u_X$ and only at that coordinate, and such that~$\gamma$ avoids the hypersurfaces defined by~$\widehat{F}_M = 0$ and~$u_M = 0$ otherwise; this is possible, since none of the~$\widehat{F}_M$ are divisible by~$u_X$ by Lemma~\ref{rigid-cant-divide}.  Now define~$\bar \gamma:[0,1]\to \bC^N$ by first applying~$\gamma$ to~$t$ and then multiplying coordinate~$u_M$ by~$\frac{\widehat{F}_M(\gamma(t))\widehat{F}_{\tau M}(\gamma(t))}{\widehat{F_{E_M}}(\gamma(t))}$.  This amounts to acting by the element of the torus $T=\Spec \mathbb C[K_0^{\spl}(\mod\Lambda)]$ which rescales each $F_{\widehat M}$ to 1. Then the image of~$\bar \gamma : (0,1]$ is in the subset of~$\bC^N$ where the~$F_{\widehat{M}}=1$ and the~$u_M \neq 0$, and~$\bar \gamma(0)$ is such that~$u_X=0$.  This proves that there is a point in the closure of~$\cM_\Lambda$ whose~$u_X$ coordinate does not vanish.  This finishes the proof.
\end{proof}

We now characterize $\widetilde R_\Lambda$ as a subring of $\mathbb C(y_1,\dots,y_n)$ in terms of tropical properties. 

\begin{proposition} 
  \label{trop-characterization} The subring of $\mathbb C(y_1,\dots,y_n)$ generated by the
  $v_M$ for $M\in\ind K_\Lambda$ (which we denote $\mathbb C[v_M]$) is exactly the subring generated by the rational functions $w$ such that:
  \begin{enumerate}
    \item $w$ is a Laurent monomial in the $y_i$ and $F_M$, and
    \item $\trop(w)$ is non-positive.
\end{enumerate}\end{proposition}

\begin{proof} That $v_M$ for $M\in\ind K_\Lambda$ satisfy (1) and (2) is immediate from Theorem \ref{trop-v-is-mult}. By Lemmas \ref{lem:Fexpand} and \ref{lem:yexpand}, we can re-express a Laurent monomial $w$ as in (1) as a Laurent monomial in the $v_M$. By Theorem \ref{trop-v-is-mult}, the power of each $v_M$ must be positive, except perhaps for any $v_M$ with $M$ non-rigid. But observe that if $M$ is non-rigid, then the $u$-equation for $M$ has a factor of $u_M$ on the lefthand side, so dividing both sides by it, we get an expression for $u_M^{-1}$ with both terms having all $u_N$ appearing only with positive powers.
\end{proof}
      
Write $F$ for $\prod_M F_M$. Let $X_\Lambda$ be the toric variety defined by the $g$-vector fan, with $y_i$ the coordinate corresponding to the ray $e_i=g(P_i)$.

A subsemigroup $S$ of $\mathbb Z^n$ is called saturated if for any $v \in
\mathbb Z^n$ such that for some positive integer $k$ we have $kv\in S$, it follows that $v\in S$ \cite[Definition 1.3.4]{CLS}. 
A lattice polytope $P$ in $\mathbb Z^n$ is called very ample if for every vertex $v$ of $P$,
the semigroup generated by $(P\cap \mathbb Z^n)-v$ is saturated \cite[Definition 2.2.17]{CLS}. It is also true that $P$ is very ample if and only if for $k$ sufficiently large, any lattice point in $kP$ can be written as a sum of $k$ lattice points of $P$ \cite[Exercise 2.23]{BrunsG}.

If $P$ is a full-dimensional very ample polytope in $\mathbb Z^n$, we can take $P\cap \mathbb Z^n = \{p_1,\dots,p_s\}$. Then the toric variety associated to the outer normal fan to $P$  can be defined as the Zariski closure of the image of the map sending $T=(\mathbb C^*)^n$ to $\mathbb P^{s-1}$, sending $t$ to $[t^{p_1}:\dots:t^{p_s}]$. (See \cite[Section 2.3]{CLS}.)

\begin{theorem}\label{thm:Mtildetoric}
  The affine scheme $\widetilde \cM_\Lambda$ is isomorphic to the affine open subset of $X_\Lambda$ where $F$ is non-vanishing. The subvariety $\cM_\Lambda$ is the intersection of $F\ne 0$ with the open torus orbit in $X_\Lambda$.
  \end{theorem}

\begin{proof} The outer normal fan of $P$, the Newton polytope of $F$, coincides with the $g$-vector fan of $\Lambda$ by \cite[Corollary 5.25(b)]{AHI}. Since the $g$-vector fan of $\Lambda$ is unimodular, $P$ is automatically very ample \cite[Proposition 2.4.4]{CLS}. (This resolves in the affirmative the first part of \cite[Question 5.4]{AHL}.) 
  
  Let $P\cap \mathbb Z^n=\{p_1,\dots,p_s\}$, and consider the embedding of $X_\Lambda$ into $\mathbb P^{s-1}$ as above.  The polynomial $F$, the exponents of whose monomials are a subset of $P\cap \mathbb Z^n$, defines a hyperplane in
  $\mathbb P^{s-1}$. The complement of this hyperplane is an affine space. The coordinate functions from $\mathbb P^{s-1}$ pull back to $y^{p_i}/F$ on $X_\Lambda$. Since $\trop(y^{p_i}/F)$ is non-positive, by Proposition \ref{trop-characterization}, $\mathbb C[y^{p_i}/F]$ is contained in $\mathbb C[u_M]$. 

  For the opposite inclusion, we want to show that $v_M$ is contained in $\mathbb C[y^{p_i}/F]$ for any $M\in\ind K_\Lambda$. Pick a particular $M$. Its
  denominator divides $F$. Since $P$ is very ample, we can choose $k$ 
  sufficiently large so that any lattice point in $kP$ can be written as a sum of $k$ lattice points of $P$. We can then write
  $v_M=\overline A(y)/F^k(y)$ for some polynomial $\overline A(y)$. Because
  $\trop(v_M)\leq 0$, the exponents of monomials appearing in $A(y)$ are all contained in $kP$. 
  Let $y^u$ be a monomial appearing in $\overline A(y)$. By the assumption on $k$, we can write $u={u_1}+{u_2}+\dots+{u_k}$ with each $u_i\in P$. Then
  $y^u/F^k(y)=(y^{u_1}/F)\dots(y^{u_k}/F)$ belongs to $\mathbb C[y^{p_i}/F]$, and
  summing over all the monomials of $\overline A$ we have the desired result. 
\end{proof}

\section{Non-negative real part of $\widetilde{\mathcal M}_\Lambda$}

In this section, we investigate the geometry of $\widetilde{\mathcal M}_\Lambda^{>0}$ and
$\widetilde{\mathcal M}_\Lambda^{\geq 0}$, the loci inside $\widetilde{\mathcal M}_\Lambda$ where the coordinates $u_M$ are real and positive (respectively, non-negative). We refer to these as the totally positive (respectively, totally non-negative) parts of $\widetilde{\mathcal M}_\Lambda$. 
We will study them by showing that they coincide with the corresponding parts of the toric variety $X_\Lambda$.

Let $\Delta$ be a complete, unimodular fan in $\mathbb Z^n$, and let $X_\Delta$ be the corresponding toric variety. It admits a covering by a finite number of patches isomorphic to
$\mathbb C^n$. The totally non-negative part of $X_\Delta$, which we denote $X_\Delta^{\geq 0}$ is glued together from the copy of $\mathbb R^n_{\geq 0}$ contained in each copy of $\mathbb C^n$. (See \cite[Section 4.2]{Fulton}, \cite[Section 12.2]{CLS} for more details.) The totally positive part is glued together in the same way, from copies of $\mathbb R^n_{>0}$.

\subsection{Agreement of totally positive parts}
We begin by showing the following proposition:

\begin{proposition}\label{prop:coincide} When $\widetilde{\mathcal M}_\Lambda$ is identified as an open subset of $X_\Lambda$, the totally positive parts of the two varieties coincide. \end{proposition}

Proposition \ref{prop:coincide} follows directly from the next two lemmas.

\begin{lemma} \label{converse-lem} If $x\in \widetilde{\mathcal M}$ is a point where $v_M$ is real and positive for all $M\in\ind K_\Lambda$, the $y_i$ are real and positive. \end{lemma}

\begin{proof} This is immediate from Lemma \ref{lem:yexpand}. \end{proof}

\begin{lemma}\label{F-positivity} If $x\in\widetilde {\mathcal M}$ is a point where $y_1,\dots,y_n$ are positive, then $v_M$ is real and positive for all $M\in\ind K_\Lambda$. \end{lemma}

\begin{proof}
Clearly, it is sufficient to show that $F_M(y)>0$ for all positive $y$. If all the $y_i$ are close to zero, then this is certainly true, because the terms in $F_M$ are dominated by the constant term.
Now assume there is a path in $\mathbb R^n_{>0}$ from a point close to the origin to some  point in $y$-space where~$v_N$ is non-positive.  Then on that path, some of the~$F_M$ have to be non-positive.  Consider the first point~$p$ on the path where some~$F_M$ vanishes.  Then, using Remark~\ref{rema::uniform-v}, we see that~$v_M$ approaches~$-\infty$ as we approach that point.  But this means that~$v_M$ has to vanish on some earlier point on the path, and so some~$F_{M'}$ vanishes too, contradicting the minimality of~$p$.
\end{proof}

\begin{remark} In all examples of finite representation type algebras that we are aware of, all the coefficients of all their $F$-polynomials are positive. If this were true for all finite representation type algebras, it would simplify the proof of the previous lemma.\end{remark} 

\begin{proof}[Proof of Proposition \ref{prop:coincide}]
  The totally non-negative part of $X_\Lambda$ consists of the points where $y_1,\dots, y_n$ are real and positive. By Lemma \ref{F-positivity}, such points are also points in the totally positive part of $\widetilde{\mathcal M}_\Lambda$. Conversely, by Lemma \ref{converse-lem}, a totally positive point of
  $\widetilde {\mathcal M}_\Lambda$ is also in the totally positive part of $X_\Lambda$. \end{proof}

\subsection{Agreement of totally non-negative parts}

A key step towards showing that the totally non-negative parts agree is the following proposition:

\begin{proposition}\label{prop:remove} $X_\Lambda \setminus \widetilde{\mathcal M}_\Lambda$ does not intersect the closure in the analytic topology of $X_\Lambda^{> 0}$. \end{proposition}

To establish this proposition, we need the following lemma:

\begin{lemma} There is some positive $\epsilon$ such that for all $M\in\ind \mod \Lambda$, we have $F_M(x)>\epsilon$ for all $x\in X_\Lambda^{>0}$. \end{lemma}

\begin{proof} The proof is by induction. Suppose that we already know the statement for all submodules of $M$; we want to prove it for $M$. 
  We established in the proof of Lemma \ref{F-positivity} that $F_M(x)>0$ for all $x\in X_{\Lambda}^{>0}$. Thus, by compactness, the only necessity is to analyze what happens as $x$ approaches one of the coordinate hyperplanes, say, the where $y_i=0$. If $M$ has no support over the vertex $i$, then $F_M$ does not depend on the value of $y_i$, so there is no problem with approaching this coordinate hyperplane. On the other hand, if $M$ has support over vertex $i$, then let $M_{\langle i\rangle}$ be the largest submodule of $M$ with no support over the vertex $i$. As $y_i\rightarrow 0$, we have that $F_M$ tends towards $F_{M_{\langle i\rangle}}$. We know by induction that $F_{M_{\langle i \rangle}}$ stays bounded away from zero, so the same thing is true for $F_M$. \end{proof}

We can now prove Proposition \ref{prop:remove}.

\begin{proof}[Proof of Proposition \ref{prop:remove}]
  Because the functions $F_M$ are bounded away from zero on $X_\Lambda^{>0}$, they are also strictly positive on $X_\Lambda^{\geq 0}$. The locus $X_\Lambda\setminus \widetilde{\mathcal M}_\Lambda$ is characterized by the vanishing of some $F_M$. Thus, there is no intersection between $X_\Lambda^{\geq 0}$ and $X_\Lambda\setminus\widetilde{\mathcal M}$.  
\end{proof}
  
\begin{theorem}\label{thm:agree} The totally non-negative parts of $X_\Lambda$ and $\widetilde{\mathcal M}$ coincide.
\end{theorem}
  
\begin{proof}
  $X_\Lambda^{\geq 0}$ can be viewed as the closure of $X^{>0}_\Lambda$ in the analytic topology on $X_\Lambda$, while $\widetilde{\mathcal M}^{\geq 0}$ can be viewed as the closure of $\widetilde{\mathcal M}^{> 0}$ in the analytic topology on
  $\widetilde{\mathcal M}_\Lambda$. By Proposition \ref{prop:coincide} the two totally positive parts coincide, and by Proposition \ref{prop:remove}, their closures do too.
\end{proof}

\subsection{Stratification}
For $P$ a lattice polytope having $\Delta$ as its outer normal fan, the  algebraic moment map $\mu_P$ provides a homeomorphism from $X_\Delta^{\geq 0}$ to $P$. This provides a stratification of $X_\Delta^{\geq 0}$ with strata indexed by the cones of $\Delta$ (or the faces of $P$): the stratum associated to the cone $\sigma\in\Delta$ is the inverse image under $\mu_P$ of the relative interior of the face having outer normal fan $\sigma$.

We now show that $\widetilde{\mathcal M}_\Lambda^{\geq 0}$ admits a stratification with the same properties.

\begin{proposition}\label{prop:strat}
$\widetilde{\mathcal M}^{\geq 0}_\Lambda$ admits a stratification whose strata correspond to the cones of the $g$-vector fan of $\Lambda$. The stratum corresponding to a $d$-dimensional cone is isomorphic to $\mathbb R^{n-d}$. \end{proposition}

\begin{proof} This follows from the above statements regarding the stratification of $X^{\geq 0}_\Lambda$, and Theorem \ref{thm:agree}. \end{proof}

\subsection{Positive geometry}
$\widetilde {\mathcal M}_\Lambda$ has a natural top form on it, namely the top form of the big torus contained in the toric variety $X_\Lambda$, which is
$dy_1/y_1\wedge\dots\wedge dy_n/y_n$. It is natural to ask whether
$\widetilde{\mathcal M}_\Lambda$ equipped with this differential form is a positive geometry in the sense of \cite{AHBL}. For a rather trivial reason, it isn't: namely, it is not projective. However, except for that issue, it is indeed a positive geometry, because its totally positive part agrees with the totally positive part of its ambient toric variety, and toric varieties are shown in \cite{AHBL} to be totally positive. (We are also implicitly using the fact that the boundaries of the totally positive part are again varieties of the same form.)

\section{Functoriality under algebra quotients}\label{quotient-section}
Let $J$ be a two-sided ideal of $\Lambda$, and consider $A = \Lambda / J$. Since indecomposable $A$-modules are also indecomposable $\Lambda$-modules, $A$ also satisfies our finiteness assumption.

Our goal in the section is to show:

\begin{theorem}\label{quotient} Let $A=\Lambda/J$. 
There is a surjective map from $\widetilde \cM_\Lambda$ to $\widetilde \cM_A$.

This map also restricts to a map from $\cM_\Lambda$ to $\cM_A$, and is natural in the sense that if we consider some further quotient $B$ of $A$, the map from $\widetilde \cM_\Lambda$ to $\widetilde \cM_B$ factors as a composition of the maps
from $\widetilde \cM_\Lambda$ to $\widetilde \cM_A$ and that from
$\widetilde \cM_A$ to $\widetilde \cM_B$.
\end{theorem}

For $\overline  M$ an indecomposable $A$-module, and $\overline N$ an $A$-module, define $[\overline N: \overline M]$ to be the multiplicity of $M$ as an indecomposable summand of $N$. 

Let $\pi$ be the functor from $K_\Lambda$ to $K_A$ defined by $- \otimes_\Lambda A$.

Consider a new set of variables $\overline u_N$ for $N\in \ind K_A$.

We prove Theorem \ref{quotient} by establishing the following proposition:

\begin{proposition} There is a ring map $\phi$ from $\widetilde R_A$ to
  $\widetilde R_\Lambda$ defined by:
  $$ \phi(\overline u_{\overline M})= z_{\overline M}:=\prod_{N\in \ind K_\Lambda} u_N^{[\pi N:\overline M]}.$$
\end{proposition}

For $\overline L\in\ind K_A$, note that we can consider its $F$-polynomial either as an $A$-module or as a $\Lambda$-module. We distinguish them by writing
$F_{\overline L_A}$ and $F_{\overline L_\Lambda}$ (and similarly for $\hF$-polynomials).

\begin{proof}
To establish the proposition, we need to check that for each $\overline L\in\ind K_A$,

$$\hF_{\overline L_A}|_{\overline u_{\overline M} \,\leftarrow z_{\overline M}} = 1,$$
that is to say, the result of substituting $z_{\overline M}$ for $\overline u_{\overline M}$ (for all $\overline M\in \ind K_A$) gives 1.
As a shorthand, we write $\hF_{\overline L_A}(z)$ for the result of this substitution. 

In fact, we will show that $\hF_{\overline L_A}(z)=\hF_{\overline L_\Lambda}$. This establishes what we need, since $\hF_{\overline L_{\Lambda}}=1$ in
$\widetilde R_\Lambda$. 

We now prove this claim. By definition,

$$\hF_{\overline L_A}(z)=
  \sum_{\bd\in \bZ_{\geq 0}} \chi\left( \Gr_{\bd}(H^0\overline L) \right)
  \prod_{\overline N\in\ind K_A}
  \prod_{M\in\ind K_\Lambda}
  u_{M} ^{[\pi M:\overline N] (\hom (H^0\overline N, H^0 \overline L) - \langle g(\overline N),\bd\rangle)}$$

  Collecting all the $u_M$ terms together, we see that they come from the indecomposable summands of $\pi M$, so we can rewrite the previous equation as:

  $$\widetilde F_{\overline L_A}(z)=
  \sum_{\bd\in \bZ_{\geq 0}} \chi\left( \Gr_{\bd}(H^0\overline L) \right)
  \prod_{M\in\ind K_\Lambda}
  u_{M} ^{\hom (H^0 \pi M, H^0 \overline L) - \langle g(\pi M),\bd\rangle)}$$

  The $g$-vector of $\pi M$ coincides with the $g$-vector of $M$ by definition,
  unless $P_i$ is contained in $J$ for some $i$. But for such $i$, $\bd_i$ is necessarily 0. Thus, we can replace $g(\pi M)$ by $g(M)$. 
  Further, any morphism from $H^0M$ into an $A$-module necessarily factors through $M\otimes_\Lambda A = H^0(\pi M)$. So we have:

  \begin{eqnarray*}\widetilde F_{\overline L_A}(z)&=&
  \sum_{\bd\in \bZ_{\geq 0}} \chi\left( \Gr_{\bd}(H^0\overline L) \right)
  \prod_{M\in\ind K_\Lambda}
  u_{M} ^{\hom (H^0 M, H^0 \overline L) - \langle g(M),\bd\rangle)}\\
  &=&\widetilde F_{\overline L_\Lambda},\end{eqnarray*}
as desired.
\end{proof}

We note that $[\pi N:\overline M]$ can be calculated solely on the basis of $g_N$ if $N$ is rigid. 

\begin{proposition} Let $N\in\ind K_\Lambda$ be rigid. The indecomposable summands of $\pi M$ can be determined by looking at the minimal cone of the $g$-vector fan for $A$ that $g_N$ appears in. The expression for $g_N$ as a positive combination of the generators of that cone gives the expression of $\pi N$ as a sum of indecomposable objects from $\ind K_A$.\end{proposition}

\begin{proof} The key point to establish is that $\pi N$ is again rigid. We know the projective presentation of $\pi N$ (it is that of $N$, unless the ideal $J$ includes idempotents), so once $\pi N$ is rigid, it must be a sum of compatible $g$-vectors from $\overline {\mathcal  I}$, and there is only one way to combine them to get $g_N$.

  In the case that $J$ contains an idempotent, we don't have an isomorphism of the two spaces of $g$-vectors, but rather a projection; however, everything works the same way.

  We now establish the key point, that $\pi N$ is again rigid. Suppose that there is a non-zero map from $N\otimes_\Lambda A$ to $\Sigma(N\otimes_\Lambda A)$. This defines a map from $N$ to $\Sigma(N\otimes_\Lambda A)$, and this map lifts to a map from $N$ to $\Sigma N$. Suppose it is null-homotopic. The homotopy descends to a homotopy from $N$ to $\Sigma(N\otimes_\Lambda A)$, and this factors through a homotopy from $N\otimes_\Lambda A$ to $\Sigma (N\otimes_\Lambda A) $. Thus, there are no non null-homotopic maps from $\pi N$ to $\Sigma(\pi N)$.
\end{proof}

Recall from Section \ref{sec:geoM} the rational functions $v_N$ for $N\in K_\Lambda$, which, by Corollary \ref{vs-solve}, solve the $\hF$ equations for $\Lambda$. Similarly, for $\overline M$ in $K_A$, let us write $\overline v_{\overline M}$ for the corresponding rational functions which solve the $\hF$ equations for $A$. 
Theorem \ref{quotient} implies that we can use the rational functions $v_{N}$ to build solutions for the $\hF$ equations for $A$; specifically, if we substitute
$$\overline u_{\overline M}\leftarrow \prod_{N\in K_\Lambda} v_N^{[\pi N:\overline M]}$$
then this solves the $\hF$ equations. The following proposition says that the solutions so obtained are exactly the same as the solutions we get by applying Corollary \ref{vs-solve} to $A$ directly.

\begin{proposition}\label{same} Let $A$ be a quotient of $\Lambda$. Then, for all $\overline M\in \ind K_A$,
  $$ \overline v_{\overline M} = \prod_{N\in\ind K_\Lambda} v_N^{[\pi N:\overline M]}$$
  \end{proposition}

\begin{proof}
  Suppose first that $\overline M\in\ind A$, and not projective. We then have an almost split sequence  $0 \rightarrow \tau_A \overline M \rightarrow \overline E \rightarrow \overline M \rightarrow 0$ in $\mod A$.  For $\overline X \in \mod A$, we have
  $$[\overline X:\overline M]=\hom_A(\overline X,\tau_A \overline M)-\hom_A(\overline X,\overline E)+\hom_A(\overline X,\overline M).$$
  Since for $X\in\mod\Lambda$ and $\overline N\in \mod A$, we have $\hom_\Lambda(X,\overline N)=\hom_A(\pi X,\overline N)$, we have
  $$[\pi X:\overline M]=\hom_\Lambda(X,\tau_A \overline M)-\hom_\Lambda(X,\overline E)+\hom_\Lambda(X,\overline M).$$
  Thus,
$$\prod_{N\in\ind K_\Lambda} v_N^{[\pi N:\overline M]}=\prod_{N\in \ind K_\Lambda} v_N^{\hom_\Lambda(N,\tau_A \overline M)-\hom_\Lambda(N,\overline E)+\hom_\Lambda(N,\overline M)}$$

Applying Lemma \ref{lem:Fexpand} three times, we get
$$\prod_{N\in\ind K_\Lambda} v_N^{[\pi N:\overline M]}=\frac {F_{\overline E}}{F_{\tau_A \overline M}F_{\overline M}},$$
and this proves the proposition in this case.

If $\overline M$ is an indecomposable projective, we proceed in the same way, but using the sequence $0 \rightarrow \rad M \rightarrow M$.

If $\overline M$ is a shifted projective, we proceed differently. At this point, we have already deduced that the solutions to the $\hF$ equations given by $\overline v_M$ and the solutions given by $\prod_N v_N^{[\pi N:\overline M]}$ agree except possibly as to the values which they assign to $\overline u_{\Sigma \overline P_i}$. But the $\hF$ equations imply the $u$-equations, and the $u$-equation for $\Sigma \overline P_i$ says that $\overline u_{\Sigma \overline P_i}$ is 1 minus a product of $\overline u_{\overline X}$ with $\overline X$ incompatible with $\Sigma \overline P_i$. Since the other $\Sigma \overline P_j$ are compatible with $\Sigma \overline P_i$, the value of the other term in the $u$-equation for $\Sigma \overline P_i$ is the same under each of the two substitutions we are considering, so the value assigned to $\overline u_{\Sigma \overline P_i}$ must also be the same. Thus the proposition in proven in this case as well.
\end{proof}

The map from $\widetilde \cM_\Lambda$ to $\widetilde \cM_A$ admits a completely different description in terms of the toric geometry of Section \ref{geometry}.

The maximal cones of $\Sigma_\Lambda$ are in bijection with the torsion classes of $\Lambda$: if $\theta$ lies in the interior of a maximal cone $\sigma$, then the modules $M$ in the corresponding torsion class can be characterized by the fact that $\langle \theta,\dimv N\rangle$ is non-negative for every quotient $N$ of $M$ \cite[Corollary 4.26]{BST}. It immediately follows that if $A$ is a quotient of $\Lambda$, then if $\theta$ and $\theta'$ determine the same torsion class for $\Lambda$, they will also do so for $A$. It follows that the image under the natural map from $K_0(\proj \Lambda)$ to $K_0(\proj A)$ of a maximal cone (and hence of any cone) of $\Sigma_\Lambda$ is contained in a cone of $\Sigma_A$. (Note that this applies when $\Lambda$ and $A$ have the same number of simple modules, so the natural map is an identification, but also when $A$ has fewer.) This condition relating the fans $\Sigma_\Lambda$ and $\Sigma_A$ is exactly the condition that the natural map from $K_0(\proj \Lambda)$ to $K_0(\proj A)$ induces a map between the corresponding toric varieties \cite[Theorem 3.3.4]{CLS}.

\begin{theorem}\label{thm:blowdown} Let $A=\Lambda/J$. The map of toric varieties from $X_\Lambda$ to $X_A$ induced by the natural quotient from $K_0(\proj \Lambda)$ to $K_0(\proj A)$ agrees with the map from $\widetilde \cM_\Lambda$ to $\widetilde \cM_A$ from Theorem \ref{quotient}.
  \end{theorem}

\begin{proof} Suppose that $A=\Lambda/J$. Let $J'=J\cap \operatorname {rad} A$, where $\operatorname {rad} A$ is the ideal of $A$ generated by its arrows. We prove the theorem in two steps, first passing from $\Lambda$ to $\Lambda/J'$, and then from $\Lambda/J'$ to $\Lambda/J$. At the first step, $J'$ has no idempotents. At the second step, we are quotienting by idempotents corresponding to isolated vertices in $\Lambda/J'$. 

It therefore suffices to prove the theorem in two special cases: when $J$ has no idempotents, and when $J$ consists of idempotents corresponding to isolated vertices. 
In the first case,
the blowdown map from $X_\Lambda$ to $X_A$ identifies the function fields $k(y_1,\dots,y_n)$ of the two toric varieties. Since, by Proposition \ref{same},  the map from Theorem \ref{quotient}
does this too, the two maps coincide.
In the second case, both maps project away the directions corresponding to the idempotents, so they coincide.
This completes the proof of the theorem.
  \end{proof}

\section{Dilogarithm identities}\label{sec:dilog}

Recall that we write $\dil(x)$ for the Rogers dilogarithm of $x$. 
In this section, we show how to adapt the argument from \cite{Chapoton} to prove the following theorem:

\begin{theorem}\label{dil-ident} Let $\Lambda$ be a finite-dimensional algebra of finite representation type, with $n$ simple modules. Then $\sum_{N} \dil(1-v_N)=n\pi^2/6$, where the sum is taken over all indecomposable $N$ in $K_\Lambda$.
\end{theorem}

The main ingredient in the proof goes back to Frenkel--Szenes \cite{FrenkelSzenes}, see also \cite{Nakanishi}. Let $G$ be the continuous functions from $\mathbb R^n_{>0}$ to $\mathbb R_{>0}$, which we think of as an abelian group written multiplicatively.  
Consider $G\otimes_{\mathbb Z} G$. Write $S^2G$ for the subgroup of $G\otimes_{\mathbb Z} G$ generated by elements of the form $a\otimes b + b\otimes a$. 

The following is essentially \cite[Proposition 3.1]{FrenkelSzenes}, but slightly reformulated for our convenience. In particular, their functions $f_i$ only take one argument, but the same approach proves the version where the functions take multiple arguments. 

\begin{proposition}[{\cite[Proposition 3.1]{FrenkelSzenes}}] Suppose that
  $f_1,\dots,f_r$ are differentiable functions in $G$, whose range lies in $(0,1)$. Suppose further that $\sum_{i=1}^r f_i\otimes (1-f_i) \in S^2G$. 
  Then
  $\sum_{i=1}^{r} \dil(f_i(y_1,\dots,y_n))$ does not depend on the values of $y_1,\dots,y_r$.\end{proposition}

In order to apply the proposition, let us consider $\sum_{N} v_N \otimes (1-v_N)$. We have
$$ \sum_{N} v_N \otimes (1-v_N) = \sum_{N,M} v_N \otimes v_M^{c(N,M)} = \sum_{N,M} c(N,M) v_N \otimes v_M.$$
Since $c(N,M)=c(M,N)$, this expression falls in $S^2G$.

\section{Small Examples}\label{section-examples}
\subsection{$A_2$ path algebra.}\label{ex:A2} Let~$\Lambda = \bC Q$ be the path algebra of the quiver~$1\xrightarrow{a} 2$. Then the category~$K_\Lambda$ has~$5$ indecomposable objects which we designate by the following shorthand.

 \begin{center}
\begin{tabular}{ccc}
 $P_1 = (0\xrightarrow{}P_1)$ & $P_2 = (0\xrightarrow{}P_2)$ &
 $S_2 = (P_1\xrightarrow{a} P_2)$ \\  $\Sigma P_1 = (P_1 \xrightarrow{} 0)$ & $\Sigma P_2 = (P_2 \xrightarrow{} 0)$ &
\end{tabular}
\end{center}
Then the~$u$-equations are as follows.
\begin{center}
 \begin{tabular}{lclcl}
  $u_{P_1} + u_{S_2}u_{\Sigma P_1} = 1$ & & $u_{P_2} + u_{\Sigma P_1}u_{\Sigma P_2}=1$ && $u_{S_2} + u_{P_1}u_{\Sigma P_2} = 1$\\
  $u_{\Sigma P_1} + u_{P_1}u_{P_2} = 1$ && $u_{\Sigma P_2} + u_{P_2}u_{S_2} = 1$
 \end{tabular}
\end{center}
The~$\widehat{F}$-polynomials are as follows.
\begin{center}
 \begin{tabular}{lclcl}
  $\widehat{F}_{P_1} = u_{P_1} + u_{S_2}u_{\Sigma P_1}$ && $\widehat{F}_{P_2} = u_{P_1}u_{P_2} + u_{P_2}u_{S_2}u_{\Sigma P_1} + u_{\Sigma P_1}u_{\Sigma P_2}$  && $\widehat{F}_{S_2} = u_{P_2}u_{S_2} + u_{\Sigma P_2}$
 \end{tabular}
\end{center}
Applying Corollary~\ref{vs-solve}, we get that the assignment~$u_X\mapsto v_X$ solves the~$u$-equations and the equations~$\widetilde{F}_X = 1$, where
\begin{center}
 \begin{tabular}{lclcl}
  $v_{P_1} = \frac{1}{1+y_1}$ & & $v_{P_2} = \frac{1+y_1}{1+y_1+y_1y_2}$ && $v_{S_2} = \frac{1+y_1+y_1y_2}{(1+y_1)(1+y_2)}$\\
  $v_{\Sigma P_1} = \frac{y_1 + y_1y_2}{1+y_1+y_1y_2}$ && $v_{\Sigma P_2} = \frac{y_2}{1+y_2}$
 \end{tabular}
\end{center}

\subsection{$A_2$ preprojective algebra $\Pi_{A_2}$.}\label{ex:PiA2} The algebra~$\Pi_{A_2}$ is the path algebra of the quiver
\[\begin{tikzcd}
	1 && 2
	\arrow["a", curve={height=-12pt}, from=1-1, to=1-3]
	\arrow["b", curve={height=-12pt}, from=1-3, to=1-1]
\end{tikzcd}\]
modulo the relations~$ab = 0$ and~$ba = 0$.  The category~$K_{\Pi_{A_2}}$ has~$6$ indecomposable objects.  Its Auslander--Reiten quiver is periodic and is depicted below, with the dotted arrows representation the action of~$\tau$.

{\small
\[\begin{tikzcd}
	\cdots && {0\xrightarrow{}P_1} && {P_2\xrightarrow{}0} && \cdots \\
	& {P_1\xrightarrow{a}P_2} && {P_2\xrightarrow{b}P_1} && {P_1\xrightarrow{a}P_2} \\
	\cdots && {P_1\xrightarrow{}0} && {0\xrightarrow{}P_2} && \cdots
	\arrow[from=1-1, to=2-2]
	\arrow[from=1-3, to=2-4]
	\arrow[dotted, from=1-5, to=1-3]
	\arrow[from=1-5, to=2-6]
	\arrow[from=2-2, to=1-3]
	\arrow[from=2-2, to=3-3]
	\arrow[from=2-4, to=1-5]
	\arrow[dotted, from=2-4, to=2-2]
	\arrow[from=2-4, to=3-5]
	\arrow[from=2-6, to=1-7]
	\arrow[dotted, from=2-6, to=2-4]
	\arrow[from=2-6, to=3-7]
	\arrow[from=3-1, to=2-2]
	\arrow[from=3-3, to=2-4]
	\arrow[dotted, from=3-3, to=3-1]
	\arrow[from=3-5, to=2-6]
	\arrow[dotted, from=3-7, to=3-5]
\end{tikzcd}\]
}
We compute the~$u$-equations and~$\widehat{F}$-polynomials.  We use the following shorthand for the objects of~$K_{\Pi_{A_2}}$:
\begin{center}
\begin{tabular}{ccc}
 $P_1 = (0\xrightarrow{}P_1)$ & $P_2 = (0\xrightarrow{}P_2)$ & $S_1 = (P_2\xrightarrow{b} P_1)$ \\
 $S_2 = (P_1\xrightarrow{a} P_2)$ &  $\Sigma P_1 = (P_1 \xrightarrow{} 0)$ & $\Sigma P_2 = (P_2 \xrightarrow{} 0)$
\end{tabular}
\end{center}
Then the~$u$-equations are as follows.
\begin{center}
 \begin{tabular}{lcl}
  $u_{P_1} + u_{S_2}u_{\Sigma P_1}u_{\Sigma P_2} = 1$ & & $u_{P_2} + u_{S_1}u_{\Sigma P_1}u_{\Sigma P_2}=1$ \\
  $u_{S_1} + u_{P_2}u_{S_2}^2u_{\Sigma P_1} = 1$ && $u_{S_2} + u_{P_1}u_{S_1}^2u_{\Sigma P_2} = 1$ \\
  $u_{\Sigma P_1} + u_{P_1}u_{P_2}u_{S_1} = 1$ && $u_{\Sigma P_2} + u_{P_1}u_{P_2}u_{S_2} = 1$
 \end{tabular}
\end{center}
The~$\widehat{F}$-polynomials are as follows.
\begin{center}
 \begin{tabular}{lcl}
  $\widehat{F}_{P_1} = u_{P_1}u_{P_2}u_{S_2} + u_{P_1}u_{S_1}u_{\Sigma P_2} + u_{S_2}u_{\Sigma P_1}u_{\Sigma P_2}$ && $\widehat{F}_{P_2} = u_{P_1}u_{P_2}u_{S_1} + u_{P_2}u_{S_2}u_{\Sigma P_1} + u_{S_1}u_{\Sigma P_1}u_{\Sigma P_2}$ \\
  $\widehat{F}_{S_1} = u_{P_1}u_{S_1} + u_{S_2}u_{\Sigma P_1}$ && $\widehat{F}_{S_2} = u_{P_2}u_{S_2} + u_{S_1}u_{\Sigma P_2}$
 \end{tabular}
\end{center}
We can check using computer algebra software that the ideal generated by the~$\widehat{F} = 1$ equations is equal to the ideal generated by the~$u$-equations.  In other words, the varieties~$\widetilde{\cM}_{\Pi_{A_2}}$ and~$\widetilde{\cU}_{\Pi_{A_2}}$ are equal.  According to Corollary~\ref{vs-solve}, the~$u$-equations are solved by the assigment~$u_X\mapsto v_X$, where
\begin{center}
 \begin{tabular}{lcl}
  $v_{P_1} = \frac{1+y_2}{1+y_2+y_1y_2}$ &&
  $v_{P_2} = \frac{1+y_1}{1+y_1+y_1y_2}$ \\
  $v_{S_1} = \frac{1+y_2+y_1y_2}{(1+y_1)(1+y_2)}$ &&
  $v_{S_2} = \frac{1+y_1+y_1y_2}{(1+y_1)(1+y_2)}$ \\
  $v_{\Sigma P_1} = \frac{y_1 + y_1y_2}{1+y_1+y_1y_2}$ &&
  $v_{\Sigma P_2} = \frac{y_2+y_1y_2}{1+y_2+y_1y_2}$
 \end{tabular}
\end{center}

\subsection{A quotient of the $A_n$ preprojective algebra}

Generalizing the previous example, consider the quiver $Q$ with vertices numbered 1 to $n$, with one arrow from $i$ to $i+1$ and one arrow from $i+1$ to $i$ for each
$1\leq i \leq n-1$. Let $I$ be the ideal in $\mathbb CQ$ generated by the composition of the two arrows between $i$ and $i+1$ in either order, for each $1\leq i\leq n-1$.

The algebra $\mathbb CQ/I$ is a gentle algebra, so is easily seen to be of finite representation type. It is a quotient of the preprojective algebra of type $A_n$. The $g$-vector fan of this preprojective algebra is isomorphic to the $A_n$ Coxeter fan \cite{Miz}.  In general, the $g$-vector fan of a quotient is a coarsening of the $g$-vector fan of the original algebra, but in this case, they coincide \cite{DIRRT}. It follows that the $g$-vector fan of $\mathbb CQ/I$ is isomorphic to the $A_n$ Coxeter fan, which can also be described as the outer normal fan to the permutohedron.  Consequently, from this algebra, we obtain a variety $\widetilde \cM_{\mathbb CQ/I}$ whose totally non-negative part has the boundary structure of the permutohedron.

The $A_n$ preprojective algebra itself is defined by the relations that, for all $i$, the sum of the length two paths from $i$ to $i$ is zero. For $n> 4$, this algebra is not of finite representation type, so the techniques in this paper do not apply directly to it.

\subsection{$A_3$ path algebra}\label{ex:A3}
Let~$\Lambda = \bC Q$ be the path algebra of the quiver~$1\xrightarrow{a} 2\xrightarrow{b} 3$. Then the category~$K_\Lambda$ has~$9$ indecomposable objects and its Auslander--Reiten quiver looks like this (with the dotted arrows representing the action of~$\tau$).

{\small
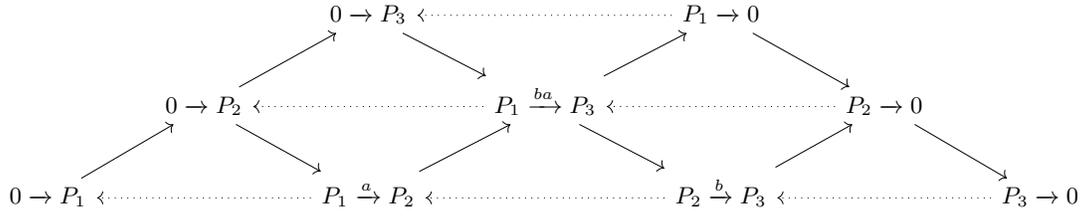
\[\begin{tikzcd}
	&& {0\xrightarrow{}P_3} && {P_1\xrightarrow{}0} \\
	& {0\xrightarrow{}P_2} && {P_1\xrightarrow{ba}P_3} && {P_2\xrightarrow{}0} \\
	{0\xrightarrow{}P_1} && {P_1\xrightarrow{a}P_2} && {P_2\xrightarrow{b}P_3} && {P_3\xrightarrow{}0}
	\arrow[from=1-3, to=2-4]
	\arrow[dotted, from=1-5, to=1-3]
	\arrow[from=1-5, to=2-6]
	\arrow[from=2-2, to=1-3]
	\arrow[from=2-2, to=3-3]
	\arrow[from=2-4, to=1-5]
	\arrow[dotted, from=2-4, to=2-2]
	\arrow[from=2-4, to=3-5]
	\arrow[dotted, from=2-6, to=2-4]
	\arrow[from=2-6, to=3-7]
	\arrow[from=3-1, to=2-2]
	\arrow[from=3-3, to=2-4]
	\arrow[dotted, from=3-3, to=3-1]
	\arrow[from=3-5, to=2-6]
	\arrow[dotted, from=3-5, to=3-3]
	\arrow[dotted, from=3-7, to=3-5]
\end{tikzcd}\]
}
We use the following shorthand for the objects of~$K_{\Pi_{A_2}}$:
\begin{center}
\begin{tabular}{ccc}
 $P_1 = (0\xrightarrow{}P_1)$ & $P_2 = (0\xrightarrow{}P_2)$ & $P_3 = (0\xrightarrow{}P_3)$ \\
 $S_2 = (P_1\xrightarrow{a} P_2)$ & $I_2 = (P_1\xrightarrow{ba} P_3)$ & $S_3 = (P_2\xrightarrow{b}P_3)$ \\
 $\Sigma P_1 = (P_1 \xrightarrow{} 0)$ & $\Sigma P_2 = (P_2 \xrightarrow{} 0)$ & $\Sigma P_3 = (P_3 \xrightarrow{} 0)$
\end{tabular}
\end{center}
Then the~$u$-equations are as follows.
\begin{center}
 \begin{tabular}{lclcl}
  $u_{P_1} + u_{S_2}u_{I_2}u_{\Sigma P_1} = 1$ &&
  $u_{P_2} + u_{I_2}u_{S_3}u_{\Sigma P_1}u_{\Sigma P_2}=1$ &&
  $u_{P_3} + u_{\Sigma P_1}u_{\Sigma P_2}u_{\Sigma P_3} = 1$ \\

  $u_{S_2} + u_{P_1}u_{S_3}u_{\Sigma P_2} = 1$ &&
  $u_{I_2} + u_{P_1}u_{P_2}u_{\Sigma P_2}u_{\Sigma P_3} = 1$ &&
  $u_{S_3} + u_{P_2}u_{S_2}u_{\Sigma P_3} = 1$ \\

  $u_{\Sigma P_1} + u_{P_1}u_{P_2}u_{P_3} = 1$ &&
  $u_{\Sigma P_2} + u_{P_2}u_{P_3}u_{S_2}u_{I_2} = 1$ &&
  $u_{\Sigma P_3} + u_{P_3}u_{I_2}u_{S_3} = 1$
 \end{tabular}
\end{center}
The~$\widehat{F}$-polynomials are as follows.
\begin{center}
 \begin{tabular}{l}
  $\widehat{F}_{P_1} = u_{P_1} + u_{S_2}u_{I_2}u_{\Sigma P_1}$ \\

  $\widehat{F}_{P_2} = u_{P_1}u_{P_2} + u_{P_2}u_{S_2}u_{I_2}u_{\Sigma P_1} + u_{I_2}u_{S_3}u_{\Sigma P_1}u_{\Sigma P_2}$ \\

  $\widehat{F}_{P_3} = u_{P_1}u_{P_2}u_{P_3} + u_{P_2}u_{P_3}u_{S_2}u_{I_2}u_{\Sigma P_1} + u_{P_3}u_{I_2}u_{S_3}u_{\Sigma P_1}u_{\Sigma P_2} + u_{\Sigma P_1}u_{\Sigma P_2}u_{\Sigma P_3}$ \\

  $\widehat{F}_{S_2} = u_{P_2}u_{S_2} + u_{S_3}u_{\Sigma P_2}$ \\
  $\widehat{F}_{I_2} = u_{P_2}u_{P_3}u_{S_2}u_{I_2} + u_{P_3}u_{I_2}u_{S_3}u_{\Sigma P_2} + u_{\Sigma P_2}u_{\Sigma P_3}$ \\

  $\widehat{F}_{S_3} = u_{P_3}u_{I_2}u_{S_3} + u_{\Sigma P_3}$
 \end{tabular}
\end{center}

\subsection{$A_3$ with a relation}
Let~$\Lambda$ be the path algebra of the quiver~$1\xrightarrow{a} 2\xrightarrow{b} 3$ modulo the relation~$ba=0$. Then the category~$K_\Lambda$ has~$8$ indecomposable objects and its Auslander--Reiten quiver looks like this (with the dotted arrows still representing the action of~$\tau$).
{\small
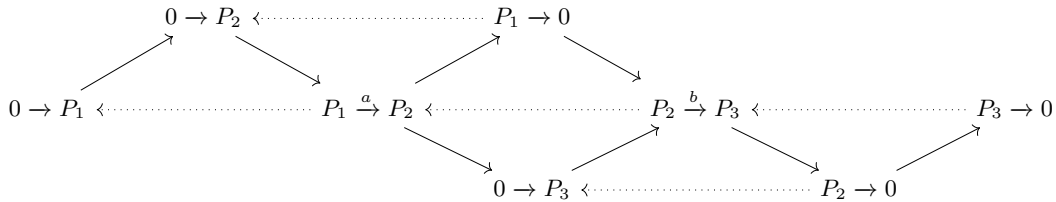
\[\begin{tikzcd}
	& {0\xrightarrow{}P_2} && {P_1\xrightarrow{}0} \\
	{0\xrightarrow{}P_1} && {P_1\xrightarrow{a}P_2} && {P_2\xrightarrow{b}P_3} && {P_3\xrightarrow{}0} \\
	&&& {0\xrightarrow{}P_3} && {P_2\xrightarrow{}0}
	\arrow[from=1-2, to=2-3]
	\arrow[dotted, from=1-4, to=1-2]
	\arrow[from=1-4, to=2-5]
	\arrow[from=2-1, to=1-2]
	\arrow[from=2-3, to=1-4]
	\arrow[dotted, from=2-3, to=2-1]
	\arrow[from=2-3, to=3-4]
	\arrow[dotted, from=2-5, to=2-3]
	\arrow[from=2-5, to=3-6]
	\arrow[dotted, from=2-7, to=2-5]
	\arrow[from=3-4, to=2-5]
	\arrow[from=3-6, to=2-7]
	\arrow[dotted, from=3-6, to=3-4]
\end{tikzcd}\]
}
We use the following shorthand for the objects of~$K_{\Lambda}$:
\begin{center}
\begin{tabular}{ccc}
 $P_1 = (0\xrightarrow{}P_1)$ & $P_2 = (0\xrightarrow{}P_2)$ & $P_3 = (0\xrightarrow{}P_3)$ \\
 $S_2 = (P_1\xrightarrow{a} P_2)$ & $S_3 = (P_2\xrightarrow{b} P_3)$ &  \\
 $\Sigma P_1 = (P_1 \xrightarrow{} 0)$ & $\Sigma P_2 = (P_2 \xrightarrow{} 0)$ & $\Sigma P_3 = (P_3 \xrightarrow{} 0)$
\end{tabular}
\end{center}
Then the~$u$-equations are as follows.
\begin{center}
 \begin{tabular}{lclcl}
  $u_{P_1} + u_{S_2}u_{\Sigma P_1} = 1$ &&
  $u_{P_2} + u_{S_3}u_{\Sigma P_1}u_{\Sigma P_2}=1$ &&
  $u_{P_3} + u_{\Sigma P_2}u_{\Sigma P_3} = 1$ \\

  $u_{S_2} + u_{P_1}u_{S_3}u_{\Sigma P_2} = 1$ &&
  $u_{S_3} + u_{P_2}u_{S_2}u_{\Sigma P_3} = 1$ && \\

  $u_{\Sigma P_1} + u_{P_1}u_{P_2} = 1$ &&
  $u_{\Sigma P_2} + u_{P_2}u_{P_3}u_{S_2} = 1$ &&
  $u_{\Sigma P_3} + u_{P_3}u_{S_3} = 1$
 \end{tabular}
\end{center}
The~$\widehat{F}$-polynomials are as follows.
\begin{center}
 \begin{tabular}{l}
  $\widehat{F}_{P_1} = u_{P_1} + u_{S_2}u_{\Sigma P_1}$ \\

  $\widehat{F}_{P_2} = u_{P_1}u_{P_2} + u_{P_2}u_{S_2}u_{\Sigma P_1} + u_{S_3}u_{\Sigma P_1}u_{\Sigma P_2}$ \\

  $\widehat{F}_{P_3} = u_{P_2}u_{P_3}u_{S_2} + u_{P_3}u_{S_3}u_{\Sigma P_2} + u_{\Sigma P_2}u_{\Sigma P_3}$ \\

  $\widehat{F}_{S_2} = u_{P_2}u_{S_2} + u_{S_3}u_{\Sigma P_2}$ \\

  $\widehat{F}_{S_3} = u_{P_3}u_{S_3} + u_{\Sigma P_3}$
 \end{tabular}
\end{center}

We leave the following as an instructive exercise for the reader to illustrate Theorem~\ref{quotient}. Calculate the $v_X$ for $A_3$ with no relation, and then check that the above $u$-equations and $\hF$-equations for $\Lambda$, the path algebra with a relation, are satisfied by $\overline v_Y$ for $Y\in K_\Lambda$, where $\overline v_{P_3}= v_{P_3}v_{I_2}$, $\overline v_{\Sigma P_1}=v_{I_2}v_{\Sigma P_1}$, and the other $\overline v_Y$ agree with the rational function for the corresponding object of $K_{A_3}$.

\subsection{$A_2$ with a loop at one vertex.} \label{A2-loop}
Consider the path algebra of the quiver
\[\begin{tikzcd}
	1 & 2
	\arrow["x", from=1-1, to=1-1, loop, in=150, out=210, distance=5mm]
	\arrow["a", from=1-1, to=1-2]
\end{tikzcd}\]
modulo the relation~$x^2=0$. The Auslander--Reiten quiver of~$K_\Lambda$ is periodic; a part of it containing a fundamental domain is depicted below.

{\tiny
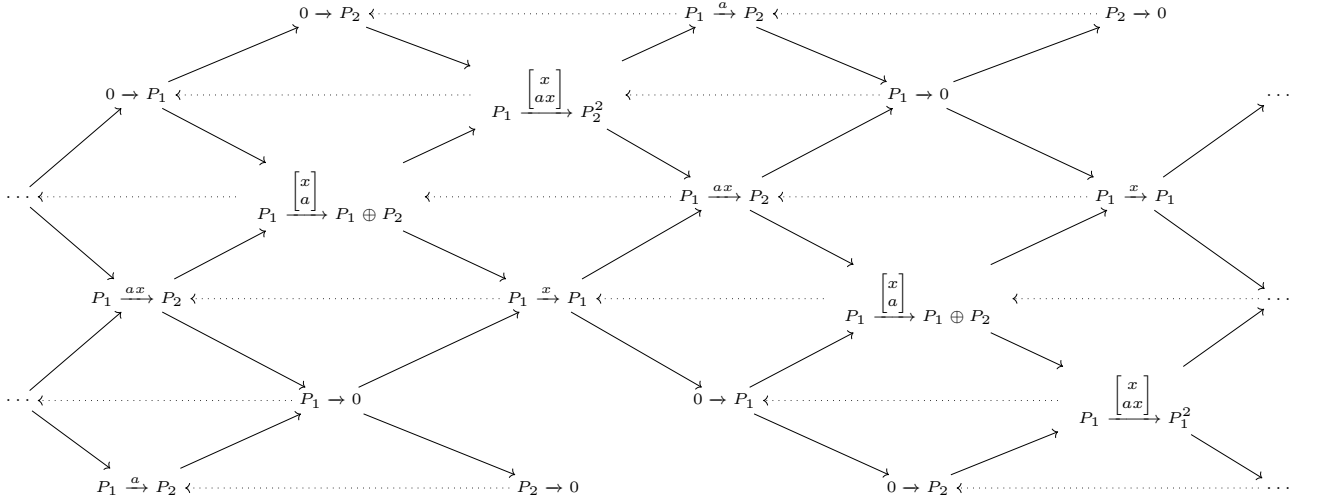
\begin{center}
\[\begin{tikzcd}
	&& {0\xrightarrow{}P_2} && {P_1\xrightarrow{a}P_2} && {P_2\xrightarrow{}0} \\
	& {0\xrightarrow{}P_1} && \begin{array}{c} P_1\xrightarrow{\begin{bmatrix} x \\ ax \end{bmatrix}}P_{2}^2 \end{array} && {P_1\xrightarrow{}0} && \cdots \\
	\cdots && \begin{array}{c} P_1\xrightarrow{\begin{bmatrix} x \\ a \end{bmatrix}}P_1\oplus P_2 \end{array} && {P_1\xrightarrow{ax}P_2} && {P_1\xrightarrow{x}P_1} \\
	& {P_1\xrightarrow{ax}P_2} && {P_1\xrightarrow{x}P_1} && \begin{array}{c} P_1\xrightarrow{\begin{bmatrix} x \\ a \end{bmatrix}}P_1\oplus P_2 \end{array} && \cdots \\
	\cdots && {P_1\xrightarrow{}0} && {0\xrightarrow{}P_1} && \begin{array}{c} P_1\xrightarrow{\begin{bmatrix} x \\ ax \end{bmatrix}}P_1^2 \end{array} \\
	& {P_1\xrightarrow{a}P_2} && {P_2\xrightarrow{}0} && {0\xrightarrow{}P_2} && \cdots
	\arrow[from=1-3, to=2-4]
	\arrow[dotted, from=1-5, to=1-3]
	\arrow[from=1-5, to=2-6]
	\arrow[dotted, from=1-7, to=1-5]
	\arrow[from=2-2, to=1-3]
	\arrow[from=2-2, to=3-3]
	\arrow[from=2-4, to=1-5]
	\arrow[dotted, from=2-4, to=2-2]
	\arrow[from=2-4, to=3-5]
	\arrow[from=2-6, to=1-7]
	\arrow[dotted, from=2-6, to=2-4]
	\arrow[from=2-6, to=3-7]
	\arrow[from=3-1, to=2-2]
	\arrow[from=3-1, to=4-2]
	\arrow[from=3-3, to=2-4]
	\arrow[dotted, from=3-3, to=3-1]
	\arrow[from=3-3, to=4-4]
	\arrow[from=3-5, to=2-6]
	\arrow[dotted, from=3-5, to=3-3]
	\arrow[from=3-5, to=4-6]
	\arrow[from=3-7, to=2-8]
	\arrow[dotted, from=3-7, to=3-5]
	\arrow[from=3-7, to=4-8]
	\arrow[from=4-2, to=3-3]
	\arrow[from=4-2, to=5-3]
	\arrow[from=4-4, to=3-5]
	\arrow[dotted, from=4-4, to=4-2]
	\arrow[from=4-4, to=5-5]
	\arrow[from=4-6, to=3-7]
	\arrow[dotted, from=4-6, to=4-4]
	\arrow[from=4-6, to=5-7]
	\arrow[dotted, from=4-8, to=4-6]
	\arrow[from=5-1, to=4-2]
	\arrow[from=5-1, to=6-2]
	\arrow[from=5-3, to=4-4]
	\arrow[dotted, from=5-3, to=5-1]
	\arrow[from=5-3, to=6-4]
	\arrow[from=5-5, to=4-6]
	\arrow[from=5-5, to=6-6]
	\arrow[from=5-7, to=4-8]
	\arrow[dotted, from=5-7, to=5-5]
	\arrow[from=5-7, to=6-8]
	\arrow[from=6-2, to=5-3]
	\arrow[dotted, from=6-4, to=6-2]
	\arrow[from=6-6, to=5-7]
	\arrow[dotted, from=6-8, to=6-6]
\end{tikzcd}\]
\end{center}
}

We use the following shorthand for the objects of~$K_{\Lambda}$:
\begin{center}
\begin{tabular}{ccc}
 $P_1 = (0\xrightarrow{}P_1)$ &
 $P_2 = (0\xrightarrow{}P_2)$ &
 $S_1 = (P_1\xrightarrow{x}P_1)$ \\
 $S_2 = (P_1\xrightarrow{a} P_2)$ &
 $I_1 = (P_1\xrightarrow{} P_2^2)$ &
 $12 = (P_1\xrightarrow{ax}P_2$ \\
 $112 = (P_1\xrightarrow{}P_1\oplus P_2)$ &
 $\Sigma P_1 = (P_1 \xrightarrow{} 0)$ &
 $\Sigma P_2 = (P_2 \xrightarrow{} 0)$
\end{tabular}
\end{center}
Then the~$u$-equations are as follows.
\begin{center}
 \begin{tabular}{lcl}
  $u_{P_1} + u_{S_1}u_{S_2}^2u_{I_1}^2 u_{12}^2 u_{112}u_{\Sigma P_1}^2 = 1$ &&
  $u_{P_2} + u_{S_1}u_{S_2}u_{12}u_{\Sigma P_1}^2u_{\Sigma P_2}=1$ \\
  $u_{S_1} + u_{P_1}u_{P_2}u_{S_1}^2u_{S_2}u_{I_1}^2u_{12}^2u_{112}^2u_{\Sigma P_1} = 1$ &&

  $u_{S_2} + u_{P_1}^2u_{P_2}u_{S_1}u_{112}u_{\Sigma P_2} = 1$ \\
  $u_{I_1} + u_{P_1}^2 u_{S_1}^2 u_{12} u_{112} u_{\Sigma P_1}^2 u_{\Sigma P_2}^2 = 1$ &&
  $u_{12} + u_{P_1}^2 u_{P_2} u_{S_1}^2 u_{I_1} u_{12}^2 u_{112}^2 u_{\Sigma P_1} u_{\Sigma P_2} = 1$ \\

  $u_{112} + u_{P_1}u_{S_1}^2 u_{S_2} u_{I_1} u_{12}^2 u_{112}^2 u_{\Sigma P_1}^2 u_{\Sigma P_2} = 1$ &&
  $u_{\Sigma P_1} + u_{P_1}^2 u_{P_2}^2 u_{S_1} u_{I_1}^2 u_{12} u_{112}^2 = 1$ \\
  $u_{\Sigma P_2} + u_{P_2} u_{S_2} u_{I_1}^2 u_{12} u_{112} = 1$
 \end{tabular}
\end{center}
The~$\widehat{F}$-polynomials are as follows.
{\footnotesize
\begin{center}
 \begin{tabular}{l}
  $\widehat{F}_{P_1} = u_{P_1}^{2}u_{S_1}u_{112} + u_{P_1}u_{S_1}u_{S_2}u_{I_1}u_{12}u_{112}u_{\Sigma P_1} + u_{S_1}u_{S_2}^{2}u_{I_1}^{2}u_{12}^{2}u_{112}u_{\Sigma P_1}^{2} $ \\

  $\widehat{F}_{P_2} = u_{P_1}^{2}u_{P_2}u_{S_1}u_{112} + u_{S_1}u_{S_2}u_{12}u_{\Sigma P_1}^{2}u_{\Sigma P_2} + u_{P_1}u_{P_2}u_{S_1}u_{S_2}u_{I_1}u_{12}u_{112}u_{\Sigma P_1} + u_{P_2}u_{S_1}u_{S_2}^{2}u_{I_1}^{2}u_{12}^{2}u_{112}u_{\Sigma P_1}^{2}$ \\

  $\widehat{F}_{S_1} = u_{P_1}u_{S_1}u_{112} + u_{S_1}u_{S_2}u_{I_1}u_{12}u_{112}u_{\Sigma P_1}$ \\

  $\widehat{F}_{S_2} = u_{\Sigma P_2} + u_{P_2}u_{S_2}u_{I_1}^{2}u_{12}u_{112}$ \\

  $\widehat{F}_{I_1} = u_{S_1}u_{12}u_{\Sigma P_1}^{2}u_{\Sigma P_2}^{2} + u_{P_1}u_{P_2}u_{S_1}u_{I_1}u_{12}u_{112}u_{\Sigma P_1}u_{\Sigma P_2} + u_{P_1}^{2}u_{P_2}^{2}u_{S_1}u_{I_1}^{2}u_{12}u_{112}^{2} +$ \\

  $\qquad + 2u_{P_2}u_{S_1}u_{S_2}u_{I_1}^{2}u_{12}^{2}u_{112}u_{\Sigma P_1}^{2}u_{\Sigma P_2} + u_{P_1}u_{P_2}^{2}u_{S_1}u_{S_2}u_{I_1}^{3}u_{12}^{2}u_{112}^{2}u_{\Sigma P_1} + u_{P_2}^{2}u_{S_1}u_{S_2}^{2}u_{I_1}^{4}u_{12}^{3}u_{112}^{2}u_{\Sigma P_1}^{2}$ \\

  $\widehat{F}_{12} = u_{S_1}u_{12}u_{\Sigma P_1}u_{\Sigma P_2} + u_{P_1}u_{P_2}u_{S_1}u_{I_1}u_{12}u_{112} + u_{P_2}u_{S_1}u_{S_2}u_{I_1}^{2}u_{12}^{2}u_{112}u_{\Sigma P_1}$ \\

  $\widehat{F}_{112} = u_{P_1}u_{S_1}u_{12}u_{112}u_{\Sigma P_1}u_{\Sigma P_2} + u_{P_1}^{2}u_{P_2}u_{S_1}u_{I_1}u_{12}u_{112}^{2} + u_{S_1}u_{S_2}u_{I_1}u_{12}^{2}u_{112}u_{\Sigma P_1}^{2}u_{\Sigma P_2} +$ \\
  $\qquad + u_{P_1}u_{P_2}u_{S_1}u_{S_2}u_{I_1}^{2}u_{12}^{2}u_{112}^{2}u_{\Sigma P_1} + u_{P_2}u_{S_1}u_{S_2}^{2}u_{I_1}^{3}u_{12}^{3}u_{112}^{2}u_{\Sigma P_1}^{2}$ \\
 \end{tabular}
\end{center}
}
In this example again, a computer check shows that~$\widetilde{\cM}_\Lambda = \widetilde{\cU}_\Lambda$.

\subsection{Pellytopes}

Let $A$ be the path algebra of the quiver $1\leftarrow 2 \leftarrow \dots \leftarrow n$, modulo the ideal generated by the relations that the composition of any two consecutive arrows is zero.

There are $3n-1$ elements of $\ind K_\Lambda$: $P_1,\dots,P_n=S_n; \Sigma P_1, \dots \Sigma P_n; S_1,\dots, S_{n-1}$. It will turn out to be more convenient to view $P_n$ as part of the family of simples, rather than the family of projectives, so we will consistently write $S_n$ instead of $P_n$ (though we will still write $\Sigma P_n$).
The $u$-equations are as follows:
\begin{align*}
  u_{P_i} + u_{S_{i-1}}u_{\Sigma P_i}u_{\Sigma P_{i+1}}&=1\\
  u_{\Sigma P_i} + u_{P_{i-1}}u_{P_i}u_{S_i}&=1\\
  u_{S_i}+u_{P_{i+1}}u_{S_{i+1}}u_{\Sigma P_{i}}u_{S_{i-1}}&=1
\end{align*}
where $i$ runs from 1 to $n-1$ in the first equation, and from 1 to $n$ in the second and third, and we drop the variables $u_{S_0}, u_{S_{n+1}}, u_{P_0}, u_{P_n},u_{P_{n+1}}$ whenever they arise. These are exactly the $u$-equations of \cite{pelly}.

We have numbered the vertices of our quiver so that the $g$-vectors of the objects of $\ind K_\Lambda$ agree with \cite{pelly}. The $g$-vector of $S_i$ is $\e_i-\e_{i+1}$ for $1\leq i\leq n-1$, while the $g$-vector of $S_n=P_n$ is $\e_n$. The $g$-vector of $P_i$ is of course $\e_i$, while the $g$-vector of $\Sigma P_i$ is $-\e_i$.

The $F$-polynomials are given by $F_{P_i}=1 + y_{i+1} + y_{i}y_{i+1}$ for $1\leq i\leq n-1$, and $F_{S_i}=1+y_i$ for $1\leq i \leq n$. 

The $\hF$-polynomials are
\begin{align*} \hF_{P_i}&=u_{P_i}u_{P_{i+1}}u_{S_{i+1}} + u_{\Sigma P_{i+1}}u_{S_i}u_{P_i}+u_{\Sigma P_{i+1}}u_{\Sigma P_i}u_{S_{i-1}}\\
  \hF_{S_i}&=u_{P_i}u_{S_i}+u_{\Sigma P_i}u_{S_{i-1}},\end{align*}
where $i$ runs from 1 to $n-1$ in the first equation and from 1 to $n$ in the second. Again, we drop $u_{S_0}$, $u_{P_n}$.

The solutions $v_X$ are given by:
\begin{center}
 \begin{tabular}{lcl}
  $v_{P_i}=\frac {F_{S_{i+1}}}{F_{P_i}}$ &&
  $v_{\Sigma P_i}=\frac{y_iF_{S_{i-1}}}{F_{P_{i-1}}} \textrm{ for $2\leq i \leq n$}$\\
  $v_{\Sigma P_1}=\frac{y_1}{1+y_1}$ &&
  $v_{S_i}=\frac{F_{P_{i}}}{F_{S_{i+1}}F_{S_i}}$
 \end{tabular}
\end{center}

We drop the terms $F_{P_n},F_{S_{n+1}}$ when they appear. (We could also use the formulas given above for them, and then decree that $y_{n+1}=0$.)

Now consider the $\hat F$-equations (or equivalently $u$-equations) for $\Lambda$, the path algebra of the same $A_n$ quiver, with no relations. We set $u_{ij}$ to correspond to $M_{ij}$, the indecomposable representation supported at vertices $k$ with $i<k<j$. This accounts for all $u_{ij}$ with $0\leq i$, $i\leq j-2$, $j\leq n+1$. 
We let $u_{-1,i}$ correspond to $\Sigma P_i$. The index set $\ind K_\Lambda$ is thereby identified with the diagonals of an $(n+3)$-gon, with vertices numbered $-1$ to $n+1$.

To find solutions to the Pellytope $u$-equations in terms of those from $\Lambda$ (which we denote $v_{ij}$), as in Section \ref{quotient-section}, we observe that for any representation of $A_n$ which is neither simple nor projective, its projective resolution, when tensored by the radical square zero quotient of the path algebra, splits into a direct sum of an indecomposable projective and a shifted indecomposable projective. Concretely, for $i\leq j-2$, we have the presentation
$$ \widehat P_j\rightarrow \widehat P_{i+1} \rightarrow M_{ij} \rightarrow 0.$$ 
(Here we write $\widehat P_i$ for the projective $\Lambda$-module with simple top $S_i$, keeping the notation $P_i$ for the corresponding projective $A$-module.)

If $i=j-2$, when we tensor the presentation of $M_{j-2,j}$ by $A$, it becomes a minimal projective presentation of $S_{j-1}$. 
However, if $i\leq j-3$, the result of tensoring the presentation of $M_{ij}$ by $A$ is a complex of the form
$ P_j\rightarrow P_{i+1}.$
Since $\Hom(P_j,P_{i+1})$ is 0, the map must be zero, so the complex splits as $(0\rightarrow P_{i+1}) \oplus (P_j\rightarrow 0)$. 
The result is that each $u_{i,j}$ with $i\ne -1$, $j\ne n+1$, and $j-i>2$, appears in the product for $P_{i+1}$ and that for $\Sigma P_{j}$.

The solution to the Pellytope $u$-equations in terms of the solutions for $\Lambda$ is therefore given by:

\[
v_{S_i}=v_{i-1,i+1}, \qquad
v_{P_i}=\prod_{j=i+2}^{n+1} v_{i-1,j}, \qquad
v_{\Sigma P_i}=\prod_{\ell=-1}^{i-3} v_{\ell,i}.
\]

\subsection*{Acknowledgements}
The authors would like to express their gratitude to the Mathematisches Forschungsinstitut Oberwolfach and to the Center of Mathematical Sciences and its Applications, where significant progress was made on this project. We thank Vladimir Fock for suggesting the link to Nahm's conjecture, which led to the results in Section \ref{sec:dilog}. We thank Nick Early for discussions on the CEGM~$u$-equations which motivated Section~\ref{grassmannian}. We thank Thomas Lam and Philippe Petit for comments on a previous draft of the paper.

The work of N.A.H. is supported by the DOE
(Grant No.~DE-SC0009988), the Simons Collaboration
on Celestial Holography and the Carl B. Feinberg cross-disciplinary program in innovation at the IAS.
Work of N.A.H, H.F, and G.S. also supported by the European Union (ERC, Universe+, 101118787).
H.F. is supported by the Sivian Fund. 
P.-G.P. was partly supported by the French ANR grant CHARMS (ANR-19-CE40-0017-02) and is supported by the Institut Universitaire de France (IUF).
H.T. is supported by NSERC Discovery Grant RGPIN-2022-03960 and the Canada Research Chairs program, grant number CRC-2021-00120.

\bibliographystyle{plain} 
\bibliography{uplusu}

\end{document}